\documentclass[a4paper, 10pt]{article}
\usepackage[dvipdfm]{graphicx}
\usepackage{amsfonts}
\usepackage{amsmath}
\usepackage{amssymb}
\usepackage{amsthm}
\usepackage{a4wide}
\usepackage{url}
\usepackage{cases}
\usepackage{hyperref}
\usepackage{algorithm}
\usepackage{algpseudocode}
\usepackage{multirow}
\usepackage{mathrsfs}
\usepackage{tabularx}

\newcolumntype{L}[1]{>{\raggedright\let\newline\\\arraybackslash\hspace{0pt}}m{#1}}
\newcolumntype{C}[1]{>{\centering\let\newline\\\arraybackslash\hspace{0pt}}m{#1}}
\newcolumntype{R}[1]{>{\raggedleft\let\newline\\\arraybackslash\hspace{0pt}}m{#1}}

\algnewcommand{\algorithmicgoto}{go to Step}%
\algnewcommand{\Goto}[1]{\algorithmicgoto~\ref{#1}}%

\newcommand{\dps}{\displaystyle}

\renewcommand{\geq}{\geqslant}
\renewcommand{\leq}{\leqslant}
\renewcommand{\ge}{\geqslant}
\renewcommand{\le}{\leqslant}

\newtheorem{remark}{Remark}
\newtheorem{example}{Example}
\newtheorem{proposition}{Proposition}
\newtheorem{theorem}{Theorem}
\newtheorem{lemma}{Lemma}
\newtheorem{assumption}{Assumption}
\newtheorem{definition}{Definition}

\allowdisplaybreaks

\begin{document}

\title{Multiple projection Markov Chain Monte Carlo algorithms on submanifolds}

\author{%
{\sc
Tony Leli\`evre\thanks{Corresponding author. Email:
  tony.lelievre@enpc.fr} and
Gabriel Stoltz\thanks{Email: gabriel.stoltz@enpc.fr}} \\[2pt]
CERMICS, Ecole des Ponts, Marne-la-Vallée, France \\[2pt]
\& MATHERIALS team-project, Inria Paris, France\\[6pt]
{\sc and}\\[6pt]
{\sc Wei Zhang}\thanks{Email: wei.zhang@fu-berlin.de}\\[2pt]
Zuse Institute Berlin,\\
 Takustrasse 7, 14195 Berlin, Germany
}

\maketitle

\begin{abstract}
{  We propose new Markov Chain Monte Carlo algorithms to sample probability
  distributions on submanifolds, which generalize previous methods by allowing
  the use of set-valued maps in the proposal step of the MCMC algorithms. The
  motivation for this generalization is that the numerical solvers used to
  project proposed moves to the submanifold of interest may find several
  solutions. We show that the new algorithms indeed sample the target
  probability measure correctly, thanks to some carefully enforced reversibility property. We demonstrate the interest of the new MCMC algorithms
  on illustrative numerical examples.
~ 
}
\end{abstract}

\section{Introduction}
\label{sec-intro}
Sampling probability distributions on submanifolds is relevant in various applications. In molecular dynamics and computational statistics for instance, one is often interested in sampling probability distributions on the (zero) level set of some lower-dimensional function~$\xi: \mathbb{R}^d\rightarrow \mathbb{R}^k$, where $1 \le k < d$:
  \begin{align}
    \Sigma = \left\{x \in \mathbb{R}^d~\middle|~ \xi(x)=0 \in \mathbb{R}^k\right\}\,.
    \label{sigma-level-set}
  \end{align}
The function~$\xi$ corresponds in molecular dynamics to molecular constraints
(fixed bond lengths between atoms, fixed bond angles, etc), and/or fixed
values of a reaction coordinate or collective variable as for free energy
calculations~\cite{lelievre2010free,non-equilibrium-2018} or model
reduction~\cite{effective_dynamics,effective_dyn_2017}. In computational
statistics, $\xi$ can be some summary statistics and sampling on $\Sigma$ is
relevant in approximate Bayesian computations~\cite{Tavare1997,Marin2012}. 

Markov chain Monte Carlo (MCMC) methods offer a generic way of sampling
probability distributions. They can be used even when the probability densities are known
up to a multiplicative constant. A prominent class of methods are
Metropolis-Hastings schemes~\cite{MRRTT53,Hastings70}. On~$\mathbb{R}^d$,
MCMC has been extensively studied due to its applications in a wide range of
areas~\cite{Liu2008,riemann-manifold-hmc-2011}. In particular, the
Metropolis-adjusted Langevin Algorithm
(MALA)~\cite{roberts1996,optimal-scaling-langevin1998}, previously
introduced in the molecular dynamics literature~\cite{RDF78}, and hybrid
Monte Carlo (HMC)~\cite{hmc-DUANE1987,neal1993probabilistic} are among the most popular MCMC methods that have been successfully used in applications.

MCMC methods that sample probability distributions on submanifolds have also been considered in the
literature~\cite{Hartmann2008,pmlr-v22-brubaker12,Tony-constrained-langevin2012,goodman-submanifold,hmc-submanifold-tony,LV20},
and have also been implemented in packages~\cite{mici}.
In particular, the authors in \cite{goodman-submanifold} proposed a MCMC
algorithm on submanifolds using a reversible Metropolis random walk, which was
then extended in~\cite{hmc-submanifold-tony} to a generalized HMC scheme on
submanifolds with non-zero gradient forces in the map used in the proposal
step. In contrast to MCMC methods on $\mathbb{R}^d$, MCMC methods on
submanifolds involve constraints in order to guarantee that the Markov chain
stays on the submanifold. As a result, the proposal maps are often nonlinear and are implicitly defined by constraint
equations. It has first been noticed in~\cite{goodman-submanifold} that performing a reversibility
check in the MCMC iterations is important in order to guarantee that the Markov
chain unbiasedly samples the target probability distribution on the
submanifold. In essence, the reversibility check amounts to verifying that the
numerical methods started from a new configuration can effectively go back to
the previous one, which may not be the case when the projection step finds a
different solution to the nonlinear equation under consideration. Let us
mention here that an alternative to solving nonlinear equation is to enforce
the constraint by following the flow of an appropriate ODE, as done
in~\cite{zhang2017} and~\cite{SZ20}, but this will not be discussed further here.

Based on the previous works~\cite{goodman-submanifold,hmc-submanifold-tony}, we study here MCMC methods to sample probability measures on the level set $\Sigma$ in~\eqref{sigma-level-set}, defined as 
\begin{align}
  \nu_{\Sigma}(dx) = \frac{1}{Z_{\nu_{\Sigma}}}\,\mathrm{e}^{-\beta V(x)}\,\sigma_\Sigma(dx)\,,
  \label{measure-mu-on-m}
\end{align}
where $\sigma_\Sigma$ is the surface measure on $\Sigma$ induced by the standard Euclidean scalar product in~$\mathbb{R}^d$, $V:\mathbb{R}^d\rightarrow\mathbb{R}$ is a $C^2$-differentiable potential energy function, $\beta> 0$ is a parameter
(proportional to the inverse temperature in the context of statistical
physics), and $Z_{\nu_{\Sigma}}$ is a normalization constant. The first method is a MCMC
algorithm in state space (i.e.\ the unknowns are~$x$ only), while the
second one is a MCMC algorithm in phase space, where some additional velocity
or momentum variable conjugated to the position variable~$x$ is
introduced. While the new algorithms share similarities with the ones
in~\cite{goodman-submanifold} and \cite{hmc-submanifold-tony}, the main novelty is that
we combine a local property of measure preservation (with a RATTLE scheme to integrate constrained Hamiltonian dynamics),
and a global construction of many solutions to propose new algorithms.
Allowing for several solutions of the constraint equation can be beneficial both because it can reduce the overall rejection rate, and probably more importantly because it may allow for larger, non-local moves. 
The algorithms we propose include a generalization of the ``reverse projection check'' of~\cite{goodman-submanifold} to the situations where multiple projections can be
computed. In particular, when the projection algorithm is able to find all the
possible projections on the manifold (which is for example possible for
algebraic submanifolds), this reverse projection check is not needed anymore and one only needs to count the number of possible projections,
see Remarks~\ref{rmk-Reversibility-check-mala} and~\ref{rmk-Reversibility-check-hmc}. 
We show that the first MCMC algorithm we propose generates a Markov
chain in state space which is \textit{reversible} with respect to the target probability distribution, while the
second one generates a Markov chain in phase space which is
\textit{reversible up to momentum reversal} with respect to the target probability
distribution (see Definitions~\ref{def-reversible}--\ref{def-rmr} in
Section~\ref{sec-two-algorithms}). In the following, we will simply say that a Markov chain is
reversible (possibly up to momentum reversal) when the target probability distribution is clear
from the context. The proofs of the consistency of the new algorithms also reveal the connections between the geometric point of view of~\cite{goodman-submanifold} and the symplectic point of view of~\cite{hmc-submanifold-tony}. 

\medskip

\paragraph{Outline of the work.} This paper is organized as follows. In Section~\ref{sec-two-algorithms}, we
introduce the new MCMC algorithms we propose and state the main results about
their consistency, whose proofs are postponed to Section~\ref{sec-proofs}. We
next demonstrate in Section~\ref{sec-example} the interest of the new MCMC
algorithms on two simple numerical examples. The code used for producing the numerical results in Section~\ref{sec-example} is available at \url{https://github.com/zwpku/Constrained-HMC}.

\paragraph{Notation and assumptions used throughout this work.}
We conclude this section by introducing some useful notation and stating the
assumptions we need on the function $\xi:\mathbb{R}^d\rightarrow\mathbb{R}^k$
with~$1\le k <d$. For any $x\in \mathbb{R}^d$,
$\nabla\xi(x)$ denotes the $d\times k$ matrix whose entries are
$(\nabla\xi(x))_{ij} = \frac{\partial \xi_j}{\partial x_i}(x)$ for~$1\le i \le
d$ and $1\le j\le k$ (i.e.\  the $j$th column of $\nabla \xi$ is $\nabla
\xi_j$). The matrix $I_n \in \mathbb{R}^{n\times n}$ denotes the identity
matrix of order $n$. The number of elements of a finite set~$A$ is denoted
by~$|A|$. The set difference of two sets $A$ and $B$, denoted by $A\setminus
B$, is defined as $A\setminus B=\{x\in A\,|\, x\not\in B\}$. The function~$\mathbf{1}_{A}(\cdot)$ is
the indicator function of the set~$A$. Let $\mathcal{M}$ be a Riemannian manifold of dimension~$m$, where $m \ge 1$.
The Riemannian measure on $\mathcal{M}$ is the canonical measure defined by
the Riemannian metric of $\mathcal{M}$~\cite[Chapter~4.10]{warner1983foundations}.
When $\mathcal{M}$ is a submanifold embedded in $\mathbb{R}^{m'}$, where $m'\ge m$, 
the standard Euclidean scalar product on $\mathbb{R}^{m'}$ induces a Riemannian metric $g$ on $\mathcal{M}$.  
The Riemannian measure on~$\mathcal{M}$ defined by $g$ will be called the
surface measure on $\mathcal{M}$ induced by the standard
Euclidean scalar product on $\mathbb{R}^{m'}$.
Similar terminologies will be adopted when we consider submanifolds of
$\mathbb{R}^{m'}$ endowed with the weighted inner product $\langle u, u'\rangle_{M}=u^TMu'$, for $u,u'\in
\mathbb{R}^{m'}$, where $M\in \mathbb{R}^{m'\times m'}$ is a symmetric positive definite matrix. 
Let $\mathcal{M}_1$ be another Riemannian manifold of dimension~$m_1$, where $m_1 \ge 1$.
When~$f: \mathcal{M}\rightarrow \mathcal{M}_1$ is a $C^1$-differentiable map from $\mathcal{M}$ to~$\mathcal{M}_1$, we denote by $Df(x): T_x
\mathcal{M} \rightarrow T_{f(x)}\mathcal{M}_1$ the differential of~$f$ at
the point~$x \in \mathcal{M}$. When $\mathcal{M}$ is an open subset of $\mathbb{R}^{m}$ 
and $\mathcal{M}_1$ is embedded in~$\mathbb{R}^{m'}$ with~$m'\ge m_1$, we write $\nabla f$ for the usual gradient of $f$, where $f$ is
viewed as a map $f: \mathcal{M} \rightarrow \mathbb{R}^{m'}$ between two
Euclidean spaces. The notations $D_x$ and $\nabla_x$ are also used to emphasize
that the differentiation is with respect to the variable~$x$. 
Given $x\in \mathcal{M}$, we say that $\mathcal{O}$ is a neighborhood of~$x$
if $\mathcal{O}$ is an open subset of~$\mathcal{M}$ such that $x\in
\mathcal{O}$. We will also use some results in differential
topology, for which we refer for instance to~\cite{banyaga2004lectures,lang2002differential}.
In the following let us assume that $m \ge m_1 \ge 1$. The map $f$ is a submersion at~$x\in \mathcal{M}$ if
$Df(x): T_x\mathcal{M}\rightarrow T_{f(x)}\mathcal{M}_1$ is onto. In this
case, $x$ is called a regular point and otherwise~$x$ is a critical point.
Denote by $C\subset \mathcal{M}$ the set of critical points of $f$. A point
$y\in \mathcal{M}_1$ is said to be a regular value if $f^{-1}(y)\cap
C=\emptyset$. Recall that the regular value theorem states that $f^{-1}(y)$ is
an~$(m-m_1)$-dimensional submanifold of~$\mathcal{M}$, provided that $y$
is a regular value and $f^{-1}(y)$ is non-empty. 
If $f:\mathcal{M}\rightarrow \mathcal{M}_1$ is $C^r$ with $r>
\max\{0,m-m_1\}$, Sard's theorem asserts that the image~$f(C)$ of~$C$ has
measure zero as a subset in~$\mathcal{M}_1$~\cite[Chapter 3.1]{hirsch2012differential}. 
In this paper we apply Sard's theorem to the case where $f$ is $C^1$-differentiable and $m=m_1$.

Finally, we also introduce a $C^2$-differentiable potential energy $\overline{V}:\mathbb{R}^d\rightarrow \mathbb{R}$ which can be different from
the function~$V$ in~\eqref{measure-mu-on-m}. Throughout this paper, we make the following assumption. 
\begin{assumption}
  Both functions $V$ and $\overline{V}$ are $C^2$-differentiable. The function
  $\xi:\mathbb{R}^{d}\rightarrow \mathbb{R}^k$ is smooth, and the level set $\Sigma$ is a compact subset of~$\mathbb{R}^d$. For all $x\in
  \Sigma$, the matrix $\nabla\xi(x)^T\nabla\xi(x) \in \mathbb{R}^{k\times k}$ is positive definite.
  \label{assump-xi}
\end{assumption}
\begin{remark}
  The assumption that $\nabla\xi(x)^T\nabla\xi(x)$ is positive definite is equivalent to the assumption that the matrix $\nabla\xi(x)
  \in \mathbb{R}^{d\times k}$ has full rank (i.e.\ rank $k$). This is
  equivalent to the fact that $\nabla\xi(x)^TM^{-1}\nabla\xi(x) \in \mathbb{R}^{k\times k}$
  is positive definite for any symmetric positive definite matrix $M\in \mathbb{R}^{d\times d}$.
  \label{rmk-on-full-rank}
\end{remark}

The regular value theorem and Assumption~\ref{assump-xi} imply that $\Sigma$ is a $(d-k)$-dimensional submanifold of~$\mathbb{R}^d$.

\section{Markov chain Monte Carlo algorithms on~$\Sigma$ and $T^*\Sigma$}
\label{sec-two-algorithms}

In this section, we introduce two MCMC algorithms sampling the probability
measure~$\nu_{\Sigma}$ on the level set $\Sigma$: a first one in the state space $\Sigma$, based on
the standard MALA method (see Section~\ref{subsec-mama-sigma}); and a second
one in the phase space $T^*\Sigma$, based on HMC and its generalizations (see
Section~\ref{subsec-hmc}). Both algorithms use set-valued proposal
maps which encode the numerical projections of an unconstrained move back to the submanifold. Examples of such proposal maps are presented in Section~\ref{sec:numerical_computation_constraint}.

\subsection{Multiple projection Metropolis-adjusted Langevin algorithm on~$\Sigma$}
\label{subsec-mama-sigma}

The first algorithm we consider generates a Markov chain on the submanifold~$\Sigma$. We suppose that Assumption~\ref{assump-xi} holds. We first introduce in Section~\ref{sec:construction_set_valued_MALA} the set-valued proposal map that will be used in the algorithm, the algorithm itself being presented in Section~\ref{sec:presentation_MALA}. Finally, we provide a result on the reversibility of this algorithm in Section~\ref{sec:consistency_MALA}.

\subsubsection{Construction of the set-valued map}
\label{sec:construction_set_valued_MALA}

The objective of this section is to build a map which will be used to propose
a move from a point in $\Sigma$ to another point in $\Sigma$ in the Metropolis-Hastings algorithm.

  \paragraph{Projection of the unconstrained move.}

The modified potential $\overline{V}$ in Assumption~\ref{assump-xi} is
used to define the drift term in the proposal function; see the discussion after Theorem~\ref{thm-mala-on-sigma} below
for further considerations on the choice of~$\overline{V}$. When $\overline{V} = 0$, random walk
type proposals are recovered, while the standard MALA algorithm corresponds to
the choice $\overline{V} = V$.

Given any $x\in \Sigma$, we denote by $U_{x}$ a $d\times (d-k)$ matrix whose
$(d-k)$ column vectors form an orthonormal basis of the tangent space
$T_{x}\Sigma$, so that $U_{x}^TU_{x} = I_{d-k} \in \mathbb{R}^{(d-k)\times (d-k)}$.
For any fixed~$x \in \Sigma$, in view of Assumption~\ref{assump-xi}, we can consider without loss of
generality that the map $x' \mapsto U_{x'}$ is chosen so that it is smooth in
a neighborhood of~$x$ (in $\Sigma$). Such a map can for example be explicitly constructed
by first considering an orthonormal basis~$(u_{x,i})_{1 \leq i \leq d-k}$
of~$T_x \Sigma$. Let us then introduce the orthogonal projector onto~$T_{x'}
\Sigma$, namely $P(x') = I_d - \nabla\xi(x')(\nabla\xi^T\nabla\xi)^{-1}(x')
\nabla\xi(x')^T$, for $x' \in \Sigma$. Then, we compute $v_{x',i} = P(x')u_{x,i}$ and
orthonormalize the family~$(v_{x',i})_{1 \leq i \leq d-k}$ using the
Gram--Schmidt algorithm, which leads to an orthonormal basis~$(u_{x',i})_{1
\leq i \leq d-k}$ of $T_{x'}\Sigma$. This construction is indeed possible in a neighborhood of~$x$ where~$\|P(x)-P(x')\|$ (in any matrix norm) is
sufficiently small so that the family $(v_{x',i})_{1 \leq i \leq d-k}$ is linearly independent. 

Starting from a point $x\in \Sigma$, the proposal map increments the position using a velocity in the tangent space~$T_x\Sigma = \{ U_x v\,|\, v\in\mathbb{R}^{d-k} \}$, adds some drift, and projects back in the direction of~$\nabla\xi(x)$. This is encoded by the function $F_x: \mathbb{R}^{d-k}\times \mathbb{R}^k\rightarrow \mathbb{R}^d$ defined as 
\begin{align}
  F_x(v,c) = x - \tau\nabla \overline{V}(x) + \sqrt{2\tau}U_x v + \nabla\xi(x)c\,,\quad \, v\in \mathbb{R}^{d-k}\,,~c\in \mathbb{R}^k\,,
  \label{fx-mala}
\end{align} 
where $\tau>0$ is a fixed timestep, and $c$ has to be chosen so that $F_x(v,c) \in \Sigma$. More precisely, for all $v\in \mathbb{R}^{d-k}$, we introduce the set of all possible images by~\eqref{fx-mala}:
    \begin{align}
      \begin{split}
     \mathcal{F}_x(v) =& \left\{ F_x(v,c) \in \mathbb{R}^d \,\middle|\, \exists c\in \mathbb{R}^k,~ \xi(F_x(v,c))=0 \right\}\\
	=& \left\{y\in\Sigma ~\middle|~ \exists c \in \mathbb{R}^k,~y = x - \tau\nabla \overline{V}(x) + \sqrt{2\tau}\,U_x v + \nabla\xi(x) c \right\} \,,
      \end{split}
      \label{psi-map-mala-on-levelset-new}
    \end{align}
where $c$ in \eqref{psi-map-mala-on-levelset-new} plays the role of Lagrange
multipliers\,\footnote{More precisely, $c$ is a Lagrange multiplier when $y = x - \tau\nabla
  \overline{V}(x) + \sqrt{2\tau}\,U_x v + \nabla\xi(y) c$, which is 
  the Euler-Lagrange equation of the constrained problem $\min_{y}\{|y-
  (x - \tau\nabla \overline{V}(x) + \sqrt{2\tau}\,U_x v)|^2 \}$ subject to
  $\xi(y)=0$. The solution to this minimization problem differs from the expression in~\eqref{psi-map-mala-on-levelset-new} in the last term since~$\nabla \xi(x)c$ has to be replaced with~$\nabla \xi(y)c$. 
  We refer to Proposition~3.1 of~\cite{projection_diffusion} and Chapter~3.2 of~\cite{lelievre2010free}
  for related discussions. 
  Note that it is actually important to use~$\nabla\xi(x)c$ instead of
  $\nabla\xi(y)c$ to project states back to $\Sigma$, in order to obtain a numerical
  scheme which enjoys nice properties such as reversibility and symplecticity. Despite this difference, we will call $c$ a Lagrange multiplier (function) throughout this paper.} associated with the constraint
$\xi(F_x(v,c))=0$. 
  It is crucial to note that, given $y\in \Sigma$, the vector $v$ such that
$y=F_x(v,c)$ for some $c\in\mathbb{R}^k$ (i.e.\ $y \in \mathcal{F}_x(v)$) is uniquely determined (using the facts that $U^T_x\nabla\xi(x)=0$ and $U_x^TU_x=I_{d-k}$) by $v=G_x(y)$, where $G_x: \Sigma\rightarrow \mathbb{R}^{d-k}$ is
\begin{align}
    G_x(y) = \frac{1}{\sqrt{2\tau}}U_x^T\left(y-x+\tau\nabla \overline{V}(x)\right)\,.
  \label{v-by-y}
  \end{align}
 On the contrary, when $G_x(y)=v$ for some $y\in \Sigma$, it holds $y=F_x(v,c)$ with
 $$c=(\nabla\xi(x)^T\nabla\xi(x))^{-1}\nabla\xi(x)^T(y-x+\tau\nabla\overline{V}(x))$$
 and, therefore, $y\in \mathcal{F}_x(v)$. 
 To summarize,  
 \begin{equation}
   \forall\,(x,y)\in \Sigma\times \Sigma, \quad \forall\,v\in \mathbb{R}^{d-k}, \qquad
   y \in \mathcal{F}_x(v)~\Longleftrightarrow~ v = G_x(y)\,.
   \label{relation-between-y-and-v}
 \end{equation}
  \paragraph{Theoretical results on the projections.}
  Before introducing the set-valued proposal map which will be used in the algorithms, let us state the following results on the differentiability of the Lagrange multiplier functions, which show that there are various branches for the solutions to the constraint equation, and which motivate Assumption~\ref{assump-psi-on-state-space} below on the set-valued proposal map. The first result (Proposition~\ref{prop-differentiability-mala}) focuses on the properties of a single branch, while the second one states properties of the ensemble of solutions. Not surprisingly, the conditions in Proposition~\ref{prop-set-f-mala} are more restrictive than the ones in Proposition~\ref{prop-differentiability-mala}, see the discussion in Remark~\ref{rmk:non-tangential} below. 

\begin{proposition}
  For any $x\in \Sigma$, define the set
    \begin{align}
    C_x = \left\{y\in\Sigma\,\middle|\, \det\left(\nabla\xi(y)^T \nabla\xi(x)\right)=0\right\}\,.
    \label{c-x-in-prop1}
  \end{align}
  For all $y \in \Sigma\setminus C_x$, denoting by $v=G_x(y)\in \mathbb{R}^{d-k}$
  (so that $y \in \mathcal{F}_x(v)$), there exists a neighborhood
  $\mathcal{Q}\subset \Sigma$ of $y$, such that the following properties are satisfied:
  \begin{enumerate}
    \item
The map $G_x|_{\mathcal{Q}}: \mathcal{Q}\rightarrow G_x(\mathcal{Q}) \subset
      \mathbb{R}^{d-k}$ is a $C^1$-diffeomorphism. Moreover, $y=(G_x|_{\mathcal{Q}})^{-1}(v)$ and $(G_x|_{\mathcal{Q}})^{-1}(\bar{v})\in \mathcal{F}_x(\bar{v})$ for all $\bar{v} \in G_x(\mathcal{Q})$;
\item
      The Lagrange multiplier function $c: G_x(\mathcal{Q})\rightarrow \mathbb{R}^k$, defined by 
       \begin{align}
	 \bar{v}\in G_x(\mathcal{Q}), \qquad c(\bar{v}) =
	 \left(\nabla\xi(x)^T \nabla\xi(x)\right)^{-1}\nabla\xi(x)^T \left[(G_x|_{\mathcal{Q}})^{-1}(\bar{v})
	 - x + \tau \nabla \overline{V}(x)\right],
	 \label{multiplier-c-explicit}
       \end{align}
is $C^1$-differentiable and satisfies $y = F_x\left(v,c(v)\right)$ and
      $F_x\left(\bar{v},c(\bar{v})\right) \in \mathcal{Q} \cap \mathcal{F}_x(\bar{v})$, for all $\bar{v}\in G_x(\mathcal{Q})$.
      Furthermore, any $C^1$-differentiable function
      $\widetilde{c}:\mathcal{O}\rightarrow \mathbb{R}^k$, where $\mathcal{O}$
      is a neighborhood of $v$, such that $y=F_x(v,\widetilde{c}(v))$ and $F_x\left(\bar{v},\widetilde{c}(\bar{v})\right) \in \mathcal{Q}\cap \mathcal{F}_x(\bar{v})$ for all $\bar{v} \in \mathcal{O}$, coincides with the function $c(\cdot)$ in~\eqref{multiplier-c-explicit} on $\mathcal{O}\cap G_x(\mathcal{Q})$.
  \end{enumerate}
  \label{prop-differentiability-mala}
\end{proposition}

Concerning the set $\mathcal{F}_x(v)$ in \eqref{psi-map-mala-on-levelset-new}, we have the following result.
\begin{proposition}
  For any $x\in \Sigma$, define the subset of $\mathbb{R}^{d-k}$ 
  \begin{equation} 
 \mathcal{N}_x = G_x(C_x) =  
    \left\{G_x(y)\,\middle|\, y\in\Sigma, \ \det\left(\nabla\xi(y)^T \nabla\xi(x)\right)=0\right\}\,,
    \label{set-nx}
  \end{equation} 
    where $C_x$ is defined in~\eqref{c-x-in-prop1}.
  Then the following properties are satisfied:
  \begin{enumerate}
    \item
      $\mathcal{N}_x$ is a closed set of zero Lebesgue measure in $\mathbb{R}^{d-k}$;
    \item
      For all $v \in \mathbb{R}^{d-k}\setminus \mathcal{N}_x$, the set $\mathcal{F}_x(v)$ is finite (possibly empty); 
    \item
      Consider $v \in \mathbb{R}^{d-k}\setminus \mathcal{N}_x$ such that
      $n=|\mathcal{F}_x(v)| \geq 1$, and denote by $\mathcal{F}_x(v) =
      \left\{y^{(1)}, y^{(2)}, \ldots, y^{(n)}\right\}$. Then there exists a
      neighborhood $\mathcal{O}$ of $v$ and $n$ different $C^1$-differentiable
      Lagrange multiplier functions $c^{(i)}: \mathcal{O}\rightarrow
      \mathbb{R}^k$, which are locally given by \eqref{multiplier-c-explicit} by applying Proposition~\ref{prop-differentiability-mala}
      with $y=y^{(i)}$, such that $F_x\left(v,c^{(i)}(v)\right)=y^{(i)}$ and 
      \begin{align}
	\forall\bar{v}\in \mathcal{O}, \qquad \mathcal{F}_x(\bar{v}) = \left\{
	  F_x\left(\bar{v}, c^{(1)}(\bar{v})\right),~
	F_x\left(\bar{v}, c^{(2)}(\bar{v})\right),\ldots, F_x\left(\bar{v}, c^{(n)}(\bar{v})\right)\right\}\,.
      \end{align}
            In particular, $|\mathcal{F}_x(\bar{v})|=n$ for all $\bar{v}\in \mathcal{O}$.
      \item
        $\mathbb{R}^{d-k}$ is the disjoint union of the subsets
	$\mathcal{N}_x$, $\mathcal{B}_{x,0}$,
      $\mathcal{B}_{x,1}, \ldots$, i.e.\ $\dps \mathbb{R}^{d-k}=\left(\bigcup_{i=0}^{+\infty} \mathcal{B}_{x,i}\right)\cup
      \mathcal{N}_x$, where the sets
      \begin{equation}
	\begin{aligned}
	  \mathcal{B}_{x,0} =& \left\{v\in \mathbb{R}^{d-k} \, \middle| \, |\mathcal{F}_x(v)| = 0\right\}\,,\\
	  \mathcal{B}_{x,i} =& \left\{v\in \mathbb{R}^{d-k}\setminus
	  \mathcal{N}_x\, \middle| \, |\mathcal{F}_x(v)| = i\right\}\,, \quad i=1,2,\ldots\,,
	\end{aligned}
      \end{equation}
      are open subsets of $\mathbb{R}^{d-k}$.   
  \end{enumerate}
  \label{prop-set-f-mala}
\end{proposition}
The proofs of Propositions~\ref{prop-differentiability-mala}
and~\ref{prop-set-f-mala} are given in Section~\ref{sec-proofs-MALA}. We point out that the condition $\det\left(\nabla\xi(y)^T \nabla\xi(x)\right)=0$ in~\eqref{c-x-in-prop1} is the non-tangential condition considered in Definition~2.1 of \cite{hmc-submanifold-tony}.

\begin{remark}
  Note that, on the one hand, the condition $v\in \mathbb{R}^{d-k}\setminus \mathcal{N}_x$ in
  Proposition~\ref{prop-set-f-mala} implies that~$y^{(i)}\in \Sigma\setminus C_x$. On the other hand, in contrast to Proposition~\ref{prop-set-f-mala},
  Proposition~\ref{prop-differentiability-mala} holds for $y\in \Sigma\setminus C_x$, even if $v=G_x(y)\in \mathcal{N}_x$. The difference comes from the fact that there may be multiple solutions for a given~$v$,
  some of which satisfy the non-tangential condition while  others do not (see
  for instance Figure~\ref{illustrative-figures}, middle picture: the upper
  point $y_1$ and the lower point $y_2$ satisfy the non-tangential condition,
  but the points on the vertical segment, \textit{e.g.}, the point $y_3$, do not). In fact, as shown in Figure~\ref{illustrative-figures}, the set $\mathcal{F}_x(v)$ may become quite complicated when $v\in \mathcal{N}_x$. In practice, in the algorithm we will present below, the probability to draw a velocity in $\mathcal{N}_x$ is zero.
 \label{rmk:non-tangential}
 \end{remark}

\begin{figure}[htp]
  \includegraphics[width=15cm]{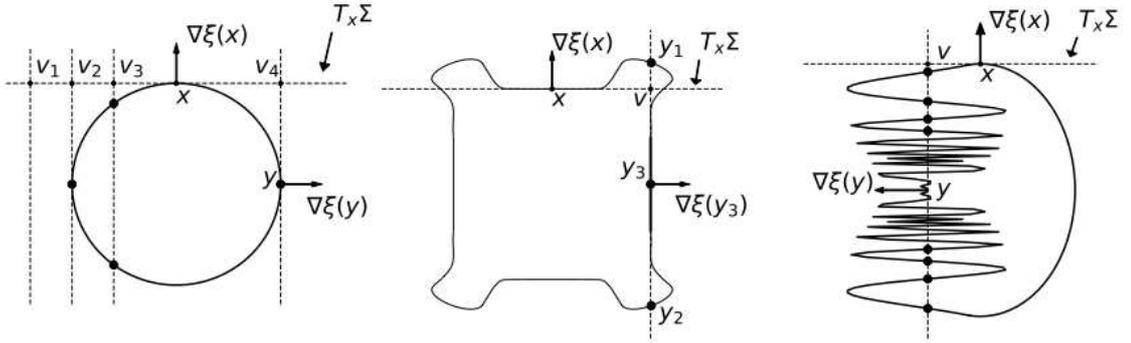}
  \caption{Illustrative examples of the sets $\mathcal{F}_x(v)$
  in~\eqref{psi-map-mala-on-levelset-new} and $\mathcal{N}_x$ in
  Proposition~\ref{prop-set-f-mala} for three different submanifolds. In each case, 
  we choose $\overline{V}\equiv 0$ and $\tau=1/2$. 
 The dashed horizontal lines show the tangent space at a given state $x$, while
  the dashed vertical lines show the direction used to find intersection
  points in \eqref{psi-map-mala-on-levelset-new}.
  In all three cases, $U_x=(1,0)^T$. Left: $\xi(x)=\frac{1}{2}(x^2_1 + x^2_2 -1): \mathbb{R}^2\rightarrow \mathbb{R}$.  At $x=(0,1)^T$, the set~$\mathcal{F}_x(v)$ has no element for $v=v_1$, $1$
  element for $v=v_2, v_4$, and $2$ elements for $v=v_3$, respectively. The vector~$v_4$ is in
  $\mathcal{N}_x$ and $\nabla\xi(y)^T\nabla\xi(x)=0$, where $y=(1,0)^T\in
  \mathcal{F}_x(v_4)$. A similar result is true for $v_2$. See Example~\ref{ex:sphere} for further comments on this case.
  Middle: $v\in \mathcal{N}_x$ since $v=G_x(y_3)$, where $y_3 \in \mathcal{F}_x(v)$ satisfies
  $\nabla\xi(y_3)^T\nabla\xi(x)=0$. The set $\mathcal{F}_x(v)$ contains a
  segment including $y_3$, i.e.\ it has infinitely many (uncountable) elements. Right: 
  $v\in \mathcal{N}_x$ since $v=G_x(y)$ where $y \in \mathcal{F}_x(v)$
  satisfies $\nabla\xi(y)^T\nabla\xi(x)=0$.
  The set $\mathcal{F}_x(v)$ contains infinitely many (countable) elements; see the intersection points of the dashed vertical line and the
  submanifold, illustrated using black dots.
  The submanifold in the right plot is smooth and can be constructed using
  the smooth curve defined as $x_1=\mathrm{e}^{-1/x_2^2}\sin(1/x_2)$ when $x_2\neq 0$, and $x_1=0$ when $x_2=0$. 
    \label{illustrative-figures}}
\end{figure}

  \paragraph{The set-valued proposal map.}
The set-valued proposal map, denoted by $\Psi_x(v)$, takes values in a subset of $\mathcal{F}_x(v)$ and encodes the projections obtained numerically by solving the (nonlinear)
constraint equation (see Section~\ref{sec:numerical_computation_constraint} for concrete examples):
\begin{align}
  \textrm{Find } c \in \mathbb{R}^k \textrm{ such that } \xi(F_x(v,c))=\xi\left(x - \tau\nabla \overline{V}(x) + \sqrt{2\tau}U_x v + \nabla\xi(x) c\right) = 0\,.
      \label{psi-by-nonlinear-constraint-eqn}
\end{align}
The set $\Psi_x(v)$ is empty if no solutions
to~\eqref{psi-by-nonlinear-constraint-eqn} are found, and contains more than
one element when multiple solutions can be found. We define the set 
\begin{equation}
  \mbox{Im}\Psi_x = \left\{y\in \Sigma\,\middle|\, \exists v\in \mathbb{R}^{d-k},\, y\in \Psi_x(v)\right\}\,, \quad x \in \Sigma\,.
  \label{def-im-of-psi-x}
\end{equation}
Note that \eqref{v-by-y} implies that
\begin{align}
  \forall\,x\in \Sigma, \quad \forall~v, \bar{v}\in \mathbb{R}^{d-k}, \qquad v
  \neq \bar{v}~ \Longrightarrow~ \Psi_x(v) \cap \Psi_x(\bar{v}) =\emptyset\,.
  \label{psi-given-y-v-is-unique}
\end{align}

Motivated by Propositions~\ref{prop-differentiability-mala} and~\ref{prop-set-f-mala}, we make the following assumption on the map $\Psi_x$.

\begin{assumption}
  \label{assump-psi-on-state-space}
 The following properties are satisfied:
  \begin{enumerate}
    \item For any~$x \in \Sigma$ and for all $v \in \mathbb{R}^{d-k}$, $\Psi_x(v)$ contains at most a finite number of elements;
    \item
The set
\begin{align}
  \mathcal{D} = \left\{(x,y)\in \Sigma\times \Sigma ~\middle|~ y \in
  \mbox{\textup{Im}}\Psi_x,~x \in \mbox{\textup{Im}}\Psi_{y} \right\}
  \label{admissible-set-two-steps}
\end{align}
      is a non-empty measurable subset of $\Sigma\times \Sigma$;
    \item
For any~$x \in \Sigma$ and for all $v \in \mathbb{R}^{d-k}$ such that $n=|\Psi_x(v)|\ge 1$,
      denote by $\Psi_x(v)=\{y_1, \ldots, y_n\}$ where (upon reordering)
      $y_1, \ldots, y_{m} \in \Sigma\setminus C_x$ for some $m \le n$.
      When $m \ge 1$, there exists a neighborhood~$\mathcal{O}$ of $v$
      and $m$ different $C^1$-diffeomorphisms $\Psi_x^{(j)}: \mathcal{O}\rightarrow \Psi_x^{(j)}(\mathcal{O}) \subset\Sigma$ (for $1 \le j \le m$) such that 
      $y_j = \Psi_x^{(j)}(v)$ and $\Psi_x^{(j)}(\bar{v}) \in \Psi_x(\bar{v})$ for all $\bar{v}\in \mathcal{O}$. 
  \end{enumerate}
\end{assumption}
We refer to Section~\ref{sec:numerical_computation_constraint} for a discussion on set-valued maps satisfying these assumptions.

\begin{example}
  \label{ex:sphere}
Let us give a simple concrete example of a set-valued map $\Psi_x$. Consider $\xi(x)=(x^2_1 + x^2_2 -1)/2$ for $x=(x_1,x_2)^T\in \mathbb{R}^2$.
  The level set $\Sigma = \{x\in \mathbb{R}^2\,|\,\xi(x)=0\}$ is the unit
  circle (see Figure~\ref{illustrative-figures}, left). Fixing without loss of generality $x=(0,1)^T$, it
  holds $\nabla\xi(x) = (0,1)^T$ so that we can choose $U_x=(1,0)^T$. We set
  $\overline{V}(x)\equiv 0$ and $\tau=1/2$. In this case, the constraint
  equation~\eqref{psi-by-nonlinear-constraint-eqn} has $0$, $1$, or $2$
  solutions, when $|v|>1$, $|v|=1$, or $|v|<1$, respectively. The sets in Proposition~\ref{prop-differentiability-mala} are
  then $C_x = \{(-1, 0)^T, (1,0)^T\}$ and $\mathcal{N}_x=\{-1,1\}$. In this
  case, one can define the map $\Psi_x(v) = \big\{(v, \sqrt{1-v^2}\,)^T, (v,
  -\sqrt{1-v^2}\,)^T\big\}$, for $|v|<1$, $\Psi_x(v) = \{(v,0)\}$ for~$|v|=1$,
  and $\Psi_x(v) = \emptyset$, for $|v|>1$.
  For $|v|<1$, the maps $\Psi^{(j)}_x$ can be taken as 
  $\Psi^{(1)}_x(\bar{v})=(\bar{v}, \sqrt{1-\bar{v}^2}\,)^T$,
  $\Psi^{(2)}_x(\bar{v})=(\bar{v}, -\sqrt{1-\bar{v}^2}\,)^T$. They are clearly $C^1$-differentiable in~$\{v, |v|<1\}$.
  It is easy to check that $\Psi_x$ satisfies Assumption~\ref{assump-psi-on-state-space}.
\end{example}

\subsubsection{Presentation of the algorithm}
\label{sec:presentation_MALA}

We are now in position to introduce the multiple projection MALA algorithm, see Algorithm~\ref{algo-mala-on-sigma}. 
\begin{algorithm}[htbp]
  \caption{Multiple projection MALA on $\Sigma$}
  \label{algo-mala-on-sigma}
  \begin{algorithmic}[1]
    \State Choose $x^{(0)} \in \Sigma$ and $N > 1$. Set $i=0$.
    \While{$i<N$}
    \State
    Set $x=x^{(i)}$.
    \State
    Randomly draw $v \in \mathbb{R}^{d-k}$ according to $v
    \sim (2\pi\beta^{-1})^{-\frac{d-k}{2}} \mathrm{e}^{-\frac{\beta|v|^2}{2}}\,dv$.
    \State 
    Compute the set $\Psi_{x}(v)$ by solving \eqref{psi-by-nonlinear-constraint-eqn}. 
   If $\Psi_x(v)=\emptyset$, set $x^{(i+1)}=x$ and \Goto{algo-mala-end-of-while}.
  \State
    Randomly draw a proposal state $y\in\Psi_x(v)$ with probability $\omega(y\,|\,x,v)$, and set $\widetilde{x}^{(i+1)}=y$. 
    \State
    Compute $v'$ as given by~\eqref{v-prime-inverse}, and the set $\Psi_{y}(v')$. 
   \State 
    Check whether $x \in \Psi_y(v')$. If this is not true,
    set $x^{(i+1)}=x$ and \Goto{algo-mala-end-of-while}.
    \State
    Draw a random number $r$ uniformly distributed in~$[0,1]$. Set 
  \begin{align*}
    x^{(i+1)} = 
    \begin{cases}
      \widetilde{x}^{(i+1)}, &\quad \mbox{if}~ r \le a(\widetilde{x}^{(i+1)}\,|\,x^{(i)})\\
      x^{(i)}\,, &\quad \mbox{otherwise}\,,
    \end{cases}
  \end{align*}
    with the acceptance probability defined in~\eqref{rate-for-mala-on-sigma}.
  \State $i \gets i+1$.  \label{algo-mala-end-of-while}
  \EndWhile
  \end{algorithmic}
\end{algorithm}

Let us discuss the different steps of one iteration of the algorithm: (i) the choice of a proposed move~$\widetilde{x}^{(i+1)}$ starting from a current state~$x=x^{(i)}$; (ii) the verification of the reversibility condition; (iii)~the acceptance/rejection of the proposal through a Metropolis-Hastings procedure. 

\paragraph{Choice of proposal.}
The proposal state $\widetilde{x}^{(i+1)}$ is constructed as follows:
\begin{enumerate}
\item
  Draw $v\in \mathbb{R}^{d-k}$ according to the $(d-k)$-dimensional Gaussian measure with density
  \begin{align}
    \gamma(dv) = \left(2\pi\beta^{-1}\right)^{-\frac{d-k}{2}} \mathrm{e}^{-\frac{\beta|v|^2}{2}}\,dv, \qquad v\in \mathbb{R}^{d-k}\,.
    \label{rescaled-gaussian}
  \end{align}
\item
  Set $\widetilde{x}^{(i+1)}=x$ if $\Psi_x(v)=\emptyset$; otherwise, 
    randomly choose an element $y\in \Psi_x(v)$ with probability
    $\omega(y\,|\,x,v)$, and set $\widetilde{x}^{(i+1)}=y$.
\end{enumerate}
 The probabilities $(\omega(y\,|\,x,v))_{y\in \Psi_x(v)}$ are chosen in
 $(0,1]$, such that the function $\omega(\cdot\,|\,\cdot,\cdot)$ is measurable and $\sum_{y\in \Psi_x(v)} \omega(y\,|\,x,v)=1$.
For instance, $\omega(\cdot\,|\,x,v)$ can be chosen as a uniform law,
i.e.\ $\omega(y\,|\,x,v) = \frac{1}{|\Psi_x(v)|}$ for all $y\in
\Psi_x(v)$. Alternatively, $\omega(\cdot\,|\,x,v)$ can be chosen so that it is more likely to jump to states that are close
to (respectively far from) the current state~$x=x^{(i)}$, i.e.\
$\omega(y\,|\,x,v) \ge \omega(y'\,|\,x,v)$ whenever $|y-x|\le |y'-x|$ (respectively
$|y-x|\ge |y'-x|$). Note also that the set $\mathcal{N}_x$ in Proposition~\ref{prop-set-f-mala}
has zero measure under the Gaussian measure in~\eqref{rescaled-gaussian}, since it has zero Lebesgue measure. Therefore, the probability to draw a velocity in $\mathcal{N}_x$ is zero.

\paragraph{Reversibility check.} 
Once the proposal states have been computed using the map~$\Psi_x$, a reversibility
check is in general needed in order to guarantee that the algorithm generates a
reversible Markov chain and therefore samples the target distribution $\nu_{\Sigma}$ without
bias~\cite{goodman-submanifold}. Specifically, after randomly choosing $y\in\Psi_x(v)$, one verifies that $x \in \Psi_{y}(v')$ for some $v'\in \mathbb{R}^{d-k}$. Note that such an element~$v'$ is uniquely given by
    \begin{align}
      v'= G_y(x) = \frac{1}{\sqrt{2\tau}} U_{y}^T\left(x-y+ \tau\nabla \overline{V}(y)\right)\,. 
      \label{v-prime-inverse}
    \end{align} 
    Therefore, one only needs in practice to check that, after numerically solving for the solutions $\bar{c}$ of the equation
     \begin{align}
       x' = F_y(v',\bar{c}), \qquad \xi\left( x'\right)=0\,,
      \label{psi-map-mala-on-levelset-inverse-check}
    \end{align}
    one of the solutions $\bar{c}$ is such that $x'=x$. 
    
    Note that, with $v'$ given in~\eqref{v-prime-inverse}, we actually have $x'=x$ in~\eqref{psi-map-mala-on-levelset-inverse-check} for the choice
  \begin{align}
    \bar{c} = \left[\nabla\xi(y)^T\nabla\xi(y)\right]^{-1}\nabla\xi^{T}(y)
    \left(x-y+\tau \nabla \overline{V}(y)\right)\,.
    \label{target-lagrange-multiplier-cv-prime}
  \end{align}
In other words, it is guaranteed that $x\in \mathcal{F}_y(v')$.
    The reversibility check thus amounts to verifying that the Lagrange
    multiplier in~\eqref{target-lagrange-multiplier-cv-prime} is indeed among the
    possibly many solutions $\Psi_y(v')$ found by the numerical method used to
    solve~\eqref{psi-map-mala-on-levelset-inverse-check} in practice.
    If this is indeed the case, since $\omega(\cdot\,|\,y,v')$ is assumed to be positive on $\Psi_y(v')$, there is a positive probability to go back to the initial state~$x^{(i)}$ when starting from the proposed state~$\widetilde{x}^{(i+1)}$. A consequence of this discussion is the following important remark.

\begin{remark}
In the case where the numerical solver is able to find all the solutions to
  \eqref{psi-by-nonlinear-constraint-eqn}, Step~$8$ in
  Algorithm~\ref{algo-mala-on-sigma} (the ``reversibility check") is actually not needed since, by construction, $x\in \Psi_y(v')$. 
  \label{rmk-Reversibility-check-mala}
\end{remark}

\paragraph{Metropolis procedure.}
When the reversibility check is successful, i.e.\ $x\in\Psi_{y}(v')$, a
Metropolis-Hastings step is performed to either accept or reject the proposed
move to the state $y\in\Psi_x(v)$, based on the acceptance
probability~$a(\widetilde{x}^{(i)}\,|\,x^{(i)})$, defined by
\begin{align}
  \begin{split}
    a(y\,|\,x) =&
    \min\left\{1,~\frac{\omega(x\,|\,y,v')}{\omega(y\,|\,x,v)}\,\mathrm{e}^{-\beta\left[\left(V(y)+\frac{1}{2}|v'|^2\right) - \left(V(x) + \frac{1}{2}|v|^2\right)\right]}\right\}\,, 
  \end{split}
\label{rate-for-mala-on-sigma}
\end{align}
where, we recall, $v=G_x(y)$ and $v'=G_y(x)$. In particular, when 
the state $y$ is drawn with a uniform probability in $\Psi_x(v)$ (i.e.\
 for all $x\in\Sigma$ and $v\in \mathbb{R}^{d-k}$, $\omega(y\,|\,x,v)=\frac{1}{|\Psi_x(v)|}$ for all $y \in \Psi_x(v)$), the acceptance rate~\eqref{rate-for-mala-on-sigma} becomes 
\begin{align*}
  a(y\,|\,x) =& \min\left\{1,~\frac{|\Psi_x(v)|}{|\Psi_y(v')|}\,\mathrm{e}^{-\beta\left[\left(V(y)+\frac{1}{2}|v'|^2\right) - \left(V(x) + \frac{1}{2}|v|^2\right)\right]}\right\}\,. 
\end{align*}

\subsubsection{Consistency of the algorithm}
\label{sec:consistency_MALA}
Let us recall the following definition of the reversibility of Markov chains.
\begin{definition}
A Markov chain on $\Sigma$ with transition probability kernel $q(x,dy)$ is
  reversible with respect to the measure $\nu_{\Sigma}(dx)$ if the following equality
  holds: For any bounded measurable function~$f: \Sigma\times \Sigma\rightarrow \mathbb{R}$,
\begin{align}
  \int_{\Sigma\times \Sigma} f(x,y)\, q(x,dy)\,\nu_{\Sigma}(dx) =
  \int_{\Sigma\times \Sigma} f(y,x)\, q(x,dy)\,\nu_{\Sigma}(dx).
  \label{usual-detailed-balance-condition-integral}
\end{align}
  \label{def-reversible}
\end{definition}
One can verify that a Markov chain $(x^{(i)})_{i\ge 0} \subset \Sigma$ is
reversible with respect to $\nu_{\Sigma}$ if and only if, for any integer $n>0$, the
law of the sequence $(x^{(0)},\, x^{(1)},\, \ldots,\, x^{(n)})$  is the same as the law of the sequence $(x^{(n)},\, x^{(n-1)},\,
\ldots,\, x^{(0)})$ when~$x^{(0)}$ is distributed according to~$\nu_{\Sigma}$. 
In particular, this implies that $\nu_{\Sigma}$ is an invariant measure of the Markov chain.

The following theorem states that Algorithm~\ref{algo-mala-on-sigma}  indeed generates a
reversible Markov chain with respect to the
measure~$\nu_{\Sigma}=Z_{\nu_{\Sigma}}^{-1}\, \mathrm{e}^{-\beta V(x)} \sigma_\Sigma(dx)$ defined in~\eqref{measure-mu-on-m}. Its proof is given in Section~\ref{sec-proofs-MALA}.

\begin{theorem}
  Under Assumptions~\ref{assump-xi} and~\ref{assump-psi-on-state-space},
  Algorithm~\ref{algo-mala-on-sigma} generates a Markov chain on $\Sigma$ which is
  reversible with respect to the probability measure $\nu_{\Sigma}$.
  In particular, this Markov chain preserves the probability measure $\nu_{\Sigma}$.
  \label{thm-mala-on-sigma}
\end{theorem}

To prove the ergodicity of the Markov chain generated by Algorithm~\ref{algo-mala-on-sigma}, one still needs to verify irreducibility. 
We refer to~\cite{M2AN_2007__41_2_351_0,convergence-hmc,livingstone2019,Schtte1999ConformationalDM} for related results in this direction in the unconstrained
case, which can be adapted to Markov chains involving constraints as in~\cite{Hartmann2008}.

As already discussed at the beginning of
Section~\ref{sec:construction_set_valued_MALA}, we have the freedom to choose
the function $\overline{V}$ when defining the set-valued proposal map $\Psi_x(v)$. 
In practice, choosing $\overline{V}$ as either~$V$ or some simplified (coarse-grained) approximation of~$V$ may be helpful in increasing the acceptance rate in the Metropolis procedure, and therefore allowing for a larger timestep~$\tau$. Alternatively, setting $\overline{V}\equiv 0$, \eqref{psi-by-nonlinear-constraint-eqn} becomes 
\begin{align}
  \textrm{Find } c \in \mathbb{R}^k \textrm{ such that } \xi\left(x + \sqrt{2\tau}U_x v + \nabla\xi(x) c\right) = 0\,.
  \label{rw-proposal-constraint}
\end{align}
This yields the multiple projection random walk Metropolis-Hastings (RWMH) algorithm on~$\Sigma$.
In this case, assuming furthermore at most one solution of~\eqref{rw-proposal-constraint}
is used (namely $|\Psi_x(v)|=0$ or~$1$), one recovers the MCMC algorithm proposed in~\cite{goodman-submanifold}. 

\subsection{Multiple projection Hybrid Monte Carlo on $T^*\Sigma$}
\label{subsec-hmc}

The second algorithm we consider generates a Markov chain on an extended configuration space, where a momentum~$p$ conjugated to the position~$x$ is introduced. We first make precise the extended target measure in Section~\ref{sec:extended_measure} and then the set-valued proposal function constructed using discretization schemes for constrained Hamiltonian dynamics in Section~\ref{sec:construction_set_valued_HMC}. The complete algorithm is next presented in Section~\ref{sec:presentation_HMC}, while its reversibility properties are stated in Section~\ref{sec:consistency_HMC}.

\subsubsection{Extended target measure}
\label{sec:extended_measure}

Suppose again that Assumption~\ref{assump-xi} holds. Instead of
considering~$v$ as coefficients in the tangent space as in~\eqref{fx-mala}, we work with momenta $p$ which belong to the cotangent space. The state of the system is then described by $(x,p)$, where, for a given position $x \in \Sigma$, the momentum~$p$ belong to~$T^*_x \Sigma$, the cotangent space of~$\Sigma$ at~$x$. This cotangent space can be identified with a linear subspace of~$\mathbb{R}^d$:
\[
T^*_x \Sigma = \left\{ p \in \mathbb{R}^d \, \middle| \, \nabla \xi(x)^T M^{-1} p = 0 \in \mathbb{R}^k \right\} \subset \mathbb{R}^d,
\]
where $M\in \mathbb{R}^{d\times d}$ is a constant symmetric positive definite mass matrix. 
Let us denote by
\[
T^* \Sigma = \left\{ (x,p) \in \mathbb{R}^{d} \times \mathbb{R}^{d} \,\middle| \, \xi(x)=0 \text{ and } \nabla \xi(x)^T M^{-1} p = 0 \right\} \subset \mathbb{R}^d \times \mathbb{R}^d
\]
the associated cotangent bundle, which can be seen as a submanifold of $\mathbb{R}^d \times \mathbb{R}^d$. The phase space Liouville measure on $T^* \Sigma$ is denoted by $\sigma_{T^* \Sigma}(dx \, dp)$. It can be written in tensorial form as
\begin{equation}
  \label{eq:tensor}
  \sigma_{T^*\Sigma}(dx \, dp)= \sigma^{M}_{\Sigma}(dx) \,\sigma^{M^{-1}}_{T^*_x \Sigma}(dp),
\end{equation}
where $\sigma^{M}_{\Sigma}(dx)$ is the surface measure on~$\Sigma$
induced by the scalar product $\langle x,\tilde x \rangle_{M}=x^T M \tilde x$
in $\mathbb{R}^d$, and $\sigma^{M^{-1}}_{T^*_x \Sigma}(dp)$ is the Lebesgue
measure on $T^*_x \Sigma$ induced by the scalar product
\begin{equation}
  \label{eq:scalar_product_M}
  \langle p,\tilde p\rangle_{M^{-1}}=p^T M^{-1} \tilde p
\end{equation}
in $\mathbb{R}^d$. In particular, $\sigma_\Sigma^M$
coincides with $\sigma_\Sigma$ in \eqref{measure-mu-on-m} when $M=I_{d}$ is the identity mass matrix. Note that, in contrast to the measures $\sigma_\Sigma^M(dx)$ and $\sigma^{M^{-1}}_{T^*_x \Sigma}(dp)$,  the measure
$\sigma_{T^* \Sigma}(dx \, dp)$ does not depend on the choice of the mass
tensor $M$~\cite[Proposition 3.40]{lelievre2010free}. For a given $x \in \Sigma$, the orthogonal projection $P_M(x):\mathbb{R}^d \to T_x^*\Sigma$ on $T_x^*\Sigma$ with respect to the scalar product $\langle\cdot, \cdot \rangle_{M^{-1}}$ is given by 
\begin{equation}
P_M(x) = I_d - \nabla\xi(x)
(\nabla\xi^TM^{-1}\nabla\xi)^{-1}(x) \nabla\xi(x)^TM^{-1}.
  \label{projection-pm}
\end{equation} 
  It is well-defined
on~$\Sigma$ thanks to Assumption~\ref{assump-xi} (see Remark~\ref{rmk-on-full-rank}). 

The target measure to sample is
  \begin{align}
    \mu(dx\,dp) = \frac{1}{Z_\mu}\,\mathrm{e}^{-\beta H(x,p)}\,\sigma_{T^*\Sigma}(dx\,dp)\,,
    \label{target-measure-in-phase-space}
  \end{align}
  where $H$ is the Hamiltonian 
  \begin{align*}
    H(x,p) = V(x) + \frac{1}{2}p^T M^{-1}p\,. 
  \end{align*}
   Note that, thanks to the tensorization property~\eqref{eq:tensor} and the separability of the Hamiltonian function,
\begin{equation}
\label{eq:mu_tensor}
\mu(dx \, dp) = \nu_{\Sigma}^M(dx) \, \kappa_x(dp),
\end{equation}
where~(see \cite[Equation~(3.25) in Section~3.2.1.3]{lelievre2010free})
\begin{equation}
  \label{eq:def_nu_M}
  \nu_{\Sigma}^M(dx) = (\mathrm{det} \, M)^{1/2} \left[\mathrm{det}\left(\nabla
  \xi(x)^T M^{-1} \nabla
  \xi(x)\right)\right]^{1/2}\left[\mathrm{det}\left(\nabla \xi(x)^T \nabla
  \xi(x)\right)\right]^{-1/2}\nu_{\Sigma}(dx)\,,
\end{equation}
with $\nu_{\Sigma}$ defined in~\eqref{measure-mu-on-m}, while $\kappa_x$ is a Gaussian measure on~$T^*_x \Sigma$:
\begin{equation}
  \label{eq:kappa}
  \kappa_x(dp)=(2\pi\beta^{-1})^{-\frac{d-k}{2}} \exp \left(-\frac{\beta p^T M^{-1} p}{2}\right)\sigma^{M^{-1}}_{T^*_x \Sigma}(dp).
\end{equation}
In particular, the marginal of $\mu$ in the variable~$x$ is~$\nu_{\Sigma}^M$, which can
be easily related to~$\nu_{\Sigma}$ by~\eqref{eq:def_nu_M} using some importance
sampling weight or by modifying the potential function $V$. 
It coincides in fact with~$\nu_{\Sigma}$ when $M=I_d$ is the identity mass matrix. The idea of the numerical method constructed in this section is to sample~\eqref{eq:mu_tensor}, and then to reweight the positions which are henceforth sampled according to~$\nu_{\Sigma}^M$ in order to obtain positions distributed according to~$\nu_{\Sigma}$.

For further use, let us introduce the momentum reversal map $\mathcal{R}$, which is an involution:
\begin{equation}
  \forall~(x,p)\in T^*\Sigma, \qquad \mathcal{R}(x,p) = (x,-p)\,.
  \label{involution-r}
\end{equation}
Notice that $\mu(dx\,dp)$ is invariant under $\mathcal{R}$. Denote by $\Pi: T^*\Sigma\rightarrow \Sigma$ the projection map 
\begin{equation}
\forall z=(x,p)\in T^*\Sigma, \qquad \Pi(z)=x\,.
  \label{projection-map-pi}
\end{equation}

\subsubsection{Construction of the set-valued map}
\label{sec:construction_set_valued_HMC}

Let us fix $(x,p)\in T^*\Sigma$. The objective of this section is to build a map which will be used to propose a move from $(x,p) \in T^*\Sigma$ to another point in $T^*\Sigma$ in the Metropolis-Hastings algorithm.

\paragraph{Projection of the unconstrained move: RATTLE with momentum reversal.} Given a timestep
$\tau > 0$ and a potential energy function~$\overline{V}$ (see Section~\ref{sec:construction_set_valued_MALA} after Theorem~\ref{thm-mala-on-sigma} for a discussion on the choices of the potential~$\overline{V}$), the multiple projection HMC algorithm uses the following unconstrained move: 
\begin{subnumcases}{}
  p^{1/2} = p -\frac{\tau}{2} \nabla \overline{V}(x) + \nabla\xi(x) \lambda_x, & \label{rattle-1st-step}\\
  x^{1} = x + \tau M^{-1}p^{1/2}, & \nonumber\\
  \xi(x^{1}) = 0, & \label{rattle-constraint-1} \\
  p^{1} = p^{1/2} - \frac{\tau}{2} \nabla \overline{V}(x^1) + \nabla\xi(x^{1})\lambda_p, & \nonumber\\
  \nabla\xi(x^1)^T M^{-1} p^{1} = 0, & \label{rattle-constraint-2}\,\\
  p^{1,-} = -p^{1}, & \label{rattle-momentum-reversal}
\end{subnumcases}
where $\lambda_x, \lambda_p\in \mathbb{R}^k$ are Lagrange multiplier functions
determined by the constraints~\eqref{rattle-constraint-1} and~\eqref{rattle-constraint-2}, respectively. 
Note that the numerical flow from $(x,p)$ to $(x^1,p^{1,-})$ combines the one-step
RATTLE scheme~\cite{geometric-integration} (i.e.\ \eqref{rattle-1st-step}--\eqref{rattle-constraint-2}) with a momentum reversal step in \eqref{rattle-momentum-reversal}.
For this reason, we will call the scheme
\eqref{rattle-1st-step}--\eqref{rattle-momentum-reversal} ``one-step RATTLE with momentum reversal''. 
Let us recall the main properties of the above scheme.

First, for any $(x,p) \in T^*\Sigma$ and $(c_x, c_p) \in \mathbb{R}^{k}\times \mathbb{R}^k$,
we denote by $F: T^*\Sigma\times \mathbb{R}^{k}\times \mathbb{R}^k\rightarrow
\mathbb{R}^d\times \mathbb{R}^{d}$ the map defined by 
\begin{equation}
  \begin{aligned}
    & F(x,p,c_x,c_v)= (\bar{x}, \bar{p}^{\,-}), \quad \mbox{where} \\
   & \bar{x}  = x + \tau M^{-1}p - \frac{\tau^2}{2}M^{-1} \nabla
    \overline{V}(x) + \tau M^{-1} \nabla\xi(x)c_x\,,\\
    & \bar{p}  = p -\frac{\tau}{2} \left(\nabla \overline{V}(x) + \nabla
    \overline{V}(\bar{x})\right) + \nabla\xi(x) c_x + \nabla\xi(\bar{x})c_p\,,\\
    & \bar{p}^{\,-} = - \bar{p}\,.
  \end{aligned}
\label{F-map-related-to-rattle}
\end{equation}
We can then write the one-step RATTLE with momentum reversal from $(x,p)$ to~$(x^1,p^{1,-})$ as 
\begin{equation}
  \left\{ \begin{aligned}
    & (x^{1}, p^{1,-}) = F(x,p, \lambda_x, \lambda_p)\,,\\
    & \xi(x^{1}) = 0\,,\\
    & \nabla\xi(x^1)^T M^{-1} p^{1,-} = 0\,. 
  \end{aligned} \right.
\label{psi-map-by-rattle}
\end{equation}
Once $\lambda_x$ (and therefore $x^{1}$) is known, the Lagrange multiplier
$\lambda_p$ (and therefore $p^{1,-}$) is uniquely defined and can actually be
analytically computed by enforcing the projection of the momenta with the
projection~$P_M$ (see \eqref{projection-pm}):
\begin{align}
\lambda_p =&\, -\left[\nabla\xi(x^1)^T
  M^{-1}\nabla\xi(x^{1})\right]^{-1}\left[\nabla\xi(x^1)^TM^{-1}\left(p
  -\frac{\tau}{2} \left(\nabla \overline{V}(x) + \nabla
  \overline{V}(x^1)\right)+ \nabla\xi(x) \lambda_x\right)\right]\,,  \notag \\
  p^{1,-} =&\, -P_M(x^1)\left(p - \frac{\tau}{2} \left(\nabla \overline{V}(x) + \nabla \overline{V}(x^1)\right) + \nabla\xi(x) \lambda_x\right)\,. \label{lambda-p-exp}
\end{align}
Therefore, the main task in finding $x^1, p^{1,-}$ is to solve the constraint equation \eqref{rattle-constraint-1} for~$\lambda_x$. 
Once $x^1\in \Sigma$ is known, $\lambda_x$ is determined by 
\begin{align}
	\lambda_x = \frac{1}{\tau}
	\left[(\nabla\xi^TM^{-1}\nabla\xi)(x)\right]^{-1}\nabla\xi(x)^T\left[x^1 - x - \tau M^{-1}p + \frac{\tau^2}{2} M^{-1} \nabla
	\overline{V}(x)\right]\,,
	\label{lambda-x-by-x1}
\end{align}
and hence $\lambda_p$, $p^{1,-}$ are uniquely given by \eqref{lambda-p-exp}.

It is crucial to realize that for given configurations $x\in
\Sigma$ and $x^1 \in
\Sigma$, the vector $p\in T_x^*\Sigma$ such that the one-step RATTLE with momentum reversal maps~$(x,p)$ to
$(x^1, p^{1,-})$ for some $p^{1,-}$ is uniquely determined by $p=G_{M,x}(x^1)$, where $G_{M,x}: \Sigma\rightarrow T^*_x\Sigma$ is defined as
\begin{equation}
  G_{M,x}(x^1) = \frac{1}{\tau} P_M(x)M \left(x^1-x+\frac{\tau^2}{2} M^{-1}\nabla \overline{V}(x)\right)\,,
  \label{Gx-hmc}
\end{equation}
and plays a role similar to the map $G_x$ defined in~\eqref{v-by-y} in
Section~\ref{subsec-mama-sigma}. Note that in \eqref{Gx-hmc} we define the map
$G_{M,x}$ with values in $T^*_x\Sigma$ without introducing a local basis,
whereas in \eqref{v-by-y} we define the map $G_x$ with values in
$\mathbb{R}^{d-k}$ instead of $T_x\Sigma$ using a local basis $U_x$. An
alternative definition of $G_{M,x}$ (equivalent to \eqref{Gx-hmc}) involving a local basis is used in the proof of Proposition~\ref{prop-differentiability-hmc} (see \eqref{gx-hmc-rewrite} in Section~\ref{sec-proofs-HMC}).

If the constraint equations~\eqref{rattle-constraint-1} and~\eqref{rattle-constraint-2} are satisfied,
we call $(\lambda_x, \lambda_p)$ an \textit{admissible pair of Lagrange
multipliers} from~$(x,p)$ to~$(x^1, p^{1,-})$. The one-step RATTLE scheme with momentum reversal has the following time-reversal symmetry.

\begin{lemma}
  Suppose that Assumption~\ref{assump-xi} is satisfied. Consider two states~$(x,p),\,(x^1, p^{1,-}) \in T^*\Sigma$ and suppose that $(\lambda_x, \lambda_p)$ is an admissible pair of Lagrange multipliers from~$(x,p)$
  to~$(x^1, p^{1,-})$. Then, 
  \begin{enumerate}
  \item $(\lambda_p, \lambda_x)$ is an admissible pair of Lagrange multipliers from $(x^1,p^{1,-})$ to $(x, p)$;
  \item Suppose that $(\lambda_x^{-}, \lambda_p^{-})$ is an admissible pair of
    Lagrange multipliers from $(x^{1}, p^{1,-})$ to $(\bar{x}, \bar{p})$  such that $\bar{x}=x$. Then, $(\lambda_x^{-}, \lambda_p^{-}) = (\lambda_p, \lambda_x)$ and $\bar{p} = p$.
  \end{enumerate}
  \label{lemma-1}
\end{lemma}

The first property is standard and expresses some form of symmetry of the
RATTLE scheme with momentum reversal~(see e.g.\ \cite[Section~VII.1.4]{geometric-integration}). The second expresses some
rigidity in the choice of the Lagrange multipliers and shows that genuinely
different choices of Lagrange multipliers necessarily correspond to a failure in the time-reversal symmetry of the algorithm; 
see the proof of Lemma~3.2 in~\cite{hmc-submanifold-tony}.
      
\paragraph{Theoretical results on the projections.}
 For $z = (x,p) \in T^*\Sigma$, we introduce the set
\begin{equation}
\begin{aligned}
  \mathcal{F}(z) = \Big\{ (x^1, p^{1,-}) \in T^*\Sigma\, \Big|& \,\exists\,(\lambda_x, \lambda_p) 
  \in \mathbb{R}^{k}\times \mathbb{R}^k, (x^1,p^{1,-}) = F(z,\lambda_x,\lambda_p), \\
  &~\mbox{such that}~ \xi(x^1)=0 ~\mbox{and}~\nabla\xi(x^1)^TM^{-1}p^{1,-}=0 \Big\}\,, 
  \end{aligned}
  \label{f-rattle-map}
\end{equation}
where $F$ is the map defined in \eqref{F-map-related-to-rattle}.
In other words, $\mathcal{F}(z)$ contains all the possible outcomes of the one-step RATTLE with momentum reversal starting from $z$.

Similarly to Propositions~\ref{prop-differentiability-mala} and~\ref{prop-set-f-mala}, we have
the following results on the differentiability of the Lagrange multiplier functions, as well as on the set $\mathcal{F}(z)$.
Their proofs are given in Section~\ref{sec-proofs-HMC}.
\begin{proposition}
  Define the set $C_{M,x}$ for $x \in \Sigma$ as 
  \begin{align}
    C_{M,x} = \left\{y\in \Sigma\,\middle|\,
    \det\left(\nabla\xi(y)^TM^{-1}\nabla\xi(x)\right)=0\right\}\,.
    \label{set-cm-x}
  \end{align}
  For all $z=(x,p)$, $z'=(x^1,p^{1,-})\in T^*\Sigma$ such that $ z'\in \mathcal{F}(z)$ and $x^1 \in \Sigma\setminus C_{M,x}$, the following properties are satisfied:
  \begin{enumerate}
    \item
There exists a neighborhood $\mathcal{O}\subset T^*\Sigma$ of $z$ and a $C^1$-differentiable function $\Upsilon_1:\mathcal{O}\rightarrow \Sigma$ such that~$x^1=\Upsilon_1(z)$ and $G_{M,\bar{x}}(\Upsilon_1(\bar{z}))=\bar{p}$ for all $\bar{z}=(\bar{x}, \bar{p})\in \mathcal{O}$.
      Moreover, this map~$\Upsilon_1$ is unique in the sense that any $C^1$-differentiable map $\widetilde{\Upsilon}_1: \mathcal{O}'\rightarrow
      \Sigma$, where $\mathcal{O}'\subset T^*\Sigma$ is a neighborhood of~$z$ such that~$x^1=\widetilde{\Upsilon}_1(z)$ and $G_{M,\bar{x}}(\widetilde{\Upsilon}_1(\bar{z}))=\bar{p}$
      for all $\bar{z}=(\bar{x}, \bar{p})\in \mathcal{O}'$, coincides with
      $\Upsilon_1$ on $\mathcal{O}\cap \mathcal{O}'$;
    \item
      Recall the function $F$ in \eqref{F-map-related-to-rattle} and consider the map
      $\Upsilon_1:\mathcal{O}\rightarrow \Sigma$ introduced in the previous
      item. The Lagrange multiplier functions
      \begin{equation}
	\begin{aligned}
	\lambda_x(\bar{z}) =& \frac{1}{\tau}
	\left[(\nabla\xi^TM^{-1}\nabla\xi)(\bar{x})\right]^{-1}\nabla\xi(\bar{x})^T\left[\Upsilon_1(\bar{z}) - \bar{x} - \tau M^{-1}\bar{p} + \frac{\tau^2}{2} M^{-1} \nabla
	\overline{V}(\bar{x})\right]\,,\\
	\lambda_p(\bar{z})=& -\left[(\nabla\xi^T M^{-1}\nabla\xi)(\Upsilon_1(\bar{z}))\right]^{-1}\\
	& \times \left[\nabla\xi(\Upsilon_1(\bar{z}))^TM^{-1}\left(\bar{p}
	-\frac{\tau}{2} \left(\nabla \overline{V}(\bar{x}) + \nabla
	\overline{V}(\Upsilon_1(\bar{z}))\right)+ \nabla\xi(\bar{x}) \lambda_x(\bar{z})\right)\right]\,,  
	\end{aligned}
	\label{multiplier-c-explicit-phase-space}
      \end{equation}
      where $\bar{z}=(\bar{x}, \bar{p})\in \mathcal{O}$, are $C^1$-differentiable functions on $\mathcal{O}$, such that $z' =
      F\left(z, \lambda_x(z), \lambda_p(z)\right)$ and~$F(\bar{z}, \lambda_x(\bar{z}),\lambda_p(\bar{z}))\in\mathcal{F}(\bar{z})$ for all~$\bar{z}\in \mathcal{O}$.
      Furthermore, any $C^1$-differentiable functions
      $\widetilde{\lambda}_x, \widetilde{\lambda}_p:\mathcal{O}'\rightarrow \mathbb{R}^k$, 
      where $\mathcal{O}'$ is a neighborhood of $z$ such that 
      $z' = F\left(z, \widetilde{\lambda}_x(z), \widetilde{\lambda}_p(z)\right)$ and $F(\bar{z},
      \widetilde{\lambda}_x(\bar{z}), \widetilde{\lambda}_p(\bar{z})) \in  \mathcal{F}(\bar{z})$ for all~$\bar{z}\in \mathcal{O}'$,
      coincide with the functions in \eqref{multiplier-c-explicit-phase-space} on $\mathcal{O}\cap \mathcal{O}'$;
    \item
      There exists a neighborhood $\mathcal{O}\subset
      T^*\Sigma$ of $z$ such that the map $\Upsilon:=F\left(\cdot, \lambda_x(\cdot),
      \lambda_p(\cdot)\right): \mathcal{O}\rightarrow
    \Upsilon(\mathcal{O})\subset T^*\Sigma$ (with $F$ defined
      in~\eqref{F-map-related-to-rattle} and the Lagrange multiplier functions $\lambda_x(\cdot), \lambda_p(\cdot)$ given in~\eqref{multiplier-c-explicit-phase-space}) is a $C^1$-diffeomorphism on $\mathcal{O}$ which satisfies $\Upsilon(z)=z'$ and $|\det D\Upsilon|\equiv 1$ on~$\mathcal{O}$.
      Moreover, $\Upsilon(\bar{z})$ is the only element of
      $\mathcal{F}(\bar{z})$ in the neighborhood $\Upsilon(\mathcal{O})$ of $z'$, for all $\bar{z} \in \mathcal{O}$.
  \end{enumerate}
  \label{prop-differentiability-hmc}
\end{proposition}
\begin{proposition}
  Define the set $\mathcal{N}=\left\{(x,p)\in T^*\Sigma\,\middle|\,x\in
  \Sigma, \,p \in G_{M,x}(C_{M,x})\right\}$, where $G_{M,x}$ is the map in
  \eqref{Gx-hmc} and the set $C_{M,x}$ is defined in \eqref{set-cm-x}.
  The following properties are satisfied:
  \begin{enumerate}
    \item
$\mathcal{N}\subset T^*\Sigma$ is a closed set of measure zero, and
      can be defined equivalently as 
      \begin{equation}
	\mathcal{N} = \Big\{(x,p)\in T^*\Sigma\,\Big|\,\exists\, z'=(x^1, p^{1,-})\in \mathcal{F}(x,p), ~\det\left(\nabla\xi(x^1)^TM^{-1}\nabla\xi(x)\right)=0 \Big\};
	\label{set-n-alternative}
      \end{equation} 
    \item 
      For all $z \in T^*\Sigma\setminus \mathcal{N}$, the set $\mathcal{F}(z)$ contains at most a finite number of elements;
    \item 
      For $z \in T^*\Sigma\setminus \mathcal{N}$, let $n=|\mathcal{F}(z)|$ and
      denote by $\mathcal{F}(z) = \{z^{(1)}, z^{(2)}, \ldots, z^{(n)}\}$. When $n
      \ge 1$, there exists a neighborhood~$\mathcal{O}$ of~$z$, as well as~$n$
      pairs of $C^1$-differentiable functions $\lambda_x^{(i)},
      \lambda_p^{(i)}: \mathcal{O}\rightarrow \mathbb{R}^k$,
      which are locally given by \eqref{multiplier-c-explicit-phase-space} in
      Proposition~\ref{prop-differentiability-hmc} with $z'=z^{(i)}$,
      such that $z^{(i)} = F\left(z, \lambda_x^{(i)}(z), \lambda_p^{(i)}(z)\right) \in T^*\Sigma$ for any $1 \leq i \leq n$ and, for all $\bar{z}\in \mathcal{O}$, 
      \[
        \mathcal{F}(\bar{z}) = \left\{ F\left(\bar{z}, \lambda_x^{(1)}(\bar{z}), \lambda_p^{(1)}(\bar{z})\right),~ F\left(\bar{z}, \lambda_x^{(2)}(\bar{z}), \lambda_p^{(2)}(\bar{z})\right),~ \ldots,~ F\left(\bar{z}, \lambda_x^{(n)}(\bar{z}), \lambda_p^{(n)}(\bar{z})\right) \right\}\,.
      \]
      In particular, $|\mathcal{F}(\bar{z})|=n$ for all $\bar{z}\in \mathcal{O}$.
    \item
      $T^*\Sigma$ is the disjoint union of the subsets $\mathcal{N}$, $\mathcal{B}_0$,
      $\mathcal{B}_1, \ldots$, i.e.\
      $\dps T^*\Sigma=\left(\bigcup_{i=0}^{+\infty} \mathcal{B}_i\right)\cup
      \mathcal{N}$, where the sets
      \begin{equation}
	\begin{aligned}
	  \mathcal{B}_0 =& \left\{z\in T^*\Sigma\, \Big| \, |\mathcal{F}(z)| =
	0\right\}\,,\\
	  \mathcal{B}_i =& \left\{z\in T^*\Sigma\setminus
	  \mathcal{N}\, \Big| \, |\mathcal{F}(z)| = i\right\}\,, \quad i=1,2,\ldots\,.
	\end{aligned}
      \end{equation}
      are open subsets of $T^*\Sigma$.   
  \end{enumerate}
  \label{prop-set-f-hmc}
\end{proposition}

Let us point out that there is some type of symmetry in Proposition~\ref{prop-differentiability-hmc}
in the condition on the two states $z=(x,p),\,z'=(x^1,p^{1,-})\in T^*\Sigma$,
in the sense that $z'\in \mathcal{F}(z),\,x^{1}\in \Sigma\setminus C_{M,x}$  
if and only if~$z\in \mathcal{F}(z'),\,x\in \Sigma\setminus C_{M,x^{1}}$,
which can be easily seen from Lemma~\ref{lemma-1} and the definition of the set $C_{M,x}$ in \eqref{set-cm-x}.
Moreover, let $\Upsilon:\mathcal{O}\rightarrow \Upsilon(\mathcal{O})$
with~$\mathcal{O}$ a neighborhood of $z$, 
and $\widetilde{\Upsilon}: \widetilde{\mathcal{O}}\rightarrow
\widetilde{\Upsilon}(\widetilde{\mathcal{O}})$ with~$\widetilde{\mathcal{O}}$ a neighborhood of $z'$ 
be the $C^1$-diffeomorphisms considered in the third item of
Proposition~\ref{prop-differentiability-hmc} for $z,z'$ respectively.
Then, by possibly shrinking the neighborhoods (still denoted by
$\mathcal{O}$ and $\widetilde{\mathcal{O}}$), Lemma~\ref{lemma-1} actually implies that
$\widetilde{\Upsilon}\circ\Upsilon=\mbox{id}$ on~$\mathcal{O}$ and
$\Upsilon\circ\widetilde{\Upsilon}=\mbox{id}$ on~$\widetilde{\mathcal{O}}$.

Let us comment on the differences between Propositions~\ref{prop-differentiability-hmc} and~\ref{prop-set-f-hmc}, with a discussion similar to the one in Remark~\ref{rmk:non-tangential}.
On the one hand, the condition $z=(x,p)\in T^*\Sigma\setminus \mathcal{N}$ in
Proposition~\ref{prop-set-f-hmc} implies $x^1\in \Sigma\setminus C_{M,x}$ for all $z'=(x^1,p^{1,-})\in \mathcal{F}(z)$.
On the other hand, in contrast to Proposition~\ref{prop-set-f-hmc}, Proposition~\ref{prop-differentiability-hmc} also holds for $z\in
\mathcal{N}$, as long as $x^1\in \Sigma\setminus C_{M,x}$.
Again, the set $C_{M,x}$ is related to the non-tangential condition of the previous work~\cite[Definition~2.1]{hmc-submanifold-tony}. We also point out that, since $\mathcal{N}$ has zero Lebesgue measure in $T^*\Sigma$, it has zero probability measure under~$\mu$.

\paragraph{The set-valued proposal map.}
We are now in position to introduce the set-valued proposal map. 
For $z = (x,p) \in T^*\Sigma$, the set-valued
proposal map $\Phi(z)$ is a subset of~$\mathcal{F}(z)$ which contain the projection points obtained by
numerically computing the outcomes of the one-step RATTLE scheme with momentum reversal. In other words,
$\Phi(z)$ contains a state $(x^1,p^{1,-})\in \mathcal{F}(z)$ if $(x^1, p^{1,-})$
is obtained numerically by solving for the Lagrange multipliers $\lambda_x$
and~$\lambda_p$ from~\eqref{rattle-constraint-1}--\eqref{rattle-constraint-2}; i.e.\ $\Phi(z)$ is a solution of ~\eqref{rattle-constraint-1}--\eqref{rattle-constraint-2} starting from~$z$.
Depending on the number of numerical solutions
of~\eqref{rattle-constraint-1}--\eqref{rattle-constraint-2} for a given $z=(x,p)$, the set~$\Phi(z)$ can either be empty or contain multiple states.

Motivated by Propositions~\ref{prop-differentiability-hmc} and~\ref{prop-set-f-hmc}, we make the following assumption on the map $\Phi$. 
   
\begin{assumption}
  The following properties are satisfied by the map $\Phi$: 
  \begin{enumerate}
    \item
      For all $z\in T^*\Sigma$, the set $\Phi(z)$ contains at most a finite number of elements;
    \item
            The set 
\begin{equation}
  \begin{aligned}
    \mathcal{D}_\Phi = \Big\{(z,z')\in T^*\Sigma\times
    T^*\Sigma~\Big|\,z'\in \Phi(z)~\mbox{and}~z\in \Phi(z') \Big\}
  \end{aligned}
  \label{admissible-set-d-r}
\end{equation}
      is a non-empty measurable subset of $T^*\Sigma\times T^*\Sigma$. 
    \item
      For all $z=(x,p)\in T^*\Sigma$ such that $|\Phi(z)|=n\ge 1$, 
      denote by $\Phi(z) = \{z_1, \ldots, z_n\}$ with (upon reordering)
      $\{z'\in\Phi(z)\,|\, \Pi(z')\not\in C_{M,\Pi(z)}\}=\{z_1,\cdots,
      z_m\}$ where $m \le |\Phi(z)|$.
When $m \ge 1$, there exists a neighborhood $\mathcal{O}$ of~$z$, as well as $m$
      different $C^1$-diffeomorphisms $\Phi^{(j)}:\mathcal{O}\rightarrow
      \Phi^{(j)}(\mathcal{O})\subset T^*\Sigma$ (for $1 \le j \le m$) such that
      $z_j=\Phi^{(j)}(z)$ and $\Phi^{(j)}(\bar{z}) \in \Phi(\bar{z})$ for any~$1 \le j \le m$ and~$\bar{z} \in \mathcal{O}$. 
  \end{enumerate}
  \label{assump-phi-phase-space}
\end{assumption}
It is easy to adapt Example~\ref{ex:sphere} to build a map $\Phi$ which satisfies Assumption~\ref{assump-phi-phase-space} on the circle in $\mathbb{R}^2$.
  We refer to Section~\ref{sec:numerical_computation_constraint} for a discussion on set-valued maps satisfying these assumptions.

\subsubsection{Presentation of the algorithm}
\label{sec:presentation_HMC}

 The multiple projection Hybrid Monte Carlo algorithm is given in
Algorithm~\ref{algo-phase-space-mc-rattle} below. It reduces to the algorithm studied in~\cite{hmc-submanifold-tony} when the proposal is a singled-valued
map obtained by computing a single solution of the constraint
equations~\eqref{rattle-constraint-1}--\eqref{rattle-constraint-2} (namely
$|\Phi(z)|=0$ or $1$, where $\Phi$ is the proposal map introduced in the previous section). Notice also that Algorithm~\ref{algo-phase-space-mc-rattle} with $M$ the identity matrix and $\alpha = 0$ in~\eqref{update-momenta} below reduces to Algorithm~\ref{algo-mala-on-sigma}.
\begin{algorithm}[h]
  \caption{Multiple projection HMC on $\Sigma$}
  \label{algo-phase-space-mc-rattle}
  \begin{algorithmic}[1]
    \State Choose $z^{(0)} \in T^*\Sigma$ and $N > 1$. Set $i=0$.
    \While{$i < N$}
    \State Update the momenta of the current state~$z^{(i)}$ according
    to~\eqref{update-momenta}, and  denote the new state by $z^{(i+\frac{1}{4})}=z=(x,p)$.
    \State Compute the set $\Phi(z)$. If $\Phi(z)=\emptyset$, set
    $z^{(i+\frac{2}{4})}=z$ and \Goto{algo-hmc-end-of-while}.
    \State Randomly choose an element~$z' \in \Phi(z)$ with
    probability $\omega(z'\,|\,z)$.
    \State Compute the set $\Phi(z')$. 
    \State 
    Check whether $z\in\Phi(z')$ is true. If this is not true, set $z^{(i+\frac{2}{4})}=z$ and
    \Goto{algo-hmc-end-of-while}.
    \State Draw a uniformly distributed random number $r \in [0,1]$. Set 
    \begin{align*}
      z^{(i+\frac{2}{4})} = 
      \begin{cases}
        z' &\quad \mbox{if}~ r \le a(z'\,|\,z)\,, \\
        z  &\quad \mbox{otherwise}\,,
      \end{cases}
    \end{align*}
    with the acceptance probability $a(z'\,|\,z)$ defined in~\eqref{rate-phase-space-rattle}.
    \State Set $z^{(i+\frac{3}{4})}=\mathcal{R}(z^{(i+\frac{2}{4})})$.\label{algo-hmc-end-of-while}
    \State Update the momenta of the state~$z^{(i+\frac{3}{4})}$ according to~\eqref{update-momenta}, and denote the new state by~$z^{(i+1)}$. 
    \State $i \gets i+1$. 
    \EndWhile
  \end{algorithmic}
\end{algorithm}

Let us now discuss the different algorithmic steps of the multiple projection Hybrid Monte Carlo algorithm.
\paragraph{Momentum update.} Given the current state $z=(x,p) \in T^*\Sigma$, the momentum is updated according to 
\begin{align}
  p \leftarrow \alpha p + \sqrt{\frac{1-\alpha^2}{\beta}} \eta\,,
  \label{update-momenta}
\end{align}
for some parameter $|\alpha| < 1$, where $\eta$ is a Gaussian
random variable in the cotangent space $T^*_x\Sigma$ with identity covariance with respect to the
weighted inner product $\langle \cdot , \cdot \rangle_{M^{-1}}$. In practice,
$\eta$ can be obtained from $\eta = P_M(x)v$, where $v$ is a Gaussian random
variable on $\mathbb{R}^d$ with covariance matrix $M$, or from $\eta =
\sum_{i=1}^{d-k} w_i e_i$, where $(e_1, e_2, \ldots, e_{d-k})$ is a basis of $T^*_x\Sigma$ orthonormal for the scalar product~$\langle\cdot, \cdot\rangle_{M^{-1}}$, and $w_i\in \mathbb{R}$ for $1 \le i \le d-k$ are independent and identically distributed (i.i.d) standard Gaussian random variables on~$\mathbb{R}$. Note also that the update rule~\eqref{update-momenta} could be generalized to matrix-valued~$\alpha$. Let us comment that the momentum update is applied both at the beginning and at the end of one iteration of the multiple projection Hybrid Monte Carlo algorithm (see Algorithm~\ref{algo-phase-space-mc-rattle}). 
We refer to the discussion in Remark~\ref{rmk:rev} for a variant with only one stochastic update on momenta per iteration.

\paragraph{Choice of proposal.} 

Remember the set-valued map $\Phi$ introduced in the previous section. Once $\Phi(z)$ is available and non-empty (namely $|\Phi(z)| \geq 1$), one randomly
chooses an element $z'=(x^1,p^{1,-}) \in \Phi(z)$ with probability
$\omega(z'\,|\,z)$, and takes 
it as the proposal state. As in Section~\ref{sec:presentation_MALA}, the
probabilities $\omega(\cdot\,|\,z)$ are chosen in $(0,1]$ such that 
$\omega(\cdot\,|\,\cdot)$ is measurable and satisfies $\sum_{z' \in \Phi(z)} \omega(z'\, | \, z)=1$ for all $z \in T^*\Sigma$. They can be
chosen uniformly ({\em i.e.} $\omega(z'\, | \, z)=1/|\Phi(z)|$), or in a way such that it is more likely to jump to
states $x^1$ that are close to, or on the contrary far from, the current state $x$.

\paragraph{Reversibility check.}
Similarly to Algorithm~\ref{algo-mala-on-sigma}, a reversibility check step is
in general needed in the MCMC algorithm~(see e.g.\ \cite{hmc-submanifold-tony}). Specifically, after the proposal state
$(x^1, p^{1,-})$ is chosen, one needs to verify that $(x, p)\in \Phi(x^1, p^{1,-})$.
For this purpose, one computes $\Phi(x^1, p^{1,-})$ by solving the constraint
equations \eqref{rattle-constraint-1}--\eqref{rattle-constraint-2} and checks
whether $(x,p)$ belongs to $\Phi(x^1, p^{1,-})$. Note that the second claim in
Lemma~\ref{lemma-1} implies that it is in fact sufficient for the
reversibility check step to verify that the positions are the same (as already
pointed out in Section~3.1 of \cite{hmc-submanifold-tony}).

Lemma~\ref{lemma-1} implies that the one-step RATTLE scheme with
momentum reversal is indeed able to map the state $(x^{1}, p^{1,-})$ to $(x,p)$
with the pair of Lagrange multipliers $(\lambda_p, \lambda_x)$, i.e.\
$(x,p)\in \mathcal{F}(x^1,p^{1,-})$. This means that  
$(x,p)\in\Phi(x^1,p^{1,-})$ as long as the pair $(\lambda_p, \lambda_x)$ is
indeed found as a solution of the constraint
equations~\eqref{rattle-constraint-1}--\eqref{rattle-constraint-2} when  $\Phi(x^1, p^{1,-})$ is computed numerically. This leads the following remark, similar to Remark~\ref{rmk-Reversibility-check-mala} in Section~\ref{sec:presentation_MALA}.

\begin{remark}
  In the case when the numerical
  solver is able to find all the solutions to \eqref{rattle-constraint-1}--\eqref{rattle-constraint-2}, the ``reversibility check'' performed in Step $7$ of
  Algorithm~\ref{algo-phase-space-mc-rattle} is actually not needed. More generally, the reversibility check step will not lead to frequent rejections when one is able to compute almost all the solutions of \eqref{rattle-constraint-1}--\eqref{rattle-constraint-2}. This is an advantage compared to the algorithms in~\cite{goodman-submanifold} and~\cite{hmc-submanifold-tony}, since it leads to a larger acceptance rate.
  \label{rmk-Reversibility-check-hmc}
\end{remark}

\paragraph{Metropolis procedure.}
When the reversibility check step is successful, i.e.\ $z=(x, p) \in \Phi(x^1, p^{1,-})$,
a Metropolis-Hastings step is performed to decide either to accept or to
reject the proposed move to the state $z' = (x^1,p^{1,-})$, based on the acceptance probability 
\begin{align}
  a(z'\,|\,z) = \min\left\{1,\frac{\omega(z\,|\,z')}{\omega(z'\,|\,
  z)}\,\mathrm{e}^{-\beta \left(H(z')-H(z) \right)}\right\}\,. 
  \label{rate-phase-space-rattle}
\end{align}
In particular, when the elements $z'$ are drawn with uniform probabilities, i.e.\
for all $z \in T^*\Sigma$, $\omega(z'\,|\,z) = \frac{1}{|\Phi(z)|}$ for all $z' \in \Phi(z)$,
\eqref{rate-phase-space-rattle} simplifies as 
\begin{align*}
  a(z'\,|\,z) = \min\left\{1,~
  \frac{|\Phi(z)|}{|\Phi(z')|}\,\mathrm{e}^{-\beta \left(H(z') - H(z)\right)}\right\} \,. 
\end{align*}

\subsubsection{Consistency of the algorithm}
\label{sec:consistency_HMC}
Recall that $\mathcal{R}(x,p)=(x,-p)$ is the momentum reversal map defined in \eqref{involution-r}, and that $\mu$ is invariant by~$\mathcal{R}$.
Let us introduce the following definition of reversibility up to momentum
reversal. It is equivalent to the ``modified detailed balance'' considered
in~\cite{compressible-hmc} and~\cite{gardiner2004handbook} for example. 

\begin{definition}
  A Markov chain on $T^*\Sigma$ with transition probability kernel $q(z,dz')$ is reversible up to momentum reversal with
  respect to the probability distribution $\mu$ if and only if the following
  equality holds: For any bounded measurable function $f: T^*\Sigma \times T^*\Sigma \rightarrow \mathbb{R}$,
\begin{align}
  \int_{T^*\Sigma\times T^*\Sigma} f(z,z')\, q(z,dz')\,\mu(dz) =
  \int_{T^*\Sigma\times T^*\Sigma} f(\mathcal{R}(z'),\mathcal{R}(z))\,
  q(z,dz')\,\mu(dz) \,.
  \label{integral-detailed-balance-with-momentum-reversal}
\end{align}
  \label{def-rmr}
\end{definition}

Similarly to Definition~\ref{def-reversible}, one can verify that a Markov chain $(z^{(i)})_{i\ge 0}$ on $T^*\Sigma$ is reversible up to momentum reversal with respect to $\mu$ if and only if, for any
integer $n>0$, the law of the sequence $(z^{(0)},\, z^{(1)},\, \ldots, z^{(n)})$ is the same as the law of the sequence $\left(\mathcal{R}(z^{(n)}),\, \mathcal{R}(z^{(n-1)}),\, \ldots, \, \mathcal{R}(z^{(0)})\right)$ when $z^{(0)}$ is distributed according to~$\mu$. 
In particular, by considering $f(z,z')=f(z')$ in
\eqref{integral-detailed-balance-with-momentum-reversal} and using the fact
that $\mu$ is invariant under $\mathcal{R}$, $\mu$ is an invariant measure if
the Markov chain is reversible up to momentum reversal with respect to $\mu$.
A simple connection between the reversibility up to momentum reversal of
Definition~\ref{def-rmr} and the reversibility of Definition~\ref{def-reversible}
is the following: 
\begin{lemma}
  Consider a Markov chain $\mathscr{C}$ on $T^*\Sigma$ with transition probability
  kernel $q(z,dz')$ and let~$\widetilde{\mathscr{C}}$ be the Markov chain on
  $T^*\Sigma$ obtained by composing a transition according to $q$ with a
  momentum reversal: starting from $z$, the new state is $\mathcal{R}(z')$ where $z'\sim q(z,dz')$.
  Then, $\mathscr{C}$ is reversible with respect to $\mu$ if and only if
  $\widetilde{\mathscr{C}}$ is reversible up to momentum reversal with respect to $\mu$.
  \label{lemma-connection-reversibility-and-rmr}
\end{lemma}

We have the following result on the Markov chain generated by Algorithm~\ref{algo-phase-space-mc-rattle}.
\begin{theorem}
  Under Assumptions~\ref{assump-xi} and \ref{assump-phi-phase-space},
  Algorithm~\ref{algo-phase-space-mc-rattle} generates a Markov chain on
  $T^*\Sigma$ that is reversible up to momentum reversal with respect to
  $\mu$. In particular, it preserves the probability measure~$\mu$ in \eqref{target-measure-in-phase-space}.
  \label{thm-algo-2-rattle}
\end{theorem}
The proofs of Lemma~\ref{lemma-connection-reversibility-and-rmr} and Theorem~\ref{thm-algo-2-rattle} are provided in Section~\ref{sec-proofs-HMC}.

\begin{remark}
  \label{rmk:rev}
The proof of Theorem~\ref{thm-algo-2-rattle} consists in checking that the probability measure $\mu$ is preserved by 
  the following transitions of Algorithm~\ref{algo-phase-space-mc-rattle}: $x^{(i)}\rightarrow x^{(i+\frac{1}{4})}$ (Step 3: momentum refreshment), $x^{(i+\frac{1}{4})}\rightarrow
x^{(i+\frac{2}{4})}$ (Steps 4-8: metropolized RATTLE with momentum reversal),
$x^{(i+\frac{2}{4})}\rightarrow x^{(i+\frac{3}{4})}$ (Step~9: momentum
  reversal), and $x^{(i+\frac{3}{4})}\rightarrow x^{(i+1)}$ (Step 10: momentum
  refreshment, which is exactly the same as Step 3).
  Thus an algorithm which combines any the three above steps 
  (momentum refreshment, metropolized RATTLE with momentum reversal and
  momentum reversal) preserves the measure $\mu$.
  By combining these three steps as in
  Algorithm~\ref{algo-phase-space-mc-rattle}, 
  we obtain a Markov chain which does not only leave~$\mu$ invariant but is
  actually reversible up to momentum reversal with respect to $\mu$. Finally, notice that combining the last Step 10 and the first Step 3 of Algorithm~\ref{algo-phase-space-mc-rattle} is actually equivalent to performing a single momentum update with a parameter~$\alpha^2$. 
  \label{rmk-different-steps-of-algo}
\end{remark}

As for Theorem~\ref{thm-mala-on-sigma}, one still needs to verify irreducibility in order to prove the ergodicity of the Markov chain generated by Algorithm~\ref{algo-phase-space-mc-rattle}. We also refer here to~\cite{M2AN_2007__41_2_351_0,convergence-hmc,livingstone2019,Schtte1999ConformationalDM} for related results in the unconstrained
case and to~\cite{Hartmann2008} for results involving constraints.

Before concluding this section, let us compare Algorithms~\ref{algo-mala-on-sigma} and~\ref{algo-phase-space-mc-rattle}. As already mentioned,
when $M=I_{d}$ and momenta are fully resampled, i.e.\ $\alpha=0$ in \eqref{update-momenta},
Algorithm~\ref{algo-phase-space-mc-rattle} is exactly the same as
Algorithm~\ref{algo-mala-on-sigma}, when considering only the position variable on $\Sigma$.
When $\alpha \neq 0$, however, the chain in the position variable obtained from Algorithm~\ref{algo-phase-space-mc-rattle} is not Markovian since momenta are only partially refreshed. This may be useful to prevent the system going back towards the previous position. Let us mention incidentally that it would be interesting to derive quantitative results supporting this reasoning.

Concerning the proofs of reversibility (see Section~\ref{sec-proofs}), the draw of the velocity and the update of position are considered as a whole in our analysis of Algorithm~\ref{algo-mala-on-sigma}, following~\cite{goodman-submanifold}.  Algorithm~\ref{algo-phase-space-mc-rattle} is analyzed differently, as the composition of separate steps which leave the measure $\mu$ invariant, namely the resampling of momentum, the momentum reversal and finally
the metropolization of the update by the deterministic (set-valued) map $\Phi$, in the spirit of~\cite{hmc-submanifold-tony}. These two approaches give different insights on the reversibility properties of these algorithms.

\subsection{Numerical computation of the projections to $\Sigma$}
\label{sec:numerical_computation_constraint}

In this short section, we explain how set-valued proposal maps 
can be obtained numerically for different types of maps
$\xi:\mathbb{R}^d\rightarrow \mathbb{R}^k$, and then we discuss how
Assumptions~\ref{assump-psi-on-state-space} and \ref{assump-phi-phase-space}
are satisfied in practice. We distinguish between two situations, depending on whether all solutions of the constraint equations~\eqref{psi-by-nonlinear-constraint-eqn} and~\eqref{rattle-constraint-1} are guaranteed to be found or not. It is  obvious that the numerical solvers discussed below find (at most) a finite number of solutions, so that the first item of Assumption~\ref{assump-psi-on-state-space} (resp. Assumption~\ref{assump-phi-phase-space}) is satisfied in practice.
Moreover, the numerical projection procedures we present below yield projection maps
such that, for Algorithm~\ref{algo-mala-on-sigma}, $\mbox{Im} \Psi_x
\setminus C_x$ is non-empty for all $x \in \Sigma$ (resp.\ for Algorithm~\ref{algo-phase-space-mc-rattle}, 
 $\{p \in T^*_x \Sigma | \exists z' \in \Phi(x,p), \Pi(z') \not \in C_{M,x}\}$
is non-empty for all $x \in \Sigma$), so that Algorithm~\ref{algo-mala-on-sigma} (resp.\ \ref{algo-phase-space-mc-rattle}) is non-trivial.

\paragraph{Case 1: All solutions of the constraint equation are guaranteed to be found.}
In some specific situations, all the solutions to the
equations~\eqref{psi-by-nonlinear-constraint-eqn}
and~\eqref{rattle-constraint-1} can be analytically computed. In other situations, for instance when $\xi$ is a scalar valued polynomial, all the solutions are guaranteed to be found up to an arbitrary small numerical precision. Indeed, in this case, numerical algorithms can be used to compute all the roots of a univariate
polynomial, such as the Julia package PolynomialRoots.jl~\cite{polynomialroots}. 

Let us discuss in this case the assumptions in the second item of  Assumptions~\ref{assump-psi-on-state-space} and~\ref{assump-phi-phase-space}. The analysis in Section~\ref{sec:construction_set_valued_MALA} (see in particular Proposition~\ref{prop-differentiability-mala}) implies that if one defines $\Psi_x(v)=\mathcal F_x(v) \setminus C_x$ (where~$C_x$ is defined by \eqref{c-x-in-prop1}), then $\mathcal{D}=\{(x,y) \in \Sigma\times \Sigma \,|\,  \det\left(\nabla\xi(y)^T \nabla\xi(x)\right)\neq0\}$.
 Similarly, for HMC, the function $\Phi$ can be defined as $\Phi(z)=\mathcal F(z) \setminus C_{M,\Pi(z)}$ (where 
 the set $C_{M,x}$ is defined by \eqref{set-cm-x} and
 $\Pi$ is defined by~\eqref{projection-map-pi}), and then, the set $\mathcal{D}_\Phi$ in Assumption~\ref{assump-phi-phase-space} is:
\begin{align*}
  \mathcal{D}_\Phi = \Big\{(z,z')\in T^*\Sigma\times
  T^*\Sigma \,\Big|\, &\exists\, x,x'\in \Sigma\times \Sigma\,,\, \det\left(\nabla\xi(x')^T M^{-1} \nabla\xi(x)\right)\neq0 ,\,\\
 & z=\big(x,G_{M,x}(x')\big),\, z'=\left(x', G_{M,x'}(x)\right) 
  \Big\}\,,
\end{align*}
where $G_{M,x}$ is the map defined in \eqref{Gx-hmc}. From the above definition, both sets $\mathcal{D}$ and $\mathcal{D}_\Phi$ are non-empty and measurable for any timestep $\tau>0$. The regularity properties in the third items of Assumptions~\ref{assump-psi-on-state-space} and~\ref{assump-phi-phase-space} are then given by Propositions~\ref{prop-differentiability-mala} and~\ref{prop-differentiability-hmc}.

\paragraph{Case 2: All solutions of the constraint equation cannot be guaranteed to be found.}
When~$\xi$ is either a multidimensional polynomial, or a nonlinear function, there are in general no numerical methods for computing all solutions of the constraint equations. The numerical methods under consideration can sometimes find all solutions, but sometimes only a subset of them. 

When the map $\xi:\mathbb{R}^d\rightarrow \mathbb{R}^k$ is defined by
polynomials, the constraint equations~\eqref{psi-by-nonlinear-constraint-eqn}
and \eqref{rattle-constraint-1} are polynomial systems. Finding numerical solutions of polynomial systems is in fact a well
studied topic in the field of numerical algebraic geometry~\cite{Sommese2005},
where the primary computational approach is the homotopy continuation method.
In particular, there are theoretical results in algebraic geometry, e.g.\ Bertini's theorem~\cite{hartshorne2010algebraic}, which
guarantee that the homotopy continuation method is able to compute all
solutions of polynomial systems in principle -- although the actual
computation of all solutions may not be achieved in practice due to various
numerical and implementation issues. The homotopy continuation method has been
implemented in several numerical packages, such as
Bertini~\cite{BHSW06} and HomotopyContinuation.jl~\cite{homotopy-continuation-julia}. They can be used to solve multiple solutions of polynomial systems.

For generic nonlinear constraints, Newton's method is commonly used to solve the constraint equations, as done
in~\cite{leimkuhler-reich-04,Tony-constrained-langevin2012,goodman-submanifold,hmc-submanifold-tony}.
It finds a solution provided that the initial
guess is not too far away from some solution. In general, it has a fast
(local) convergence rate and is easy to implement. 

\begin{remark}
In our analysis, we assume that the set-valued proposal maps are deterministic functions of the current states. For example, they can be built using Newton's method with either various fixed initial guesses for the Lagrange multipliers or various initial guesses which depend only on the current state (positions and momenta) in a deterministic way. Actually, these initial guesses could vary from one call to the projection procedure to find the Lagrange multipliers to another. For example, they could be drawn randomly at each call of the projection procedure, independently from all the other random variables needed in the algorithm. \label{rmk-initial-guesses}\end{remark}

Let us finally discuss the second item of
Assumptions~\ref{assump-psi-on-state-space} and~\ref{assump-phi-phase-space}
for the set-valued proposal maps obtained by the methods discussed here. It is
expected that the set $\mathcal{D}$ in
Assumption~\ref{assump-psi-on-state-space} (resp. the set $\mathcal{D}_\Phi$
in Assumption~\ref{assump-phi-phase-space}) will be non-empty provided that
the timestep $\tau$ is not too large, in particular when Newton's method is
used (see~Lemma~2.8 of \cite{hmc-submanifold-tony} for a related result). 
The measurability of the sets $\mathcal{D}$ and $\mathcal{D}_\Phi$ is 
guaranteed by the expected continuous behavior of numerical solvers (see the discussion below).

The regularity properties in the third items of
Assumptions~\ref{assump-psi-on-state-space} and~\ref{assump-phi-phase-space}
need to be checked for the specific numerical method at hand; see
Section~2.3.3 of~\cite{hmc-submanifold-tony} for the Newton scheme.
More precisely, for Assumption~\ref{assump-psi-on-state-space}, Proposition~\ref{prop-differentiability-mala}
  implies that one can choose $\mathcal{O}=\cap_{j=1}^{m} G_x(\mathcal{Q}_j)$, where $m=|\Psi_x(v)\setminus C_x|$, and  
  define $\Psi_x^{(j)}(\bar{v})= (G_x|_{\mathcal{Q}_j})^{-1}(\bar{v})$ for $\bar{v}\in
  \mathcal{O}$, where $\mathcal{Q}_j$ is the neighborhood of $y_j$ such that
  $G_x|_{\mathcal{Q}_j}$ is a $C^1$-diffeomorphism.
Therefore, the third item of Assumption~\ref{assump-psi-on-state-space} will be satisfied as long as the solutions 
  $\Psi_x^{(j)}(\bar{v})$ can be numerically computed for~$\bar{v}$ belonging to a neighborhood of $v$. For a numerically computed element~$y_j \in \Psi_x(v)$ such that the non-tangential
  condition in Definition~2.1 of~\cite{hmc-submanifold-tony} is satisfied
 (i.e.\ $y_j\in \Psi_x(v)\setminus C_x$), the numerical solvers typically
  behave continuously for~$\bar{v}$ belonging to a neighborhood of $v$
  (see Section~2.3.3 of \cite{hmc-submanifold-tony} for a discussion concerning Newton's method), and thus, it is expected that $\Psi_x^{(j)}(\bar{v})$ is
  indeed in the list of the numerical solutions for~$\bar{v}$ belonging to a neighborhood of $v$. Likewise, for Assumption~\ref{assump-phi-phase-space},
 since $\Pi(z_j)\not\in C_{M,\Pi(z)}$ for~$1 \le j
  \le m$, Proposition~\ref{prop-differentiability-hmc}
  implies that one can choose $\Phi^{(j)}(\bar{z})= \Upsilon^{(j)}(\bar{z})$ for $\bar{z}\in
  \mathcal{O}=\cap_{j=1}^{m} \mathcal{O}_j$,
  where $\mathcal{O}_j$ is the neighborhood of $z$ and $\Upsilon^{(j)}:
  \mathcal{O}_j\rightarrow \Upsilon^{(j)}(\mathcal{O}_j)\subset T^*\Sigma$ is the
  $C^1$-diffeomorphism considered in the third item of Proposition~\ref{prop-differentiability-hmc} for~$z'=z_j$.
Therefore, the third item of Assumption~\ref{assump-phi-phase-space} is satisfied as long as the solutions 
  $\Phi^{(j)}(\bar{z})$ can be numerically computed for $\bar{z}$ belonging to a neighborhood of $z$, which is again expected to be the case for many numerical solvers.
  

\section{Numerical examples}
\label{sec-example}

We illustrate the algorithms introduced in Section~\ref{sec-two-algorithms} to
sample probabilities measures on two submanifolds: a two-dimensional torus in
a three-dimensional space (see Section~\ref{subsec-ex1}), already considered
in~\cite{goodman-submanifold,hmc-submanifold-tony}; and disconnected
components which are subsets of a nine-dimensional sphere (see Section~\ref{subsec-example-2}). We use Algorithm~\ref{algo-phase-space-mc-rattle} with $M=I_d$ and varying $\tau$ and $\alpha$. Remember that when $\alpha=0$, since the mass is the identity, Algorithm~\ref{algo-phase-space-mc-rattle} is equivalent to  Algorithm~\ref{algo-mala-on-sigma}.
All simulations were done using the Julia  programming language (see \url{https://github.com/zwpku/Constrained-HMC} for the codes) and
carried out on the same desktop computer which has 8 CPUs (Intel Xeon E3-1245), 8G memory, and operating system Debian 10.

\subsection{Two-dimensional torus in a three-dimensional space}
\label{subsec-ex1}

In this example, we apply Algorithm~\ref{algo-phase-space-mc-rattle} to study
a two-dimensional torus in $\mathbb{R}^3$ (see
Figure~\ref{fig-ex1-torus-potential}, left). More precisely, the torus is a
two-dimensional submanifold of $\mathbb{R}^3$ defined as the solution in
$(x_1,x_2,x_3)$ of the equation
\begin{align}
  \left(R-\sqrt{x_1^2+x_2^2}\right)^2 + x_3^2 - r^2 = 0 \,, \quad x = (x_1, x_2, x_3)^T\in \mathbb{R}^3\,,
  \label{ex1-torus-equation-with-sqrt}
\end{align} 
where $0 < r < R$. 
After simple algebraic calculations, we see that \eqref{ex1-torus-equation-with-sqrt} is equivalent to the polynomial equation
\begin{align}
  \xi(x) = 0\,,\quad \mbox{where}~~\xi(x) = \left(R^2-r^2 + x_1^2+x_2^2+x_3^2\right)^2 - 4R^2\left(x_1^2 + x_2^2\right)\,.
  \label{ex1-torus-xi}
\end{align}
Therefore, the torus is the zero level set $\Sigma$ of the polynomial $\xi$ as defined in \eqref{ex1-torus-xi}.
In the following, we will use the following parametrization of the torus: 
\begin{align}
  x_1 = (R+r\cos \phi) \cos\theta, \quad
  x_2 = (R+r\cos \phi) \sin\theta, \quad
  x_3 = r \sin\phi,
  \label{ex1-polar}
\end{align}
where $(\phi, \theta) \in [0, 2\pi)^2$. In particular, it
can be verified that the normalized surface measure of the torus in the variables $\phi,\theta$ (still denoted by~$\sigma_\Sigma$ with some abuse of notation) is given by
\begin{align}
  \sigma_\Sigma(d\phi\,d\theta) = \frac{1}{(2\pi)^2} \left(1 + \frac{r}{R}\cos\phi\right)\,d\phi\,d\theta\,.
  \label{ex1-sigma}
\end{align}

To apply Algorithm~\ref{algo-phase-space-mc-rattle}, the constraint equation
\eqref{rattle-constraint-1} needs to be solved to compute the set~$\Phi(x,p)$.
Once \eqref{rattle-constraint-1} is solved, the solution to the constraint equation \eqref{rattle-constraint-2} for momenta can be computed directly from \eqref{lambda-p-exp}. 
Since $\xi$ defined by \eqref{ex1-torus-xi} is a fourth order polynomial, 
\eqref{rattle-constraint-1} has at most four real solutions. 
In this experiment, besides Newton's method that provides at most one solution of~\eqref{rattle-constraint-1} for a given initial guess,
we also apply the Julia packages PolynomialRoots.jl~\cite{polynomialroots} and
HomotopyContinuation.jl~\cite{homotopy-continuation-julia}, which can compute multiple
solutions of \eqref{rattle-constraint-1}. 
Since each iteration of Algorithm~\ref{algo-phase-space-mc-rattle} preserves
the target distribution regardless of the choice of the numerical solver, we can
also use different solvers to compute $\Phi(x,p)$ in different MCMC iterations. 
By exploiting the freedom provided by Algorithm~\ref{algo-phase-space-mc-rattle},
we obtain the following schemes that we will use in the subsequent numerical tests (see also the summary in Table~\ref{tab-schemes}): 
\begin{itemize}
  \item
  ``Newton'', where Newton's method is used to compute the set $\Phi(x,p)$ at each MCMC iteration. 
  \item
   ``PR'' (resp. ``Hom''), where the PolynomialRoots.jl (resp. the HomotopyContinuation.jl) Julia
   package is used at each MCMC iteration to compute the set $\Phi(x,p)$
   and one state is chosen randomly with uniform probability. 
 \item
   ``PR-far'' (resp. ``Hom-far''), where the PolynomialRoots.jl (resp. the
    HomotopyContinuation.jl) Julia package is used at each MCMC iteration to compute the set $\Phi(x,p)$. The states in $\Phi(x,p)$ are sorted in ascending order based on the Euclidean distances between their position components and the current position $x$, and one state is randomly chosen according to the probability distributions in Table~\ref{ex1-pj-nonuniform}.
 \item
   ``PR50-far'' (resp. ``Hom50-far''), where the set $\Phi(x,p)$ is computed using 
PolynomialRoots.jl (resp. HomotopyContinuation.jl) Julia package every $50$ MCMC iterations, while
    Newton's method is used in the other iterations. As in the previous item, when multiple states in~$\Phi(x,p)$ are obtained, 
    they are sorted in ascending order based on the Euclidean distances between their position components to the current
    position~$x$, and one state is chosen
   randomly according to the probability distributions in Table~\ref{ex1-pj-nonuniform}.
\end{itemize}
Note that in the schemes ``$\ast$-far'', the probability distributions in Table~\ref{ex1-pj-nonuniform} slightly favor states that are at larger distances from the current state.
As hybrid schemes, we expect ``PR50-far'' and ``Hom50-far'' to have a better sampling
efficiency to explore the space than "Newton", together with a lower computational cost than ``PR"(``Hom'') or ``PR-far'' (``Hom-far'').
\begin{table}[htp]
\centering
  \begin{tabular}{c|c|c|c}
    \hline
    Scheme & Solver & Period & Distribution $\omega(\cdot\,|\,\cdot)$ \\
    \hline
    Newton & Newton& N/A & N/A\\
    PR & PolynomialRoots & $1$ & uniform\\
    PR-far & PolynomialRoots & $1$ & non-uniform  \\
    PR50-far & Newton + PolynomialRoots & $50$ & non-uniform \\
    Hom & HomotopyContinuation & $1$ & uniform\\
    Hom-far & HomotopyContinuation & $1$ & non-uniform  \\
    Hom50-far & Newton + HomotopyContinuation & $50$ & non-uniform \\
    \hline
  \end{tabular}
  \caption{Sampling schemes on the torus based on Algorithm~\ref{algo-phase-space-mc-rattle}. 
 For each scheme, the column ``Solver'' shows the numerical solvers used for solving the constraint equation \eqref{rattle-constraint-1}. The column ``Period'' shows how often multiple solutions of the constraint equation are computed, using either the PolynomialRoots.jl or HomotopyContinuation.jl Julia packages. When multiple solutions (states) are found, one state is 
  chosen randomly according to either a uniform distribution (``uniform'') or
  to the probability distributions in Table~\ref{ex1-pj-nonuniform} (``non-uniform'').
    \label{tab-schemes}}
\end{table}

We fix $R=1.0$ and $r=0.5$ in
\eqref{ex1-torus-equation-with-sqrt}--\eqref{ex1-sigma}. In each test, we run
$N=10^7$ MCMC iterations using
Algorithm~\ref{algo-phase-space-mc-rattle}, with either $\alpha=0$ or $\alpha=0.7$  in \eqref{update-momenta}
for updating the momenta $p$. The initial position $x$ is fixed to be $(R-r, 0, 0)^T$, while the initial momentum $p$ is drawn randomly according to~\eqref{eq:kappa}.
When Newton's method is used to compute $\Phi(x,p)$ given the current state $(x,p)$
(by solving the Lagrange multiplier $\lambda_x$ from the constraint equation \eqref{rattle-constraint-1}), 
we set $\lambda_x$ to zero initially and the convergence criterion is based on
the Euclidean norm of~$\xi$ being sufficiently small,\footnote{The convergence
criterion we consider is sufficiently tight for our test example in order to remove the
observed bias on the invariant measures due to non-reversibility. However, one
could of course further refine the convergence up to machine precision by
performing a few extra Newton iterations. Alternatively, there are ways to
assess the convergence of fixed point methods without introducing a
convergence threshold, for example by monitoring the norm of the differences
between two successive iterates as the algorithm proceeds, see Section~3 of~\cite{round-off-error-rk-method}.} namely $|\xi(x)|<10^{-8}$. At most~$10$ Newton's iterations are performed to solve the constraint equation \eqref{rattle-constraint-1}. As explained above, when multiple proposal states are found (by the PolynomialRoots.jl or HomotopyContinuation.jl Julia packages), we randomly choose one state either with
uniform probability or according to the probability distributions shown in
Table~\ref{ex1-pj-nonuniform}, based on the Euclidean distances of their position components to the current position $x$. In the reversibility check step of Algorithm~\ref{algo-phase-space-mc-rattle}, two states are considered the same if their Euclidean distance in the position component is less than~$10^{-6}$.

\begin{remark}
Let us mention that it is important to choose the tolerance for the reversibility check larger than the tolerance for the convergence of the Newton algorithm. When the tolerance for the reversibility check is too small, it may happen that the two states are in fact the same, but, due to the incomplete convergence of the Newton step and/or numerical round-off errors, the states are considered to be different.
\end{remark}

\begin{figure}[t!]
\includegraphics[width=0.9\textwidth]{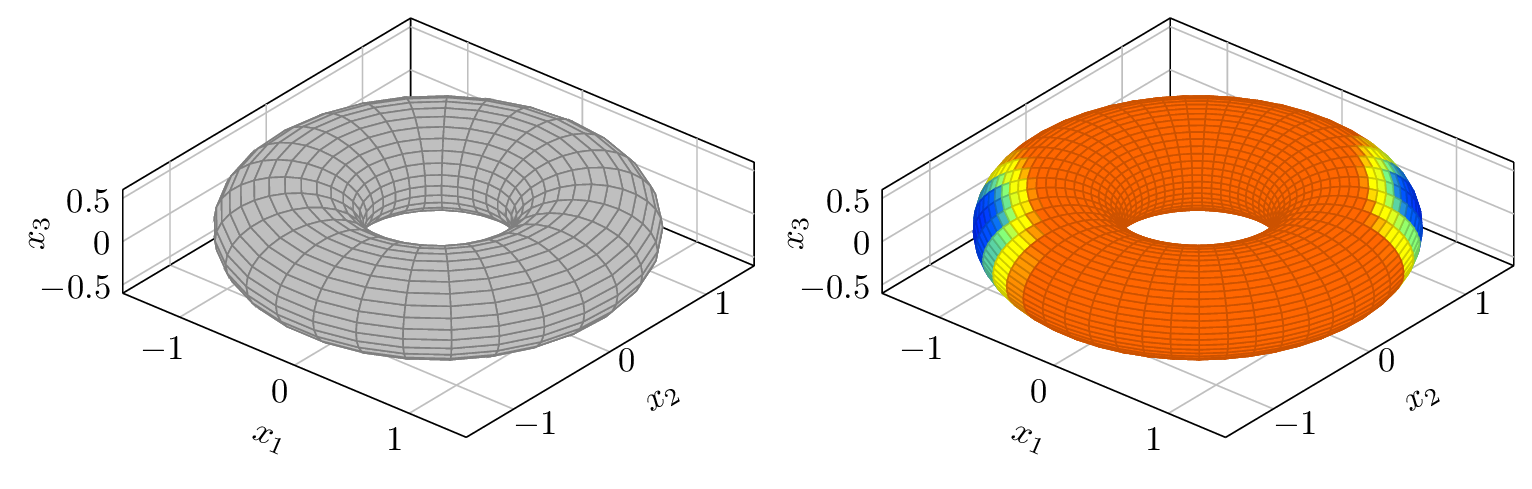}
\centering
  \caption{Left: Torus in $\mathbb{R}^3$. Right: Potential $V(x)$ in
  \eqref{ex1-mu-pot} on the torus (blue color corresponds to low values of
  potential). The function $V$ has two local minima at $((R+r)/\sqrt{2},
  (R+r)/\sqrt{2}, 0)^T$ and $(-(R+r)/\sqrt{2}, -(R+r)/\sqrt{2}, 0)^T$.
  \label{fig-ex1-torus-potential}}
\end{figure}
\begin{table}[t!]
  \center
  \begin{tabular}{c|cccc}
    $n$ & $\omega(z_1\,|\,z)$ & $\omega(z_2\,|\,z)$ & $\omega(z_3\,|\,z)$ & $\omega(z_4\,|\,z)$ \\
    \hline
  $1$ & $1.0$ & & & \\
  $2$ & $0.4$ & $0.6$ & & \\
  $3$ & $0.2$ & $0.4$ & $0.4$ & \\
  $4$ & $0.2$ & $0.3$ & $0.3$ & $0.2$ \\
  \end{tabular}
  \caption{Non-uniform probability distributions $\omega$. At each $z=(x,p)$,
  assuming that $\Phi(z)=\{z_1, \ldots, z_n\}$, $1 \le n \le 4$, where the
  states $z_j$ are sorted in ascending order according to the distances of
  their position components to~the current position $x$, one state~$z_j$ is
  chosen randomly according to the probability distribution $\omega(\cdot\,|\,z)$ on the corresponding row. \label{ex1-pj-nonuniform}}
\end{table}

\paragraph{Uniform law on the torus.}
We start by considering the sampling of the distribution $\sigma_\Sigma$ in~\eqref{ex1-sigma}, i.e.~$V=0$. The aim is to validate Algorithm~\ref{algo-phase-space-mc-rattle} and to investigate the performance of the different sampling schemes in Table~\ref{tab-schemes}.
In each case, we choose the step-size $\tau = 0.8$ in the one-step RATTLE with
momentum reversal scheme~\eqref{psi-map-by-rattle}.
The results are summarized  in Tables~\ref{percentage-of-diff-strategy}--\ref{tab-ex1-summary} and
Figure~\ref{fig-ex1-angle-distribution}.
We see from Figure~\ref{fig-ex1-angle-distribution} that all schemes in
Table~\ref{tab-schemes} are able to correctly sample the
distribution~\eqref{ex1-sigma} of the angles $\phi, \theta$ with $N=10^7$
samples (results with $\alpha=0.7$ are very similar and therefore are not shown).  
As expected, Table~\ref{percentage-of-diff-strategy} confirms that sampling
schemes using either PolynomialRoots.jl or HomotopyContinuation.jl Julia packages 
can find multiple solutions of the constraint equation. 
Note that the empirical probability to find more than one states 
is smaller for ``PR50-far'' and ``Hom50-far'' simply because multiple
solutions are sought only every $50$ MCMC iterations.
In Table~\ref{tab-ex1-summary}, 
we show the average jump distance when the system jumps to a new state,
i.e.\ the average value of $|x^{(i+1)}-x^{(i)}|$ among the MCMC iterations such that $x^{(i+1)} \neq x^{(i)}$.
As shown in the column ``Distance'' in Table~\ref{tab-ex1-summary}, 
using multiple proposal states in $\Phi(x,p)$ leads to larger jump distances on average for the Markov chains.
At the same time, we observe that the PolynomialRoots.jl package is slightly
faster than the HomotopyContinuation.jl Julia package, but both packages are
computationally more expensive than Newton's method.
The two (hybrid) schemes ``PR50-far'' and ``Hom50-far'', which use the
PolynomialRoots.jl and HomotopyContinuation.jl packages every $50$ iterations, 
require run-times ($1.8\times 10^{3}$ seconds) similar to the ones for 
``Newton'' ($1.6\times 10^{3}$ seconds), but at the same time also allow the Markov chains to make large (non-local) jumps when multiple solutions are computed.
From Table~\ref{percentage-of-diff-strategy} and Table~\ref{tab-ex1-summary} 
we can also observe that the results computed using both $\alpha=0$ and
$\alpha=0.7$ are very similar for this example (Since the momentum has been
updated twice in Algorithm~\ref{algo-phase-space-mc-rattle}, using
$\alpha=0.7$ is equivalent to performing a single momentum update with a
parameter~$\alpha^2=0.49$; see Remark~\ref{rmk-different-steps-of-algo}). In
fact, the results are similar for a very wide range of values of
$\alpha$, including values up to~$\alpha=0.99$.
(Let us however mention that Algorithm~2 becomes deterministic
for~$\alpha=1.0$, and the convergence to equilibrium is therefore not guaranteed because of the breakdown of ergodicity. We indeed observed in this case that the sampler can be trapped in certain states or fails to sample the entire torus.)

\begin{table}[t!]
  \renewcommand{\tabcolsep}{5pt}
  \centering
  \begin{tabular}{C{0.07\textwidth}|C{0.12\textwidth}|*{4}{R{0.07\textwidth}}|*{4}{R{0.07\textwidth}}}
    \hline
    \multirow{2}{*}{$\alpha$} & \multirow{2}{*}{Scheme} &
    \multicolumn{4}{c|}{\small{No. of solutions in forward step}} &
    \multicolumn{4}{c}{\small{No. of solutions in reversibility check}}  \\
      \cline{3-10}
     & & \multicolumn{1}{c}{$0$} & \multicolumn{1}{c}{$1$} & \multicolumn{1}{c}{$2$} & \multicolumn{1}{c|}{$4$} & \multicolumn{1}{c}{$0$} & \multicolumn{1}{c}{$1$} & \multicolumn{1}{c}{$2$} & \multicolumn{1}{c}{$4$} \\
    \hline
 \multirow{7}{*}{$0.0$}  & Newton& $48.0\%$ & $52.0\%$ & $0.0\%$ & $0.0\%$ & $1.2\%$ & $98.8\%$ & $0.0\%$ & $0.0\%$ \\
    \cline{2-10}
 & PR & $45.9\%$ & $0.0\%$ & $49.9\%$ & $4.2\%$ & $0.0\%$ & $0.0\%$ & $91.2\%$ & $8.8\%$ \\
     & PR-far& $45.9\%$ & $0.0\%$ & $49.9\%$ & $4.2\%$ & $0.0\%$ & $0.0\%$ & $91.3\%$ & $8.7\%$ \\
     & PR50-far& $48.0\%$ & $50.9\%$ & $1.0\%$ & $0.1\%$ & $1.2\%$ & $96.8\%$ & $1.9\%$ & $0.2\%$ \\
    \cline{2-10}
     & Hom & $46.0\%$ & $0.0\%$ & $49.9\%$ & $4.1\%$ & $0.0\%$ & $0.0\%$ & $91.2\%$ & $8.8\%$ \\
     & Hom-far& $45.9\%$ & $0.0\%$ & $49.9\%$ & $4.2\%$ & $0.0\%$ & $0.0\%$ & $91.3\%$ & $8.6\%$ \\
     & Hom50-far & $48.0\%$ & $50.9\%$ & $1.0\%$ & $0.1\%$ & $1.2\%$ & $96.8\%$ & $1.9\%$ & $0.2\%$ \\
     \hline
     \hline
    \multirow{7}{*}{$0.7$} & Newton & $48.0\%$ & $52.0\%$ & $0.0\%$ & $0.0\%$ & $1.2\%$ & $98.8\%$ & $0.0\%$ & $0.0\%$ \\
    \cline{2-10}
     & PR & $45.9\%$ & $0.0\%$ & $49.9\%$ & $4.1\%$ & $0.0\%$ & $0.0\%$ & $91.2\%$ & $8.8\%$ \\
     & PR-far& $45.9\%$ & $0.0\%$ & $49.9\%$ & $4.2\%$ & $0.0\%$ & $0.0\%$ & $91.3\%$ & $8.7\%$ \\
     & PR50-far& $48.0\%$ & $50.9\%$ & $1.0\%$ & $0.1\%$ & $1.1\%$ & $96.8\%$ & $1.9\%$ & $0.2\%$ \\
    \cline{2-10}
     & Hom & $45.9\%$ & $0.0\%$ & $49.9\%$ & $4.1\%$ & $0.0\%$ & $0.0\%$ & $91.2\%$ & $8.8\%$ \\
     & Hom-far& $45.9\%$ & $0.0\%$ & $49.9\%$ & $4.2\%$ & $0.0\%$ & $0.0\%$ & $91.3\%$ & $8.7\%$ \\
     & Hom50-far & $48.0\%$ & $50.9\%$ & $1.0\%$ & $0.1\%$ & $1.2\%$ & $96.8\%$ & $1.9\%$ & $0.2\%$ \\
     \hline
  \end{tabular}
  \caption{Sampling of the uniform law on the torus. The percentages in ``forward step'' are the
  proportions of MCMC iterations among the total number $N=10^7$ of MCMC
  iterations where $i$ state(s) are found in the set $\Phi(x,p)$.
  The percentages in ``reversibility check'' are the proportions of the number
  of MCMC iterations where $i$ state(s) are found in the set $\Phi(x^1, p^{1,-})$
  (where $(x^1, p^{1,-})\in \Phi(x,p)$ is the selected proposal state) for MCMC
  iterations where the reversibility check step has been invoked (i.e.\ when at least one state has been found in $\Phi(x,p)$ in the forward step).
  \label{percentage-of-diff-strategy}}
\end{table}

\begin{table}[t!]
  \renewcommand{\tabcolsep}{7pt}
  \begin{tabular}{c|c|c|c|c|c|c}
    \hline
      $\alpha$ & Scheme &  FSR & BSR & TAR & Distance & Time (s)\\
    \hline
    \multirow{7}{*}{$0.0$} &  Newton & $0.52$ & $0.90$ & $0.45$ & $0.73$ & $1.6\times 10^{3}$ \\
    \cline{2-7}
     & PR & $0.54$ & $1.00$ & $0.44$ & $1.13$ & $5.7 \times 10^{3}$ \\
     & PR-far & $0.54$ & $1.00$ & $0.43$ & $1.18$ & $5.9\times 10^{3}$ \\
     & PR50-far & $0.52$ & $0.90$ & $0.45$ ($0.43$) & $0.74$ ($1.18$) & $1.6\times 10^{3}$ \\
    \cline{2-7}
     & Hom & $0.54$ & $1.00$ & $0.44$ & $1.13$ & $8.4 \times 10^{3}$ \\
     & Hom-far & $0.54$ & $1.00$ & $0.43$ & $1.18$ & $8.9 \times 10^{3}$ \\
     & Hom50-far & $0.52$ & $0.90$ & $0.45$ ($0.43$) & $0.74$ ($1.18$) & $1.8\times 10^{3}$ \\
    \hline
    \hline
    \multirow{7}{*}{$0.7$} &  Newton & $0.52$ & $0.90$ & $0.45$ & $0.73$ & $1.6\times 10^{3}$ \\
    \cline{2-7}
     & PR & $0.54$ & $1.00$ & $0.44$ & $1.13$ & $5.8 \times 10^{3}$ \\
     & PR-far & $0.54$ & $1.00$ & $0.43$ & $1.18$ & $5.4\times 10^{3}$ \\
     & PR50-far & $0.52$ & $0.90$ & $0.45$ ($0.43$) & $0.74$ ($1.18$) & $1.8\times 10^{3}$ \\
    \cline{2-7}
     & Hom & $0.54$ & $1.00$ & $0.44$ & $1.13$ & $8.9 \times 10^{3}$ \\
     & Hom-far & $0.54$ & $1.00$ & $0.43$ & $1.18$ & $8.3 \times 10^{3}$ \\
     & Hom50-far & $0.52$ & $0.90$ & $0.45$ ($0.43$) & $0.74$ ($1.18$) & $1.8\times 10^{3}$ \\
    \hline
  \end{tabular}
  \centering
  \caption{Sampling of the uniform law on the torus. The column ``FSR'' (short for ``forward
  success rate'') shows the ratios of MCMC iterations within $N=10^7$ MCMC
  iterations where at least one state is found in $\Phi(x,p)$ as a proposal state.
  The column ``BSR'' (short for ``backward success rate'') shows the ratios of
  the number of MCMC iterations for which the reversibility check step is
  successful, within the total number of MCMC iterations where the reversibility check step is invoked.  
 The column ``TAR'' (short for ``total acceptance rate'') displays the proportions of MCMC iterations within
  $N=10^7$ iterations in which the system jumps to a new state.
 The column ``Distance'' shows the average jump distance of the position
  component, provided that the system jumps to a new position. For each
  scheme, the run-time (seconds) spent in performing $N=10^7$ MCMC iterations is shown in the column ``Time (s)''.  For the (hybrid) schemes ``PR50-far''
  and ``Hom50-far'', the average jump rates and the average jump distances
  among the MCMC iterations for which the set $\Phi(x,p)$ is computed using either
  the PolynomialRoots.jl or HomotopyContinuation.jl packages (i.e.\
  every $50$ iterations; in total $2\times 10^5$ times) are shown in bracket.
  \label{tab-ex1-summary}}
\end{table}
\begin{figure}[t!]
\includegraphics[width=0.47\textwidth]{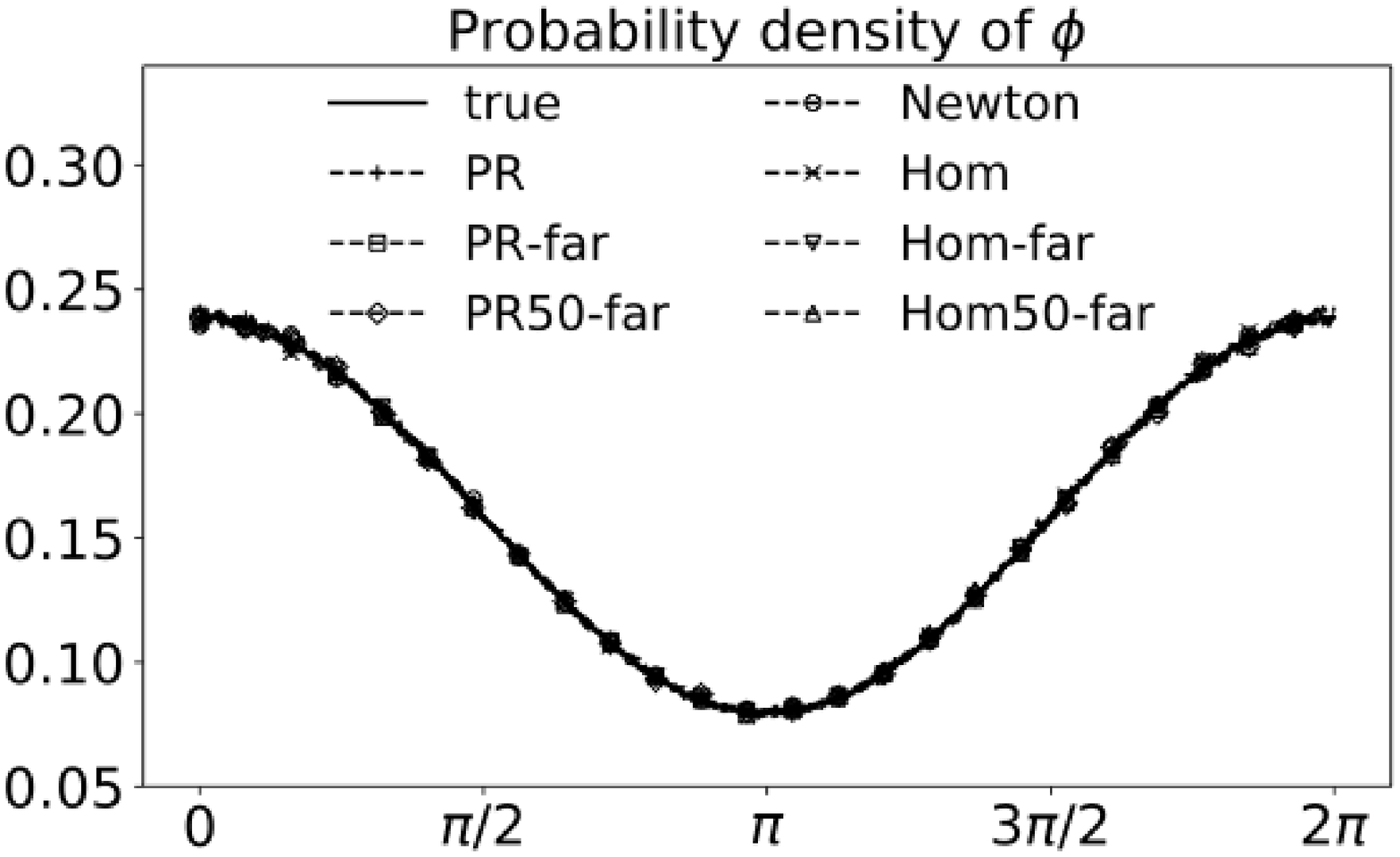}
\includegraphics[width=0.47\textwidth]{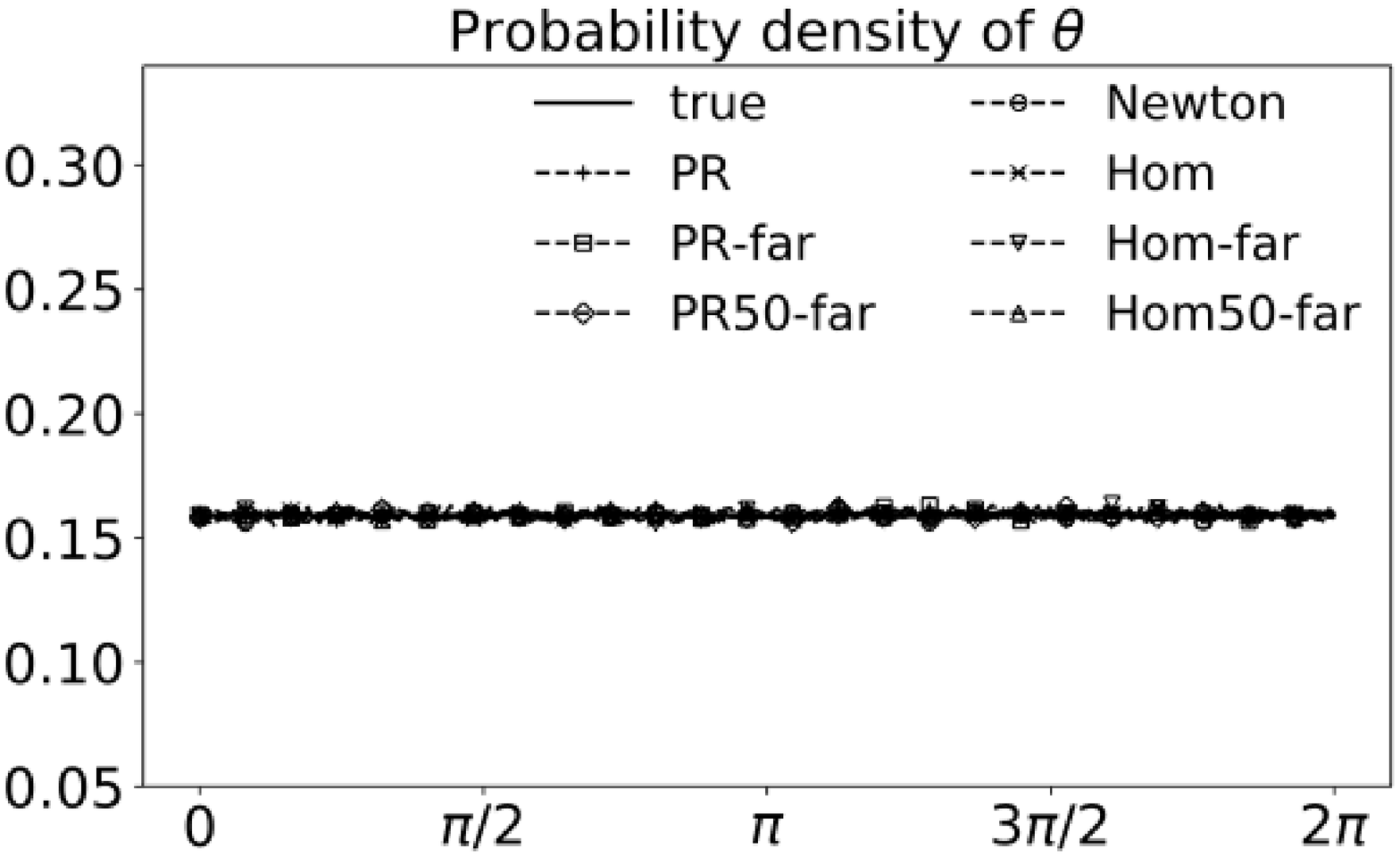}
\centering
  \caption{Sampling of the uniform law on the torus (for $\alpha=0$). Probability density profiles of the two angles $\phi,
  \theta$ in \eqref{ex1-polar}, estimated from $N=10^7$ sampled states using
  different schemes in Table~\ref{tab-schemes} with $\alpha=0$. The curves with label ``true'' are the reference probability density functions computed from \eqref{ex1-sigma}. 
  \label{fig-ex1-angle-distribution}}
\end{figure}

\paragraph{Bimodal distribution on the torus.}
In the second task, we apply Algorithm~\ref{algo-phase-space-mc-rattle}
to sample a bimodal distribution on the torus. The aim is to demonstrate that
using multiple proposal states in $\Phi(x,p)$ can improve the sampling
efficiency of MCMC schemes by introducing non-local jumps in Markov chains.
We consider the target distribution~\eqref{measure-mu-on-m} with $\beta = 20$ and (see Figure~\ref{fig-ex1-torus-potential}, right)
\begin{align}
V(x) = (x_1 - x_2)^2 + 5\left(\frac{x_1^2 + x_2^2}{(R+r)^2} - 1\right)^2\,.
  \label{ex1-mu-pot}
\end{align}
The potential has two global minima on~$\Sigma$ at $\left(\frac{R+r}{\sqrt{2}},
\frac{R+r}{\sqrt{2}}, 0\right)^T$ and $\left(-\frac{R+r}{\sqrt{2}},
-\frac{R+r}{\sqrt{2}}, 0\right)^T$.

We simulate the distributions of the two angles $\phi, \theta$ in \eqref{ex1-polar}
using the schemes ``Newton'', ``PR'', ``PR50-far'', ``Hom'', and ``Hom50-far''
from Table~\ref{tab-schemes} with both $\alpha=0$ and $\alpha=0.7$ (the two schemes ``PR-far'' and ``Hom-far'' are not used because they provide results similar to those of ``PR'' and ``Hom'', respectively).
The step-size in the one-step RATTLE with momentum reversal scheme \eqref{psi-map-by-rattle} is set
to $\tau = 0.8$. The value of the timestep is chosen in order to offer some empirical optimal trade-off in
terms of sampling the equilibrium measure at hand. 
The values of the other parameters are the same as in the first sampling task.
The numerical results are summarized in
Tables~\ref{tab-ex1-percentage-of-diff-strategy-bimode}--\ref{tab-ex1-summary-bimode}
and Figures~\ref{fig-ex1-angle-distribution-bimode}--\ref{fig-ex1-bimodel-trajectories}.
From Figure~\ref{fig-ex1-angle-distribution-bimode}, we observe that, unlike
the schemes ``PR'', ``PR50-far'', ``Hom'' and ``Hom50-far'', the scheme ``Newton''
still could not reproduce the correct density profile of $\theta$ within $N=10^7$ sampled states. Figure~\ref{fig-ex1-bimodel-trajectories} shows that,
in the scheme ``Newton'', the transition of the Markov chain from the basin of one local
minimum to the basin of the other one rarely happens, due to the bimodality of the target distribution,
implying that a larger sample size is needed in order to correctly sample the target
distribution using the ``Newton'' scheme.
On the other hand, by computing multiple proposal states in~$\Phi(x,p)$, the Markov chains in
the schemes ``PR'',  ``PR50-far'', ``Hom'' and ``Hom50-far'' are able to
perform non-local jumps, as shown in
Tables~\ref{tab-ex1-percentage-of-diff-strategy-bimode}--\ref{tab-ex1-summary-bimode}.
To investigate the effect of these non-local jumps, in the column ``Large jump rate'' 
of Table~\ref{tab-ex1-summary-bimode} we record the frequency of MCMC iterations when the first
  component $x_1$ of the state $x$ changes its sign among the total $N=10^7$ iterations. 
This can be used as an indicator for the occurrence of large jumps from one local minimum to the other. We see that
the frequencies of large jumps are higher when
multiple proposal states are computed, leading to better sampling performances
compared to the scheme ``Newton''. In particular, the two hybrid schemes
``PR50-far'' and ``Hom50-far'' achieve better sampling efficiency with
computational cost ($2.0\times
10^{3}$ seconds) similar to the one for the scheme ``Newton'' ($1.7\times 10^{3}$~seconds).
Finally, from Table~\ref{tab-ex1-percentage-of-diff-strategy-bimode} and
Table~\ref{tab-ex1-summary-bimode}, we can again observe that the results
computed using both $\alpha=0$ and $\alpha=0.7$ are very similar for this test example
(More generally, as in the previous test example, the results are in fact similar for a very wide range of values of $\alpha$). 

\begin{table}[t!]
  \renewcommand{\tabcolsep}{4.2pt}
  \begin{tabular}{C{0.05\textwidth}|C{0.12\textwidth}|*{4}{R{0.07\textwidth}}|R{0.06\textwidth}R{0.08\textwidth}R{0.07\textwidth}R{0.07\textwidth}}
    \hline
      \multirow{2}{*}{$\alpha$} & \multirow{2}{*}{Scheme} &
      \multicolumn{4}{c|}{\small{No. of solutions in forward step}} & \multicolumn{4}{c}{\small{No. of solutions in reversibility check}}  \\
      \cline{3-10}
     & & \multicolumn{1}{c}{$0$} & \multicolumn{1}{c}{$1$} &
     \multicolumn{1}{c}{$2$} &
     \multicolumn{1}{c|}{$4$} &
     \multicolumn{1}{c}{$0$} &
     \multicolumn{1}{c}{$1$} &
     \multicolumn{1}{c}{$2$} &
     \multicolumn{1}{c}{$4$} \\
    \hline
     \multirow{5}{*}{$0.0$} & Newton & $2.2\%$ & $97.8\%$ & $0.0\%$ & $0.0\%$
     & $0.0\%$ & $100.0\%$ & $0.0\%$ & $0.0\%$ \\
    \cline{2-10}
     & PR & $2.1\%$ & $0.0\%$ & $51.8\%$& $46.1\%$&$0.0\%$& $0.0\%$ & $75.2\%$ & $24.8\%$ \\
     & PR50-far& $2.2\%$ & $95.8\%$ & $1.0\%$ & $0.9\%$ & $0.0\%$ & $98.0\%$ & $1.6\%$ & $0.5\%$ \\
    \cline{2-10}
     & Hom & $2.1\%$ & $0.0\%$ & $51.8\%$ & $46.1\%$ & $0.0\%$ & $0.0\%$ & $75.2\%$ & $24.8\%$ \\
     & Hom50-far & $2.2\%$ & $95.8\%$ & $1.0\%$ & $0.9\%$ & $0.0\%$ & $98.0\%$ & $1.6\%$ & $0.5\%$ \\
     \hline
     \hline
     \multirow{5}{*}{$0.7$} & Newton & $2.2\%$ & $97.8\%$ & $0.0\%$ & $0.0\%$
     & $0.0\%$ & $100.0\%$ & $0.0\%$ & $0.0\%$ \\
    \cline{2-10}
     & PR & $2.1\%$ & $0.0\%$ & $51.8\%$ & $46.0\%$ & $0.0\%$ & $0.0\%$ & $75.2\%$ & $24.8\%$ \\
     & PR50-far& $2.2\%$ & $95.8\%$ & $1.0\%$ & $0.9\%$ & $0.0\%$ & $98.0\%$ & $1.6\%$ & $0.5\%$ \\
    \cline{2-10}
     & Hom & $2.1\%$ & $0.0\%$ & $51.8\%$ & $46.1\%$ & $0.0\%$ & $0.0\%$ & $75.2\%$ & $24.8\%$ \\
     & Hom50-far & $2.2\%$ & $95.8\%$ & $1.0\%$ & $0.9\%$ & $0.0\%$ & $98.0\%$ & $1.5\%$ & $0.5\%$ \\
     \hline
  \end{tabular}
  \caption{Sampling of the bimodal distribution on the torus. See Table~\ref{percentage-of-diff-strategy} for the meanings of the percentages. The schemes ``PR-far'' and ``Hom-far'' in Table~\ref{tab-schemes} provide similar results, which are not reported here.  \label{tab-ex1-percentage-of-diff-strategy-bimode}
  }
\end{table}

\begin{figure}[t!]
\includegraphics[width=0.47\textwidth]{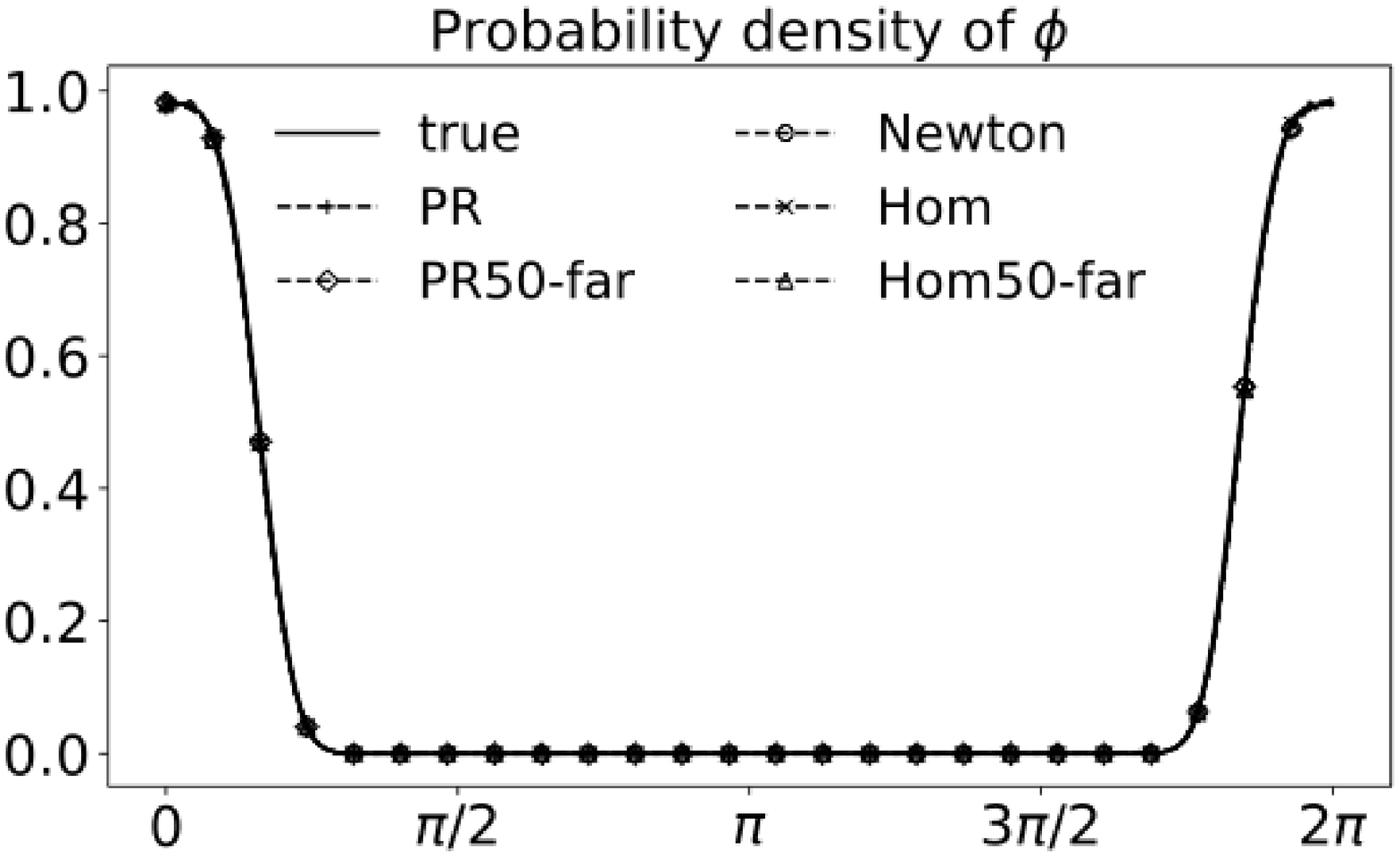}
\includegraphics[width=0.47\textwidth]{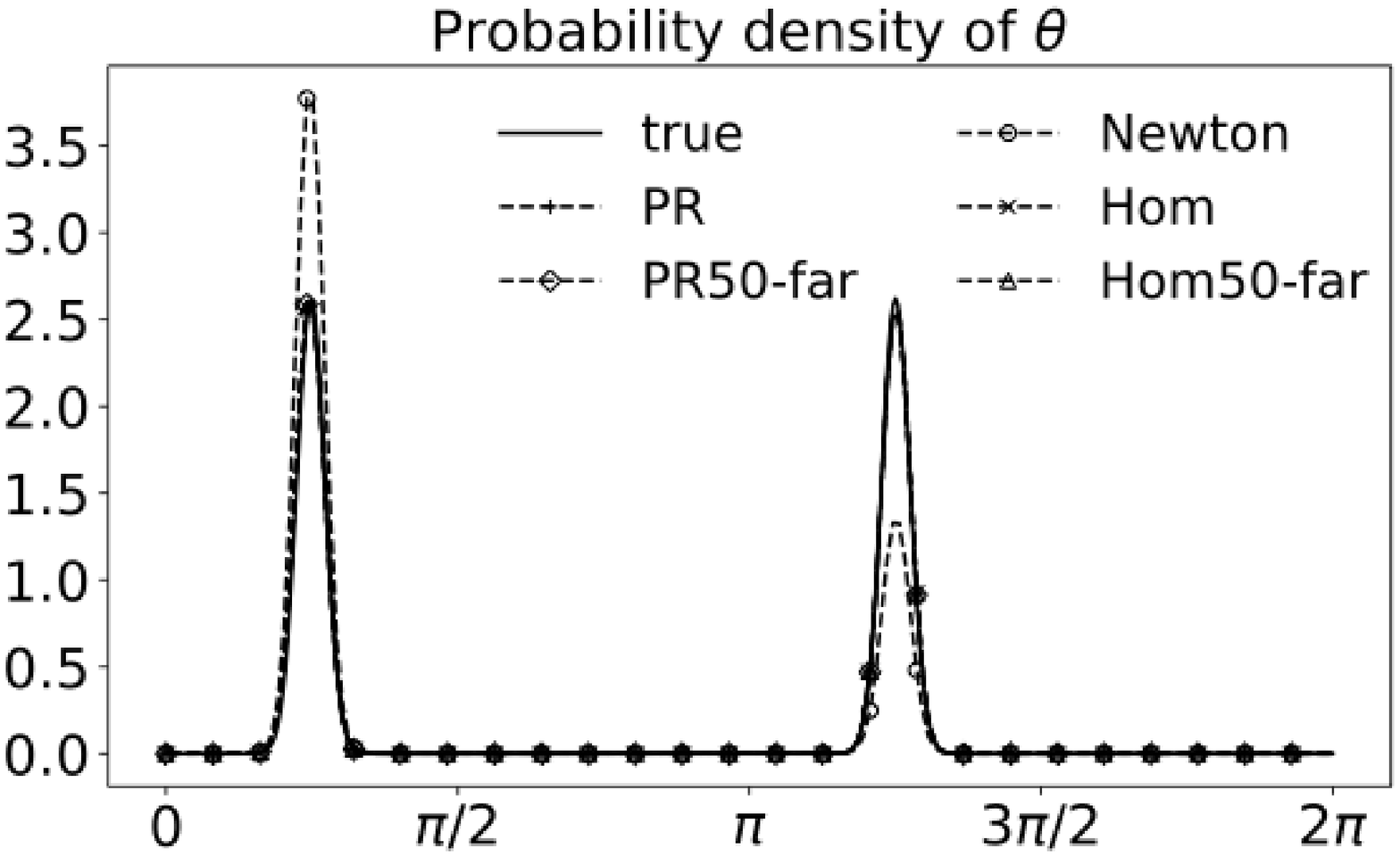}
\centering
  \caption{Sampling of the bimodal distribution on the torus (for $\alpha=0$). Probability
  density profiles of the two angles $\phi, \theta$ in \eqref{ex1-polar},
  estimated from $N=10^7$ sampled states using different schemes from
  Table~\ref{tab-schemes} with $\alpha=0$.
   (The results for both the schemes ``PR-far'' and ``Hom-far'', as well as for $\alpha=0.7$, are similar and therefore not reported.)
   The plots with label ``true'' are the reference probability density functions computed from \eqref{ex1-sigma} and \eqref{ex1-mu-pot}. 
   \label{fig-ex1-angle-distribution-bimode}}
\end{figure}

\begin{table}[t!]
  \renewcommand{\tabcolsep}{7pt}
  \centering
  \begin{tabular}{c|c|c|c|c|c|c}
    \hline
      $\alpha$ & Scheme &  FSR & BSR & TAR & Large jump rate &Time (s)\\
    \hline
     \multirow{5}{*}{$0.0$} & Newton & $0.98$ & $1.00$ & $0.60$ & $2.0 \times 10^{-7}$ & $1.7\times 10^{3}$ \\
    \cline{2-7}
     & PR & $0.98$ & $1.00$ & $0.22$ & $4.0 \times 10^{-3}$ & $7.5\times 10^{3}$ \\
     & PR50-far & $0.98$ & $1.00$ & $0.60$\,($0.17$) & $6.5\times 10^{-5}$ & $1.8\times 10^{3}$ \\
    \cline{2-7}
     & Hom & $0.98$ & $1.00$ & $0.22$ & $4.0 \times 10^{-3}$ & $1.2\times 10^{4}$ \\
     & Hom50-far & $0.98$ & $1.00$ & $0.60$\,($0.18$) & $6.7\times 10^{-5}$ & $2.0\times 10^{3}$ \\
    \hline
    \hline
     \multirow{5}{*}{$0.7$} & Newton & $0.98$ & $1.00$ & $0.60$ & $0.0$ & $1.7\times 10^{3}$ \\
    \cline{2-7}
     & PR & $0.98$ & $1.00$ & $0.22$ & $4.0 \times 10^{-3}$ & $6.9\times 10^{3}$ \\
     & PR50-far & $0.98$ & $1.00$ & $0.60$\,($0.17$) & $6.5\times 10^{-5}$ & $1.9\times 10^{3}$ \\
    \cline{2-7}
     & Hom & $0.98$ & $1.00$ & $0.22$ & $4.0 \times 10^{-3}$ & $9.9\times 10^{3}$ \\
     & Hom50-far & $0.98$ & $1.00$ & $0.60$\,($0.18$) & $6.3\times 10^{-5}$ & $1.9\times 10^{3}$ \\
    \hline
  \end{tabular}
  \caption{Sampling of the bimodal distribution on the torus. See
  Table~\ref{tab-ex1-summary} for the meaning of each column. The column
  ``Large jump rate'' records the frequency of MCMC iterations when the first
  component~$x_1$ of the state $x$ changes its sign among the total $N=10^7$ iterations. 
  \label{tab-ex1-summary-bimode}}
\end{table}
\begin{figure}[t!]
\includegraphics[width=1.05\textwidth]{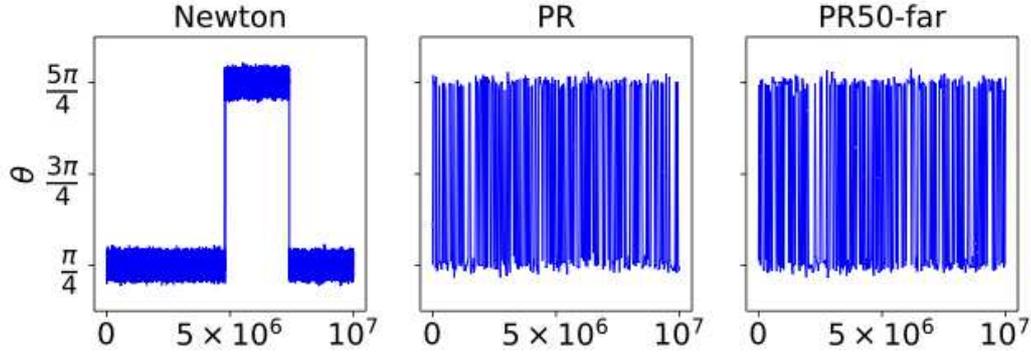}
\centering
  \caption{Sampling of the bimodal distribution on the torus (for $\alpha=0$). Sample trajectories of the angle $\theta$ in
  \eqref{ex1-polar} for the schemes ``Newton'', ``PR'', and
  ``PR50-far'' (see Table~\ref{tab-schemes}) with $\alpha=0$. To have a good visual effect,
  sampled states are plotted every $10$ iterations in ``Newton'', while 
  in both ``PR'' and ``PR50-far'' sampled states are shown every $3000$ iterations.
  The results should be analyzed together with the column ``Large jump rate'' in Table~\ref{tab-ex1-summary-bimode}. Trajectories for the schemes ``Hom'' and ``Hom50-far'' are not reported here because they are very similar to those of ``PR'' and ``PR50-far'',
  respectively.\label{fig-ex1-bimodel-trajectories}}
\end{figure}

\subsection{Disconnected components on a nine-dimensional sphere}
\label{subsec-example-2}

In the second example, we consider the level set
    $\Sigma = \left\{x \in \mathbb{R}^{10}~\middle|~ \xi(x)=0 \in \mathbb{R}^2\right\}$,
where $\xi:\mathbb{R}^{10}\rightarrow \mathbb{R}^2$ is given by 
\begin{align}
  \xi_1(x) = \frac{1}{2}\left(\sum_{i=1}^{10} x_i^2 - 9\right)\,,\qquad \xi_2(x) = x_1x_2x_3 - 2\,, 
  \label{ex2-equation-with-sqrt}
\end{align} 
for $x=(x_1, x_2, \ldots, x_{10})^T\in \mathbb{R}^{10}$.
Note that $\Sigma$ is a $8$-dimensional submanifold in $\mathbb{R}^{10}$ composed of $4$ connected components due to the second constraint $\xi_2(x) =x_1x_2x_3-2=0$.
Let us denote these $4$ connected components by  
\begin{align}
  \begin{split}
  \mathcal{C}_0 & = \left\{x\in \Sigma\,|\,x_1>0,\, x_2>0,\, x_3 > 0\right\},\,\\
  \mathcal{C}_1 & = \left\{x\in \Sigma\,|\,x_1 > 0,\, x_2 < 0,\, x_3 < 0\right\},\,\\
    \mathcal{C}_2 & = \left\{x\in \Sigma\,|\,x_1 < 0,\, x_2 > 0,\, x_3 < 0\right\},\,\\
    \mathcal{C}_3 & = \left\{x\in \Sigma\,|\,x_1 < 0,\, x_2 < 0,\, x_3 > 0\right\}.
  \end{split}
  \label{ex2-c1-to-c4}
\end{align}

We apply Algorithm~\ref{algo-phase-space-mc-rattle} with $\alpha = 0$ to sample the probability measure~\eqref{measure-mu-on-m} with $\beta=1.0$ and $V(x) = \frac{1}{2}(x_1 - 0.6)^2$. The following three schemes are used in the following numerical experiments:  
\begin{itemize}
  \item
  ``Newton'', where Newton's method is used to compute the set $\Phi(x,p)$ at each MCMC iteration. 
  \item
   ``Hom'', where the HomotopyContinuation.jl Julia package is used at each iteration to compute the set $\Phi(x,p)$, and one state is randomly chosen with uniform probability. 
 \item
   ``Hom10'', where the set $\Phi(x,p)$ is computed using  the
    HomotopyContinuation.jl Julia package every $10$ MCMC iterations, while
    Newton's method is used for the other iterations. As in the previous item, when multiple states in $\Phi(x,p)$ are obtained, one state is randomly chosen with uniform probability.
\end{itemize}

In contrast to the example of Section~\ref{subsec-ex1}, the PolynomialRoots.jl Julia package cannot be used since~$\xi$ is vector-valued.
Also, the scheme ``Hom10'' is used instead of ``Hom50'' (in the previous example),
because we observe that it is important to compute multiple solutions more frequently in order to sample the various components of $\Sigma$.
In each simulation, $N=10^{7}$ MCMC iterations are performed using Algorithm~\ref{algo-phase-space-mc-rattle}, with step-size $\tau = 0.5$.

As shown in Table~\ref{tab-ex2-percentage-of-diff-strategy}, multiple proposal
states are found in the MCMC iterations for both the ``Hom'' and ``Hom10'' schemes.
The empirical density distributions of the state's components $x_1$, $x_2$, and
$x_4$ under the target distribution $\nu_{\Sigma}$ are shown in Figure~\ref{fig-ex2-hist-x1-x2-x4}.
(Due to the symmetry of $\Sigma$ and the choice of potential $V$, $x_3$ has
the same distribution as $x_2$, while $x_5, x_6, \ldots, x_{10}$ have the same
distribution as $x_4$.)
While the same distributions of $x_1$ and $x_4$ are obtained using all three
schemes, we observe that the distribution of $x_2$ provided by the ``Newton''
scheme is different from the results given by the schemes ``Hom'' and ``Hom10''. The distribution of~$x_2$ with "Newton" is clearly far from the ground truth since it is not even.

To have a better understanding of the performance of the different schemes,
let us investigate the sampling of the four connected components $\mathcal{C}_0, \mathcal{C}_1, \mathcal{C}_2, \mathcal{C}_3$ in each scheme. 
First of all, it is easy to see that $\mathcal{C}_0$ and $\mathcal{C}_1$ have the same probability under $\nu_{\Sigma}$, 
while $\mathcal{C}_2$ and $\mathcal{C}_3$ also have the same (smaller than~$\mathcal{C}_0$ and $\mathcal{C}_1$) probability under $\nu_{\Sigma}$.
With this observation, from  the last four columns of Table~\ref{tab-ex2-summary}, we can already conclude that while both the ``Hom'' and
``Hom10'' schemes provide reasonable probabilities, the probabilities estimated
using the ``Newton'' scheme are inaccurate. This indicates that the sample size $N=10^7$ is still not
enough for the ``Newton'' scheme to correctly estimate the probability of each connected component. 
In Figure~\ref{fig-ex2-traj-phase}, the transitions among the components $\mathcal{C}_i$ for $0 \le i \le 3$ are shown for the three different schemes.
We see that, in the ``Newton'' scheme, the state of the Markov chain stays within the same connected component most of the time and the change from one
component to another happens very rarely (in total $18$ times in $10^7$ iterations). This implies that Newton's method mostly proposes moves to states within a connected component, which is indeed expected.
On the other hand, from Figure~\ref{fig-ex2-traj-phase} we observe that the
transitions among the connected components occur frequently both for ``Hom'' and ``Hom10''.
To provide more details, we plot in Figure~\ref{fig-ex2-phase-change-graphs}
the transitions among the connected components as directed graphs, whose nodes and edges represent the four components~$\mathcal{C}_i$ and the transitions among the components, respectively. The frequency of the transition from one connected component $\mathcal{C}_i$
to another component $\mathcal{C}_j$ for $0 \le i\neq j \le 3$, is shown on the edge from node $i$ to node~$j$.
The total frequency of transitions in each scheme is also given in the column ``CTF'' of Table~\ref{tab-ex2-summary}.
These results show that finding multiple proposal states in
Algorithm~\ref{algo-phase-space-mc-rattle} is helpful, particularly when the
submanifold $\Sigma$ contains multiple connected components. Notice moreover that the
hybrid scheme ``Hom10'' achieves a much better sampling efficiency than the
``Newton'' scheme for a comparable computational cost
(Table~\ref{tab-ex2-summary}). Finally, let us point out that in principle one
could also compute multiple proposal states using Newton's method with
different initial guesses (see Remark~\ref{rmk-initial-guesses}). We leave the
study of this approach to future work, since a careful choice of the initial
guesses is required in order to effectively find multiple solutions and some tuning is probably needed depending on the concrete applications at hand.

\begin{table}[hptb]
  \renewcommand{\tabcolsep}{5pt}
  \begin{tabular}{c|rrrrr|rrrrr}
    \hline
      \multirow{2}{*}{Scheme} & \multicolumn{5}{c|}{No. of solutions in forward step} & \multicolumn{5}{c}{No. of solutions in reversibility check}  \\
      \cline{2-11}
     & \multicolumn{1}{c}{$0$} & 
     \multicolumn{1}{c}{$1$} & 
  \multicolumn{1}{c}{$2$} & 
  \multicolumn{1}{c}{$4$} & 
  \multicolumn{1}{c|}{$6$} & 
\multicolumn{1}{c}{$0$} & 
\multicolumn{1}{c}{$1$} & 
\multicolumn{1}{c}{$2$} & 
\multicolumn{1}{c}{$4$} & 
\multicolumn{1}{c}{$6$} \\
    \hline
     Newton & $15.9\%$ & $84.1\%$ & $0.0\%$ & $0.0\%$ & $0.0\%$ & $0.1\%$ & $99.9\%$ & $0.0\%$ & $0.0\%$ & $0.0\%$ \\
     Hom & $13.3\%$ & $0.0\%$ & $76.6\%$ & $9.8\%$ & $0.2\%$ & $0.0\%$ & $0.0\%$ & $84.7\%$ & $15.2\%$ & $0.1\%$ \\
     Hom10 & $15.7\%$ & $75.7\%$ & $7.7\%$ & $1.0\%$ & $0.0\%$ & $0.1\%$ & $89.6\%$ & $8.7\%$ & $1.6\%$ & $0.0\%$ \\
     \hline
  \end{tabular}
  \caption{Sampling the submanifold of the 9D sphere. See Table~\ref{percentage-of-diff-strategy} for the meaning of the percentages. It is possible to find $3$ or $5$ states in ``Hom'' and ``Hom10'', but their percentages are below $0.1\%$.
    \label{tab-ex2-percentage-of-diff-strategy} }
\end{table}

\begin{table}[hptb]
  \renewcommand{\tabcolsep}{6pt}
  \centering
  \begin{tabular}{c|c|c|c|c|c|cccc}
    \hline
    \multirow{2}{*}{Scheme} & \multirow{2}{*}{FSR} & 
    \multirow{2}{*}{BSR} & 
    \multirow{2}{*}{Jump rate} & \multirow{2}{*}{CTF} & \multirow{2}{*}{Time
    (s)} & \multicolumn{4}{c}{Prob. in each component} \\
      \cline{7-10}
    & & & & & & $\mathcal{C}_0$ & $\mathcal{C}_1$ & $\mathcal{C}_2$ & $\mathcal{C}_3$ \\
    \hline
     Newton & $0.84$ & $1.00$ & $0.76$ & $1.8\times 10^{-6}$ & $2.4\times
     10^{3}$ & $0.44$ & $0.37$ & $0.15$ & $0.03$\\
     Hom & $0.87$ & $1.00$ & $0.43$ & $9.4\times 10^{-3}$ & $2.5\times 10^{4}$ & $0.40$ & $0.39$ & $0.11$ & $0.10$\\
     Hom10 & $0.84$ & $1.00$ & $0.73$ & $9.3\times 10^{-4}$ & $4.8\times 10^{3}$ & $0.39$ & $0.39$ & $0.10$ & $0.11$\\
    \hline
  \end{tabular}
  \caption{Sampling the submanifold of the 9D sphere. 
  The column ``CTF'' (short for ``component transition frequency'') shows the
  frequency at which the Markov chain jumps from one component to another in
  $N=10^{7}$ MCMC iterations. The empirical probabilities of visits to a component along a trajectory
  are shown in the last four columns. See
  Table~\ref{tab-ex1-summary} for the meaning of the other columns
  (i.e.\ ``FSR'', ``BSR'', ``Jump rate'', and ``Time (s)'').
  \label{tab-ex2-summary}}
\end{table}
\begin{figure}[htp]
\includegraphics[width=0.95\textwidth]{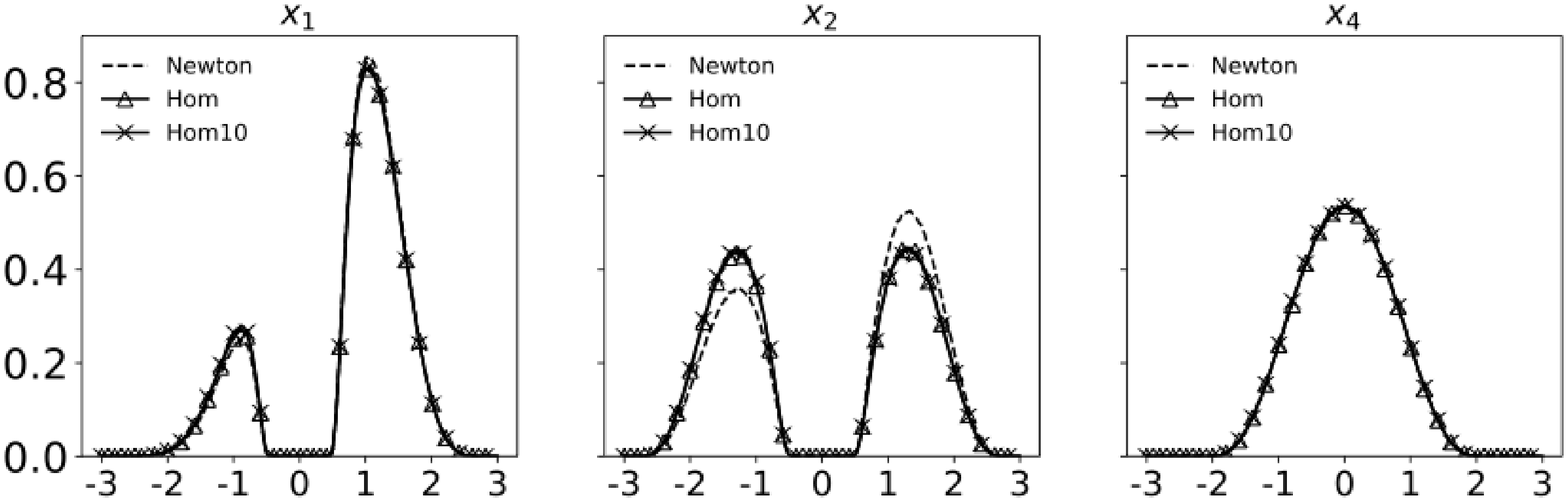}
\centering
  \caption{Sampling the submanifold of the 9D sphere: empirical probability densities of $x_1$, $x_2$ and $x_4$.  \label{fig-ex2-hist-x1-x2-x4}}
\end{figure}
\begin{figure}[htp]
\includegraphics[width=0.95\textwidth]{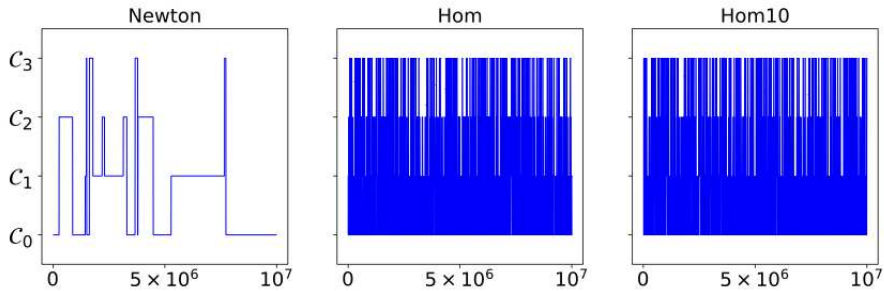}
\centering
  \caption{Sampling the submanifold of the 9D sphere: visits along the trajectories of the connected components $\mathcal{C}_i$ in \eqref{ex2-c1-to-c4} for $0\le i \le 3$, as functions of the iteration index, using the schemes ``Newton'', 
  ``Hom'', and ``Hom10''. While data for all MCMC iterations are plotted for the ``Newton'' scheme,
  data are shown only every $5000$ iterations for both ``Hom'' and ``Hom10'', in order to have a good visual effect.  \label{fig-ex2-traj-phase}}
\end{figure}
\begin{figure}[htp]
\includegraphics[width=0.95\textwidth]{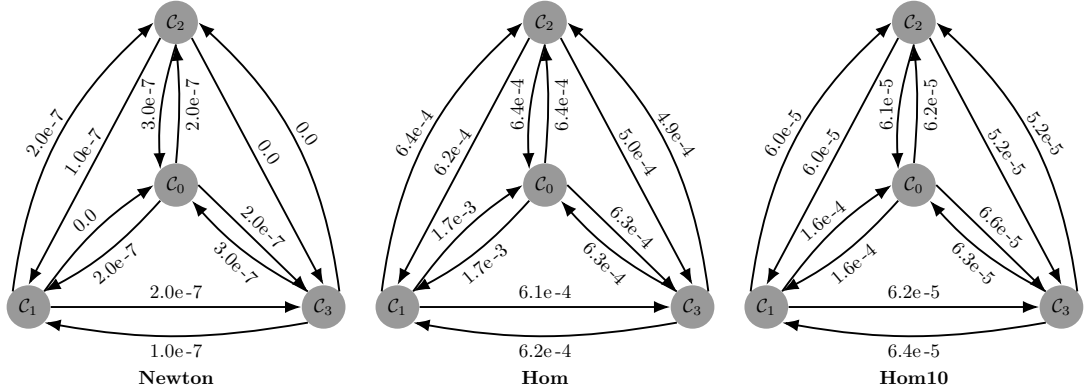}
\centering
  \caption{Sampling the submanifold of the 9D sphere:
    graph representations of the transitions among the connected components $\mathcal{C}_i$ for $0 \le i \le 3$. The frequencies of the transitions from one connected component~$\mathcal{C}_i$ to another component~$\mathcal{C}_j$ are shown on the edge from node~$i$ to~$j$. The reported frequencies are empirical frequencies obtained on a single trajectory of~$N=10^7$ iterations. \label{fig-ex2-phase-change-graphs}}
\end{figure}

\section{Proofs of the results presented in Section~\ref{sec-two-algorithms}}
\label{sec-proofs}

\subsection{Proofs of the results of Section~\ref{subsec-mama-sigma}}
\label{sec-proofs-MALA}

We prove Propositions~\ref{prop-differentiability-mala} and~\ref{prop-set-f-mala}, and then Theorem~\ref{thm-mala-on-sigma}.

\begin{proof}[Proof of Proposition~\ref{prop-differentiability-mala}]
  Concerning the first item, the map $G_x$ in \eqref{v-by-y} is clearly $C^1$-differentiable. Using the
  fact that, for all $x\in \Sigma$, the $k$ columns of $\nabla\xi(x)$ concatenated with the $d-k$ columns of $U_x$ form a basis in $\mathbb{R}^d$ such that $\nabla\xi(x)^TU_x=0$, we have, for all $x, y \in \Sigma$, 
  \begin{align}
    \begin{split}
     \det\left(\nabla\xi(y)^T \nabla\xi(x)\right)=0~ \Longleftrightarrow&~ \exists\, c\in \mathbb{R}^{k} \backslash \{ 0 \},\quad \nabla\xi(y)^T \nabla\xi(x)c = 0\,, \\
    \Longleftrightarrow&~ \exists\, (c,v) \in \left( \mathbb{R}^{k}\backslash \{ 0 \} \right) \times \left( \mathbb{R}^{d-k} \backslash \{ 0 \}\right), \quad \nabla\xi(x)c = U_y v\,,\\
    \Longleftrightarrow&~ \exists\, v\in \mathbb{R}^{d-k} \backslash \{ 0 \}, \quad U_x^TU_yv =0\,, \\
    \Longleftrightarrow&~ \det\left(U_x^TU_y\right)=0\,,
    \end{split}
    \label{equivalent-definitions-of-cx}
  \end{align}
  which implies that 
  \begin{align} 
    C_x = \left\{y\in\Sigma\,\middle|\, \det\left(\nabla\xi(y)^T \nabla\xi(x)\right)=0\right\} = \left\{y\in\Sigma\,\middle|\, \det\left(U_x^TU_y\right)=0\right\}\,,
    \label{cx-set-mala}
  \end{align}
  where the rightmost set in \eqref{cx-set-mala} is the set of critical points of $G_x$.
  Therefore, the differential of~$G_x$ has full rank at $y\in \Sigma\setminus
  C_x$ and, as a result,  there exists a neighborhood $\mathcal{Q}$ of
  $y$ such that $G_x|_{\mathcal{Q}}$ is a $C^1$-diffeomorphism.
  Since $G_x((G_x|_{\mathcal{Q}})^{-1}(\bar{v}))=\bar{v}$ for $\bar{v} \in
  G_x(\mathcal{Q})$, we have $(G_x|_{\mathcal{Q}})^{-1}(\bar{v}) \in \mathcal{F}_x(\bar{v})$ from~\eqref{relation-between-y-and-v}.

  For the second item, it is clear that the function $c(\cdot)$ in \eqref{multiplier-c-explicit} is $C^1$-differentiable. 
Substituting~\eqref{multiplier-c-explicit} in the definition of $F_x$ in
  \eqref{fx-mala}, using the fact that the column vectors of $\nabla\xi(x)$ and $U_x$ form a basis of $\mathbb{R}^d$ with
  $\nabla\xi(x)^TU_x=0$ and $U_x^TU_x=I_{d-k}$, and using the conclusion of the first item, 
   it is also straightforward to verify that $y = F_x(v,c(v))$ and
   $F_x(\bar{v}, c(\bar{v})) = (G_x|_{\mathcal{Q}})^{-1}(\bar{v}) \in
   \mathcal{Q}\cap \mathcal{F}_x(\bar{v})$ for all $\bar{v} \in G_x(\mathcal{Q})$. 
  Now, assume that there exist a neighborhood $\mathcal{O}$ of~$v$ and a $C^1$-differentiable function $\widetilde{c}: \mathcal{O}\rightarrow \mathbb{R}^k$ such that $y=F_x(v,\widetilde{c}(v))$ and $F_x(\bar{v}, \widetilde{c}(\bar{v})) \in \mathcal{Q}\cap \mathcal{F}_x(\bar{v})$ for all $\bar{v} \in \mathcal{O}$. 
       Clearly, we have $G_x(F_x(\bar{v}, \widetilde{c}(\bar{v}))) = \bar{v}$ for all $\bar{v}\in
       \mathcal{O}\cap G_x(\mathcal{Q})$ (see
       \eqref{relation-between-y-and-v}), which implies that $F_x(\bar{v}, \widetilde{c}(\bar{v})) = (G_x|_{\mathcal{Q}})^{-1}(\bar{v})$, since $G_x|_{\mathcal{Q}}$ is a $C^1$-diffeomorphism and 
       $F_x(\bar{v}, \widetilde{c}(\bar{v})) \in \mathcal{Q}$. Using the
       definition of $F_x$ in \eqref{fx-mala} we can directly compute $\widetilde{c}$ and obtain 
       \begin{align*}
	 \widetilde{c}(\bar{v}) =
	 \left(\nabla\xi(x)^T\nabla\xi(x)\right)^{-1}\nabla\xi(x)^T\left((G_x|_{\mathcal{Q}})^{-1}(\bar{v})
	 - x + \tau \nabla\overline{V}(x)\right)\,,  \quad \forall \bar{v}\in \mathcal{O}\cap
	 G_x(\mathcal{Q})\,.
       \end{align*}
       Therefore, $\widetilde{c}(\cdot)$ coincides with $c(\cdot)$ on $\mathcal{O}\cap G_x(\mathcal{Q})$.
 This proves the assertion in the second item.
\end{proof}

\medskip

\begin{proof}[Proof of Proposition~\ref{prop-set-f-mala}]
  We first show that $\mathcal{N}_x$ is closed and has zero Lebesgue measure. 
  Let $v_1, v_2, \ldots$ be an infinite sequence in $\mathcal{N}_x$ such that
  $\lim_{i\rightarrow +\infty} v_i = v \in \mathbb{R}^{d-k}$.
  Since $\mathcal{N}_x = G_x(C_x)$, there are $y^{(1)}, y^{(2)},\ldots \in C_x\subset \Sigma$,
   such that $v_i=G_x(y^{(i)})$, where $i=1,2,\dots$.
  Since $\Sigma$ is compact (Assumption~\ref{assump-xi}), we can assume (upon extracting a subsequence) that $\lim_{i\rightarrow +\infty} y^{(i)} = y \in
  \Sigma$. Moreover, we actually have $y\in C_x$ since the set $C_x$ in \eqref{c-x-in-prop1} is clearly a closed set.
  Using the continuity of the map $G_x$ in \eqref{v-by-y}, we can deduce that $v= \lim_{i\rightarrow +\infty} v_i = \lim_{i\rightarrow +\infty} G_x(y^{(i)}) = G_x(y) \in \mathcal{N}_x$. 
  This shows that $\mathcal{N}_x$ is a closed subset of $\mathbb{R}^{d-k}$.   
  Furthermore, \eqref{cx-set-mala} implies that $C_x$ is in fact the set of critical points of $G_x$.
  As a result, Sard's theorem asserts that~$\mathcal{N}_x = G_x(C_x)$ is a
  subset of $\mathbb{R}^{d-k}$ with Lebesgue measure zero. This proves the assertion in the first item.

  We next show that $\mathcal{F}_x(v)$ is a finite set for any~$v\in
  \mathbb{R}^{d-k}\setminus \mathcal{N}_x$. Assume by contradiction that there
  exists an element~$v\in \mathbb{R}^{d-k}\setminus\mathcal{N}_x$ for which
  this is not true. Then there are infinitely many Lagrange multipliers $c_1,
  c_2, \ldots \in \mathbb{R}^{k}$, which are different from each other, such
  that $F_x(v, c_i)=y^{(i)} \in \mathcal{F}_x(v)\subset\Sigma$. Since $\Sigma$ is compact, we can assume (upon extracting a subsequence) that $\lim_{i\rightarrow +\infty} y^{(i)} = y \in \Sigma$. 
  Using the definition of $F_x$ in \eqref{fx-mala}, we find $y^{(i)} - y^{(j)} = \nabla\xi(x)(c_i-c_j)$. Using the fact that $\nabla\xi(x)$ has full rank by~Assumption~\ref{assump-xi}, it holds
  \[
    c_i-c_j = \left(\nabla\xi(x)^T\nabla\xi(x)\right)^{-1}\nabla\xi(x)^T\left(y^{(i)}-y^{(j)}\right)\,,
  \]
  and so there exists $c \in \mathbb{R}^k$ such that $\lim_{i\rightarrow
  +\infty} c_i = c$ (since $(c_i)_{i \geq 1}$ is a Cauchy sequence) and $F_x(v,c) = y$. 
  It follows from \eqref{relation-between-y-and-v} that $G_x(y)=v$ and we have $y\not\in C_x$ since $v\in \mathbb{R}^{d-k}\setminus\mathcal{N}_x$.
  Therefore, by the first item of
  Proposition~\ref{prop-differentiability-mala}, we can find a neighborhood $\mathcal{Q}$ of $y$ such that $G_x: \mathcal{Q}\rightarrow
  G_x(\mathcal{Q})$ is a $C^1$-diffeomorphism. However, this leads to a contradiction with
  the assumption that $G_x(y^{(i)})=v$ and $\lim_{i\rightarrow +\infty}
  y^{(i)}=y$, and so, the set $\mathcal{F}_x(v)$ contains at most a finite number of elements.
This proves the assertion in the second item.

 Concerning the third item, by the first item of Proposition~\ref{prop-differentiability-mala}, we can find for $1 \le i \le n$ a neighborhood $\mathcal{Q}_i$ of $y^{(i)}$ such that $G_x|_{\mathcal{Q}_i}$ is a $C^1$-diffeomorphism. 
  Since $y^{(i)}$ for $1 \le i \le n$ are $n$ different states, by shrinking the
  neighborhoods $\mathcal{Q}_i$ if necessary, we can assume without loss of generality that $\mathcal{Q}_i\cap\mathcal{Q}_{i'}=\emptyset$, for 
  $1 \le i \neq i'\le n$. 
  Then, the second item of Proposition~\ref{prop-differentiability-mala} already implies that there are
  $n$ different Lagrange multiplier functions~$c^{(i)}$, locally given by
  \eqref{multiplier-c-explicit} with $\mathcal{Q}=\mathcal{Q}_i$, and the set
  $\mathcal{F}_x(\bar{v})$ has at least $n$ elements (since 
  $F_x(\bar{v}, c^{(i)}(\bar{v}))\in \mathcal{Q}_i \cap \mathcal{F}_x(\bar{v})$ and $\mathcal{Q}_i\cap \mathcal{Q}_{i'}=\emptyset$ for $1 \le i \neq i'\le n$)
  for $\bar{v}$ in the neighborhood $\mathcal{O}=\cap_{i=1}^n G_x(\mathcal{Q}_i)$ of $v$. Suppose that
  $\mathcal{F}_x(\bar{v})$ has $n+1$ elements or more for vectors $\bar{v} \in \mathcal{O}$ arbitrarily close to~$v$. This means that
 we can find a sequence of configurations $\left(\bar{y}^{(j)}\right)_{j \geq 1} \subset \Sigma$ such that 
  $\bar{y}^{(j)}\in \mathcal{F}_x(v_j)\setminus\{F_x(v_j,
  c^{(1)}(v_j)),\ldots, F_x(v_j, c^{(n)}(v_j))\}$, where $v_j\in \mathcal{O}$ with $v_j \rightarrow v$ as $j\rightarrow +\infty$. Then, we have $G_x(\bar{y}^{(j)}) = v_j$ and $\bar{y}^{(j)}\not\in\cup_{i=1}^n (G_x|_{\mathcal{Q}_i})^{-1}(\mathcal{O})$ (since, for each $1 \le i\le n$, $F_x(v_j, c^{(i)}(v_j))\in
  (G_x|_{\mathcal{Q}_i})^{-1}(\mathcal{O})\subseteq \mathcal{Q}_i$,
  $G_x(F_x(v_j, c^{(i)}(v_j)))=v_j$ and $G_x|_{\mathcal{Q}_i}$ is a diffeomorphism). It implies that $G_x(y)=v$, or equivalently $y\in \mathcal{F}_x(v)$ (see \eqref{relation-between-y-and-v}), where $y\not\in \cup_{i=1}^n (G_x|_{\mathcal{Q}_i})^{-1}(\mathcal{O})$ is a limiting point of $\bar{y}^{(j)}$ on $\Sigma$ (obtained after a possible extraction). This is in contradiction with the fact that
  $|\mathcal{F}_x(v)|=n$ and so, there exists a neighborhood $\mathcal{O}' \subset \mathcal{O}$ of $v$ such that $\mathcal{F}_x(\bar{v})$ has exactly $n$ values for any~$\bar{v}$ in~$\mathcal{O}'$.

  Finally, concerning the fourth item, we first note that $\mathcal{N}_x\cap
  \mathcal{B}_{x,0}=\emptyset$ (indeed, if $v \in \mathcal{N}_x$, then there
  exists $y \in C_x$ such that $v=G_x(y)$, or equivalently $y\in \mathcal{F}_x(v)$ by \eqref{relation-between-y-and-v},
  so that $|\mathcal{F}_x(v)| \geq 1$, which implies that $v \not \in \mathcal{B}_{x,0}$). Since $\mathcal{N}_x\cap \mathcal{B}_{x,i}=\emptyset$ by definition for $i \geq 1$, we can conclude that $\mathcal{N}_x$, $\mathcal{B}_{x,0}$, $\mathcal{B}_{x,1}, \ldots$ are disjoint subsets which form a partition of~$\mathbb{R}^{d-k}$. 
 The openness of the sets $\mathcal{B}_{x,i}$ for $i \ge 1$ follows from the third item. 
 The fact that $\mathcal{B}_{x,0}$ is open can easily be shown by proving that $\mathbb{R}^{d-k} \backslash
 \mathcal{B}_{x,0}$ is closed. Consider to this end $v \in \mathbb{R}^{d-k}$ and assume that there is a
 sequence $(v_j)_{j \ge 1} \subset \mathbb{R}^{d-k} \backslash
 \mathcal{B}_{x,0}$ such that $v_j \to v$ as $j \rightarrow +\infty$; in
 particular, $v_j = G_x(y_j)$ for some~$y_j \in \Sigma$. Then, up to
 extraction, $y_j \to y$ so that $v = G_x(y)$ and $|\mathcal{F}_x(v)| \geq 1$, and so $v \in \mathbb{R}^{d-k} \backslash\mathcal{B}_{x,0}$.
\end{proof}

\medskip

\begin{proof}[Proof of Theorem~\ref{thm-mala-on-sigma}]
  This proof is inspired by~\cite{goodman-submanifold}.
  For $x\in \Sigma$, we denote by $\widetilde{q}(x, \cdot)$ and $q(x, \cdot)$
  the distribution of the proposal state $\widetilde{x}^{(i+1)}\in\Sigma$ and
  the transition probability kernel of the Markov chain generated by
  Algorithm~\ref{algo-mala-on-sigma}, respectively. Let us show that $q$ satisfies 
  the equality~\eqref{usual-detailed-balance-condition-integral}
  in Definition~\ref{def-reversible}.
  Recall the set $\mathcal{D}$ defined in \eqref{admissible-set-two-steps} and
  the acceptance probability~$a(y\,|\,x)$ defined in~\eqref{rate-for-mala-on-sigma}. According to Algorithm~\ref{algo-mala-on-sigma}, $q$ is given by
  \begin{align}
    q(x,dy) = a(y\,|\,x)\, \mathbf{1}_{\mathcal{D}}(x,y)\, \widetilde{q}(x,dy)
    + \left[1-\int_{\Sigma}\, a(y'\,|\,x)\,\mathbf{1}_{\mathcal{D}}(x,y')\,\widetilde{q}(x,dy') \right] \delta_{x}(dy)\,,
    \label{q-in-metropolis-hasting}
  \end{align}
  where $\mathbf{1}_{\mathcal{D}}$ is the indicator function of the set $\mathcal{D}$
  and $\delta_{x}$ denotes the Dirac measure on $\Sigma$ centered at $x$.

  Next, we compute $\widetilde{q}(x,\cdot)$. Note that, unlike $q(x,\cdot)$, the measure~$\widetilde{q}(x,\cdot)$
  is not a probability measure. In fact, recall that $\widetilde{x}^{(i+1)}=y$ in Algorithm~\ref{algo-mala-on-sigma} is chosen
  randomly from the set $\Psi_x(v)$, with $v\in \mathbb{R}^{d-k}$ following
  the Gaussian distribution $\gamma$ in~\eqref{rescaled-gaussian}. 
  Since $\Sigma$ is compact and the map $G_x$ in~\eqref{v-by-y} is continuous,
  the image $G_x(\Sigma)$ is a compact subset of~$\mathbb{R}^{d-k}$. Together
  with the relation~\eqref{relation-between-y-and-v}, we know that $\Psi_x(v)=\emptyset$ for all $v$ in
  the set $\mathbb{R}^{d-k} \setminus G_x(\Sigma)$, which has positive measure
  under $\gamma$. In other words, $\widetilde{x}^{(i+1)}$ is defined 
      with probability less than $1$ or, equivalently, $\widetilde{q}(x,\Sigma) < 1$.
  Clearly, a state $y\in\Sigma$ can be picked only if $y\in \mbox{Im}\Psi_x$ (see \eqref{def-im-of-psi-x}). In
  other words, we have $\widetilde{q}(x,\Sigma\setminus \mbox{Im}\Psi_x)=0$ for all~$x \in \Sigma$. 
  Moreover, note that $y\in C_x$ implies $v\in \mathcal{N}_x$. Therefore, we have
  $\widetilde{q}(x,C_x)=0$ for all~$x \in \Sigma$, since $\gamma(\mathcal{N}_x)=0$ from the first item of Proposition~\ref{prop-set-f-mala}. 
  Applying Lemma~\ref{lemma-on-push-forward-by-psi} below, we obtain the following equality as measures on~$\Sigma$:
  \begin{equation}
    \widetilde{q}(x,dy)= \omega(y\,|\,x,G_x(y))\,\left|\det D G_x(y)\right|\left(\frac{2\pi}{\beta}\right)^{-\frac{d-k}{2}}\!\!
      \mathrm{e}^{-\frac{\beta|G_x(y)|^2}{2}}\, \mathbf{1}_{\mbox{Im}\Psi_x \setminus C_x}(y)\, \sigma_\Sigma(dy)\,,
    \label{proposal-density-q-push-forward-by-psi}
    \end{equation}
  where, using the definition of $G_x$ in~\eqref{v-by-y}, we have 
    \begin{align}
      \det DG_x(y) = (2\tau)^{-\frac{d-k}{2}}\det\left(U_x^TU_y\right) \,.
      \label{eqn-of-dg-at-y}
    \end{align}

      Let us define the set
  \begin{align}
    \widetilde{\mathcal{D}} = \left\{(x,y)\in \Sigma\times \Sigma ~\middle|~ y
    \in \mbox{Im}\Psi_x,~ x \in \mbox{Im}\Psi_{y},~ \det(\nabla\xi(y)^T\nabla \xi(x)) \neq 0 \right\}\,.
    \label{d-bar-set-mala}
  \end{align}
  Since $\widetilde{q}(x, \mbox{Im}\Psi_x \cap C_x) = 0$ for any $x\in \Sigma$ and
  $\mathcal{D}\setminus \widetilde{\mathcal{D}} \subseteq \{ 
  (x,y)\in \Sigma\times \Sigma ~|~ y\in \mbox{Im}\Psi_{x} \cap C_x\}$, the
  set $\widetilde{\mathcal{D}}$ differs from $\mathcal{D}$ by a set that has
  zero measure under the product measure~$\sigma_\Sigma(dx)\widetilde{q}(x,dy)$. Using this fact, as well as \eqref{q-in-metropolis-hasting}, we can compute 
  \begin{equation}
    \label{q-in-metropolis-hasting-precise}
    \begin{aligned}
      & \int_{\Sigma\times \Sigma} f(x,y)\, q(x,dy)\,\mathrm{e}^{-\beta
      V(x)}\,\sigma_\Sigma(dx) =  \int_{\Sigma\times \Sigma} f(x,y)\,
      a(y\,|\,x)\, \mathbf{1}_{\widetilde{\mathcal{D}}}(x,y)\,
      \mathrm{e}^{-\beta V(x)}\,\sigma_\Sigma(dx)\,\widetilde{q}(x,dy)\,\\
      & \qquad \qquad + \int_{\Sigma}\, f(x,x)\,\left[1-\int_{\Sigma}
      \,a(y'\,|\,x)\,\mathbf{1}_{\widetilde{\mathcal{D}}}(x,y')\,\widetilde{q}(x,dy')
      \right] \mathrm{e}^{-\beta V(x)}\,\sigma_\Sigma(dx)\,.    
    \end{aligned}
  \end{equation}
  Now, for $(x,y)\in \widetilde{\mathcal{D}}$, recall that $v'=G_y(x) \in
  \mathbb{R}^{d-k}$ is the unique element such that $x \in \Psi_{y}(v')$. Since $(x,y)\in \widetilde{\mathcal{D}}$ implies $x \not\in C_y$, i.e.\ $x$ is a regular point of $G_y$ (recall that $C_y$ is the set of critical points of
  $G_y$), we have $|\det DG_y(x)| \neq 0$. For the same reason, we have $|\det DG_x(y)| \neq 0$.
The choice
  \begin{align}
      \forall~(x, y) \in \widetilde{\mathcal{D}}, \quad a(y\,|\,x) =
      \min\left\{1,\frac{\omega(x\,|\,y,G_y(x))\left|\det D
      G_y(x)\right|}{\omega(y\,|\,x,G_x(y))\left|\det D
      G_x(y)\right|}\,\mathrm{e}^{-\beta \left[\left(V(y) +
      \frac{1}{2}|G_y(x)|^2\right) - \left(V(x)+\frac{1}{2}|G_x(y)|^2\right)\right]}\right\}
      \label{acceptance-rate-two-steps-in-proposal}
    \end{align}
    then implies with~\eqref{proposal-density-q-push-forward-by-psi} that
    \begin{align}
      \begin{aligned}
        & \int_{\Sigma\times \Sigma} f(x,y)\, a(y\,|\,x)\,
	\mathbf{1}_{\widetilde{\mathcal{D}}}(x,y)\, \mathrm{e}^{-\beta
	V(x)}\,\sigma_\Sigma(dx)\,\widetilde{q}(x,dy) \\
        & \qquad \qquad = \int_{\Sigma\times \Sigma} f(x,y)\, a(x\,|\,y)\,
	\mathbf{1}_{\widetilde{\mathcal{D}}}(x,y)\, \mathrm{e}^{-\beta
	V(y)}\,\sigma_\Sigma(dy)\,\widetilde{q}(y,dx)\,,
      \end{aligned}
          \label{eqn-identity-proof-of-mala-symmetry}
    \end{align}
    where we have also used the fact that $\mathbf{1}_{\mbox{Im}\Psi_x \setminus C_x}(y)
    =
\mathbf{1}_{\mbox{Im}\Psi_y \setminus C_y}(x) = 1$ when $\mathbf{1}_{\widetilde{\mathcal{D}}}(x,y)=1$.
    The equalities~\eqref{q-in-metropolis-hasting-precise} and
    \eqref{eqn-identity-proof-of-mala-symmetry} together with the fact that
    $\mathbf{1}_{\widetilde{\mathcal{D}}}$ is symmetric in turn imply~\eqref{usual-detailed-balance-condition-integral}. Moreover, 
using~\eqref{eqn-of-dg-at-y} it is clear that\,\footnote{The quantity~$|\det DG_x(y)|$ is
actually the absolute value of the determinant of the differential at
point $y$ of $G_x$ viewed as a map from $\Sigma$ to $T_x\Sigma$, and is
therefore independent of the choices of the orthonormal bases $U_x$ and
$U_y$.\label{footnote-on-determinant}} 
    \begin{align*}
      \det DG_x(y) = \det DG_y(x)  \,. 
    \end{align*}
Therefore,  \eqref{acceptance-rate-two-steps-in-proposal} reduces to the acceptance probability~\eqref{rate-for-mala-on-sigma}. 
    This allows to conclude that the Markov chain generated by
    Algorithm~\ref{algo-mala-on-sigma} is indeed reversible with respect to~$\nu_{\Sigma}$ on~$\Sigma$.
  \end{proof}

  Finally, we present the proof of the following technical Lemma, which has
  been used in the proof of Theorem~\ref{thm-mala-on-sigma} above.
  
  \begin{lemma} 
    Consider the Gaussian measure~$\gamma$ defined
    in~\eqref{rescaled-gaussian}. For $x \in \Sigma$, let $C_x$ be the set 
    in~\eqref{c-x-in-prop1}, $\mbox{Im}\Psi_x$ be the set in \eqref{def-im-of-psi-x}, and $\widetilde{q}(x, \cdot)$ be the distribution of the proposal state $\widetilde{x}^{(i+1)}\in\Sigma$ of the Markov chain generated by Algorithm~\ref{algo-mala-on-sigma}.
    Under Assumptions~\ref{assump-xi} and~\ref{assump-psi-on-state-space}, the following equality of measures on~$\Sigma$ holds:
    \begin{align}
      \widetilde{q}(x,dy)= \omega(y\,|\,x,G_x(y))\,\left|\det D G_x(y)\right|\left(\frac{2\pi}{\beta}\right)^{-\frac{d-k}{2}}\!\!
      \mathrm{e}^{-\frac{\beta|G_x(y)|^2}{2}}\, \mathbf{1}_{\mbox{Im}\Psi_x \setminus C_x}(y)\, \sigma_\Sigma(dy)\,,
    \label{proposal-density-q-push-forward-by-psi-repeat}
  \end{align}
    where the map $G_x$ is defined in \eqref{v-by-y}, 
    $\det D G_x(y)$ is given by \eqref{eqn-of-dg-at-y} and satisfies $\det D G_x(y)\neq 0$ for $y\in \Sigma\setminus C_x$.
      \label{lemma-on-push-forward-by-psi}
    \end{lemma}
  \begin{proof}
We have already shown that $\widetilde{q}(x, \Sigma\setminus \mbox{Im}\Psi_x)
    = \widetilde{q}(x,C_x) =0$, for all $x \in \Sigma$ (see the discussion above~\eqref{proposal-density-q-push-forward-by-psi}). 
    This implies that the support of $\widetilde{q}(x,\cdot)$ is included in
    $\mbox{Im}\Psi_x\setminus C_x$. It is obvious that the support of the
    measure on the right hand side of
    \eqref{proposal-density-q-push-forward-by-psi-repeat} is also included in
    $\mbox{Im}\Psi_x\setminus C_x$. Therefore, to verify~\eqref{proposal-density-q-push-forward-by-psi-repeat}, it is sufficient to show that both measures agree on $\mbox{Im}\Psi_x\setminus C_x$.
    Note that we only need to consider the case where $\mbox{Im}\Psi_x\setminus C_x \neq \emptyset$, since otherwise $\widetilde{q}(x,\cdot)\equiv 0$ and therefore \eqref{proposal-density-q-push-forward-by-psi-repeat} holds trivially (both sides are zero). 
    Intuitively, $\widetilde{q}(x,\cdot)$ corresponds to the image measure of the Gaussian density in~$v$ by~$\Psi_x$. 
   However, the push-forward measure by a set-valued map is in general not well defined. 
   Nonetheless, in our context, since the map $\Psi_x$ is injective (in the sense
    of \eqref{psi-given-y-v-is-unique}), locally $C^1$-differentiable and
    invertible (in the sense of the third item of
    Assumption~\ref{assump-psi-on-state-space}), there is a natural way to
    define the push-forward measure of $\gamma$ by $\Psi_x$. We use the
    notation $(\Psi_x)_{\widetilde{\#}}\gamma$ instead of the standard
    one $(\Psi_x)_{\#}\gamma$ for push-forward measures to indicate the difference.
    Specifically, in view of Algorithm~\ref{algo-mala-on-sigma}, we define 
  \begin{equation}
    \forall~ \mathcal{B} \in \mathbb{S}, \qquad
    \big((\Psi_x)_{\widetilde{\#}}\gamma\big)(\mathcal{B}) = \gamma\left(
    \left\{v\in \mathbb{R}^{d-k}\,\middle|\, \Psi_x(v)\cap \mathcal{B} \neq
    \emptyset\right\}\right)\,,
    \label{def-push-forward-by-psi} 
  \end{equation} 
  where 
  \begin{equation}
    \mathbb{S}=\{\emptyset\}\cup \left\{ \mathcal{B} \mbox{~non-empty and measurable}\,|\,\mathcal{B}\subseteq
      \mbox{Im}\Psi_x\setminus C_x, \ G_x|_{\mathcal{B}}~ \mbox{is injective} \right\}\,.
      \label{def-semi-ring-s}
  \end{equation}
    Note that Assumption~\ref{assump-psi-on-state-space} implies that the non-empty set $\mbox{Im}\Psi_x \setminus C_x$ is an open subset of $\Sigma$. The set $\mathbb{S}$ forms a semi-ring, since it is straightforward to see that $\mathcal{B}_1\cap \mathcal{B}_2, \mathcal{B}_1\setminus \mathcal{B}_2\in \mathbb{S}$ when $\mathcal{B}_1,\mathcal{B}_2\in \mathbb{S}$.
    Applying the extension theorem~\cite[Theorem 11.3 of Section 11, Chapter 2]{Probability-Billingsley-3rd}, we can extend $(\Psi_x)_{\widetilde{\#}}\gamma$ to a measure over the $\sigma$-ring
    generated by $\mathbb{S}$. This $\sigma$-ring is actually a
    $\sigma$-algebra of $\mbox{Im}\Psi_x\setminus C_x$, since it contains
    $\mbox{Im}\Psi_x\setminus C_x$. To prove the latter statement, we can write $\mbox{Im}\Psi_x\setminus C_x$ as a countable union of compact subsets
    of $\Sigma$
    \[
    \mbox{Im}\Psi_x\setminus C_x =\bigcup_{i=1}^{+\infty} \left\{y\in
    \Sigma\,\middle|\,d_{\Sigma}(y, \Sigma\setminus(\mbox{Im}\Psi_x\setminus C_x))\ge \frac{1}{i} \right\}.
    \]
    The subsets on the right hand side of the above equality are indeed compact since they are closed subsets of the compact set~$\Sigma$. Since each such compact subset has a finite covering consisting of open subsets of
    $\mbox{Im}\Psi_x\setminus C_x$ that belong to $\mathbb{S}$~(the open
    subsets being of the form $\mathcal{Q}\cap  \Psi_x^{(j)}(\mathcal{O})$, where 
    $\mathcal{Q}$ and $\Psi_x^{(j)}(\mathcal{O})$ are the neighborhoods
    discussed in Proposition~\ref{prop-differentiability-mala} and in the
    third item of Assumption~\ref{assump-psi-on-state-space}, respectively), we know that $\mbox{Im}\Psi_x\setminus C_x$ is a countable union of sets in \eqref{def-semi-ring-s}. 
     Moreover, the $\sigma$-algebra generated by $\mathbb{S}$ is indeed the Borel $\sigma$-algebra of $\mbox{Im}\Psi_x\setminus C_x$, 
     since the set $\mathbb{S}$ defined in~\eqref{def-semi-ring-s} consists of Borel
     measurable sets, and includes all open sets small enough in order for~$G_x$ to be injective. 
The measure~$(\Psi_x)_{\widetilde{\#}}\gamma$ can then be further extended to the Borel $\sigma$-algebra of $\Sigma$, such that the set $\Sigma\setminus (\mbox{Im}\Psi_x\setminus C_x)$ has measure zero. Since we are dealing with $\sigma$-finite measures, such an extension is unique, and we will still denote it by~$(\Psi_x)_{\widetilde{\#}}\gamma$ with some abuse of notation.
According to Algorithm~\ref{algo-mala-on-sigma} and the definition~\eqref{def-push-forward-by-psi}, we have the following equality in the sense of measures on~$\Sigma$:\footnote{According to Algorithm~\ref{algo-mala-on-sigma} and~\eqref{def-push-forward-by-psi}, the two measures coincide for sets in~\eqref{def-semi-ring-s}. As a result, they are equal as measures over the Borel $\sigma$-algebra of~$\Sigma$, due to the uniqueness of the extension of a $\sigma$-finite measure.}
\begin{equation}
    \widetilde{q}(x,dy)=\, \omega(y\,|\,x,G_x(y))\,
    \big((\Psi_x)_{\widetilde{\#}}\gamma\big)(dy) \,.
    \label{tilde-q-by-push-forward-by-psi}
    \end{equation}
To prove \eqref{proposal-density-q-push-forward-by-psi-repeat}, it is therefore sufficient to prove the identity
    \begin{equation}
      \int_{\Sigma} f(y)\, ((\Psi_x)_{\widetilde{\#}}\gamma)(dy) = 
      \int_{\Sigma}
      f(y) \,\left|\det D G_x(y)\right|\left(\frac{2\pi}{\beta}\right)^{-\frac{d-k}{2}}\!\!
      \mathrm{e}^{-\frac{\beta|G_x(y)|^2}{2}}\,
      \mathbf{1}_{\mbox{Im}\Psi_x \setminus C_x}(y)\, \sigma_\Sigma(dy)\,,
      \label{int-identity-pushfwd}
    \end{equation}
    for all bounded measurable test functions $f: \Sigma \rightarrow \mathbb{R}$.
    We prove this equality by a limiting procedure involving the dominated convergence theorem.
    
    For any $\epsilon > 0$, the set $\mathcal{Q}^\epsilon:=\{y\in \Sigma\,|\,d_{\Sigma}(y,\Sigma\setminus
    (\mbox{Im}\Psi_x\setminus C_x))< \epsilon\}$, where $d_{\Sigma}$ denotes the Riemannian distance on $\Sigma$, is a neighborhood of the set 
    $\Sigma\setminus (\mbox{Im}\Psi_x\setminus C_x)$, and we have 
    \begin{equation}
    \forall~0 < \epsilon < \epsilon', \quad \mathcal{Q}^\epsilon \subseteq \mathcal{Q}^{\epsilon'}\,, \qquad\mbox{and}\qquad \bigcap_{\epsilon >0}\mathcal{Q}^\epsilon= \Sigma\setminus (\mbox{Im}\Psi_x\setminus C_x)\,,
      \label{properties-of-sets-q-eps}
    \end{equation}
 where the second identity holds because $\Sigma\setminus (\mbox{Im}\Psi_x\setminus C_x)$ is a closed set. 
      For any $y \in \mbox{Im}\Psi_x\setminus C_x$, the third item of
      Assumption~\ref{assump-psi-on-state-space} implies that there exists
      a neighborhood $\mathcal{O}_v$ of $v$, where $v=G_x(y)\in \mathbb{R}^{d-k}$ satisfies $y\in \Psi_x(v)$, 
      and a $C^1$-diffeomorphism $\Psi_x^{(j)}: \mathcal{O}_v\rightarrow
      \Psi_x^{(j)}(\mathcal{O}_v)$, such that $y=\Psi_x^{(j)}(v)$ and $\Psi_x^{(j)}(\mathcal{O}_v)
      \subseteq \mbox{Im}\Psi_x$. We set $\mathcal{Q}_y=\Psi_x^{(j)}(\mathcal{O}_v)$. 
      Since $\Psi_x^{(j)}$ is a $C^1$-diffeomorphism on $\mathcal{O}_v$ and
      $G_x(\Psi_x^{(j)}(\bar{v}))=~\bar{v}$ for all $\bar{v}\in \mathcal{O}_v$, we have that $\mathcal{Q}_y$ is a neighborhood of $y$ and 
     $G_x|_{\mathcal{Q}_y}=(\Psi_x^{(j)})^{-1}$, which implies that
$G_x|_{\mathcal{Q}_y}$ is a $C^1$-diffeomorphism and in particular that
    $\mathcal{Q}_y \subseteq \mbox{Im}\Psi_x \setminus C_x$ (since $C_x$ is the set of critical points of the map $G_x$).
By the definition \eqref{def-semi-ring-s}, we have $\mathcal{Q}_y\in \mathbb{S}$.

    Let us fix $\epsilon>0$. Note that the union of the open sets $\mathcal{Q}^\epsilon$ and $\mathcal{Q}_y$ for $y\in \mbox{Im}\Psi_x \setminus C_x$ forms an open covering of $\Sigma$,
    i.e.\ $\mathcal{Q}^\epsilon \cup (\cup_{y\in \mbox{Im}\Psi_x \setminus C_x} \mathcal{Q}_y)=\Sigma$.
    Since $\Sigma$ is compact, we can find a finite subcovering, which we denote
    by $\mathcal{Q}^{(0)}, \mathcal{Q}^{(1)}, \ldots, \mathcal{Q}^{(n)}$,
    where $n\ge 1$. We can assume without loss of generality that $\mathcal{Q}^{(0)}=\mathcal{Q}^\epsilon$ and,
for each $i=1,\ldots, n$, $\mathcal{Q}^{(i)} = \mathcal{Q}_y \subseteq \mbox{Im}\Psi_x \setminus
    C_x$ for some $y\in\mbox{Im}\Psi_x \setminus C_x$.
    We use the fact that there exists a partition of unity for the finite
    covering $\mathcal{Q}^{(0)}, \mathcal{Q}^{(1)}, \ldots, \mathcal{Q}^{(n)}$ of~$\Sigma$; see Chapter 2.3 of~\cite{lang2002differential}. 
    In other words, there exist functions $\rho_0,\rho_1, \ldots, \rho_n: \Sigma\rightarrow [0,1]$,
    such that the support of $\rho_i$ is contained in $\mathcal{Q}^{(i)}$,
    where $i=0, 1,\dots, n$, and
    $\sum_{i=0}^{n}\rho_i(y)=1$ for all $y\in \Sigma$.
    Define $f^\epsilon(y) = f(y)\mathbf{1}_{\Sigma\setminus
    \mathcal{Q}^\epsilon}(y)=f(y)\mathbf{1}_{\Sigma\setminus \mathcal{Q}^{(0)}}(y)$, for all $y \in \Sigma$.
Since the support of $\rho_0$ is contained in $\mathcal{Q}^{(0)}$, it is clear that $\rho_0(y)f^\epsilon(y) = 
\rho_0(y) f(y) \mathbf{1}_{\Sigma\setminus \mathcal{Q}^{(0)}}(y) = 0$. 
Therefore, we have 
    \begin{equation}
      \begin{aligned}
	\int_{\Sigma} f^\epsilon(y)\, ((\Psi_x)_{\widetilde{\#}}\gamma)(dy) = &
	\int_{\Sigma} \Big(\sum_{i=0}^n\rho_i(y)\Big) f^\epsilon(y)\,
	((\Psi_x)_{\widetilde{\#}}\gamma)(dy) \\
	 =& \sum_{i=1}^n \int_{\Sigma} \rho_i(y) f^\epsilon(y)\,
	 ((\Psi_x)_{\widetilde{\#}}\gamma)(dy) \\
	 =& \sum_{i=1}^n \int_{\mathcal{Q}^{(i)}} \rho_i(y) f^\epsilon(y)\,
	 ((\Psi_x)_{\widetilde{\#}}\gamma)(dy) \,,
	\end{aligned}
	\label{lfs-f-eps-pu-1}
    \end{equation}
    where in the last equality we have used the fact that the support of $\rho_i$ is contained in $\mathcal{Q}^{(i)}$ for $i = 1, \ldots, n$.
To proceed, we note that, by the construction of the sets $\mathcal{Q}^{(i)}$, $i = 1, \ldots, n$,
the map $G_x|_{\mathcal{Q}^{(i)}}$ is a $C^1$-diffeomorphism and one has, for all measurable sets $\mathcal{B}\subseteq \Sigma$,
$\{v\in \mathbb{R}^{d-k}\,|\, \Psi_x(v)\cap (\mathcal{B}\cap \mathcal{Q}^{(i)}) \neq \emptyset\} = G_x(\mathcal{B}\cap \mathcal{Q}^{(i)})$.
Therefore, since $\mathcal{B}\cap \mathcal{Q}^{(i)} \in \mathbb{S}$, we find, using \eqref{def-push-forward-by-psi}, 
\begin{equation*}
    ((\Psi_x)_{\widetilde{\#}}\gamma)(\mathcal{B} \cap \mathcal{Q}^{(i)})  =
    \gamma\big(\{v\in \mathbb{R}^{d-k}\,|\, \Psi_x(v)\cap (\mathcal{B}\cap \mathcal{Q}^{(i)}) \neq \emptyset\}\big) 
    =  \gamma\big(G_x(\mathcal{B} \cap \mathcal{Q}^{(i)})\big)\,, 
\end{equation*}
which is equivalent to 
\begin{equation*}
  \int_{\mathcal{Q}^{(i)}}  \mathbf{1}_{\mathcal{B}}(y)\, ((\Psi_x)_{\widetilde{\#}}\gamma)(dy) 
  = \int_{G_x(\mathcal{Q}^{(i)})}
  \mathbf{1}_{\mathcal{B}}\left((G_x|_{\mathcal{Q}^{(i)}})^{-1}(v)\right)\gamma(dv) \,,
\end{equation*}
for all measurable sets $\mathcal{B}\subseteq \Sigma$. 
By density, this identity also holds when we replace
$\mathbf{1}_{\mathcal{B}}(\cdot)$ by any bounded measurable function. Then,
since $(G_x|_{\mathcal{Q}^{(i)}})^{-1}$ is a $C^1$-diffeomorphism from
$G_x(\mathcal{Q}^{(i)})\subset\mathbb{R}^{d-k}$ to $\mathcal{Q}^{(i)}\subset
\Sigma$ which provides a local parameterization for~$\Sigma$, we can compute,
by definition\footnote{In fact, since the surface measure $\sigma_\Sigma$ on $\Sigma$
is induced by the standard Euclidean scalar product in~$\mathbb{R}^d$ and 
$\big|\det D(G_x|_{\mathcal{Q}^{(i)}})^{-1}\big| = \big[\det
\big(D(G_x|_{\mathcal{Q}^{(i)}})^{-1} (D(G_x|_{\mathcal{Q}^{(i)}})^{-1})^T\big)\big]^{\frac{1}{2}}$, the
determinant factor here coincides with the factor to define the surface
measure~$\sigma_\Sigma$ \cite[Section III.3]{chavel_2006}.} of the surface measure~$\sigma_\Sigma(dy)=|\det D(G_x|_{\mathcal{Q}^{(i)}})^{-1}(v)| \, dv=|\det DG_x(y)|^{-1} \,dv$,
\begin{equation}
  \begin{aligned}
    &	 \sum_{i=1}^n \int_{\mathcal{Q}^{(i)}} \rho_i(y) f^\epsilon(y)\,
    ((\Psi_x)_{\widetilde{\#}}\gamma)(dy) 
    = \sum_{i=1}^n \int_{G_x(\mathcal{Q}^{(i)})}
    \rho_i\big((G_x|_{\mathcal{Q}^{(i)}})^{-1}(v)\big) f^\epsilon\big((G_x|_{\mathcal{Q}^{(i)}})^{-1}(v)\big)\, \gamma(dv) \\
    &\qquad \qquad \qquad =  \sum_{i=1}^n \int_{\mathcal{Q}^{(i)}}
      \rho_i(y)\,f^\epsilon(y) \,\left|\det D G_x (y)\right|\left(\frac{2\pi}{\beta}\right)^{-\frac{d-k}{2}}\!\!  \mathrm{e}^{-\frac{\beta|G_x(y)|^2}{2}}\,\sigma_\Sigma(dy)\,,
  \end{aligned}
	\label{lfs-f-eps-pu-2}
\end{equation}
where we recall that $\det D G_x (y)$ is defined by \eqref{eqn-of-dg-at-y} and is non-zero at $y\in \mathcal{Q}^{(i)}\subset \mbox{Im}\Psi_x \setminus C_x$. 
On the other hand, since $\rho_0(y)f^\epsilon(y)=0$ for all~$y\in \Sigma$, 
\begin{equation}
  \begin{aligned}
  & \int_{\Sigma}
      f^\epsilon(y) \,\left|\det D G_x(y)\right|\left(\frac{2\pi}{\beta}\right)^{-\frac{d-k}{2}}\!\!
      \mathrm{e}^{-\frac{\beta|G_x(y)|^2}{2}}\,
      \mathbf{1}_{\mbox{Im}\Psi_x \setminus C_x}(y)\,
      \sigma_\Sigma(dy)\\
     & \qquad  = \sum_{i=1}^n \int_{\Sigma}
       \rho_i(y) f^\epsilon(y) \,\left|
       \det D G_x(y)\right|\left(\frac{2\pi}{\beta}\right)^{-\frac{d-k}{2}}\!\!
      \mathrm{e}^{-\frac{\beta|G_x(y)|^2}{2}}\, \mathbf{1}_{\mbox{Im}\Psi_x \setminus C_x}(y)\,
      \sigma_\Sigma(dy)\\
      & \qquad = \sum_{i=1}^n \int_{\mathcal{Q}^{(i)}}
       \rho_i(y) f^\epsilon(y) \,\left|
       \det D G_x(y)\right|\left(\frac{2\pi}{\beta}\right)^{-\frac{d-k}{2}}\!\!
      \mathrm{e}^{-\frac{\beta|G_x(y)|^2}{2}}\,\sigma_\Sigma(dy)\,,
\end{aligned}
  \label{rhs-f-eps-pu}
\end{equation}
where the last equality follows from the fact that the support of~$\rho_i$ is contained in $\mathcal{Q}^{(i)}\subseteq 
\mbox{Im}\Psi_x \setminus C_x$. Combining~\eqref{lfs-f-eps-pu-1},
\eqref{lfs-f-eps-pu-2}, and~\eqref{rhs-f-eps-pu}, we obtain the following equality for any $\epsilon > 0$: 
    \begin{equation}
      \int_{\Sigma} f^\epsilon(y)\, ((\Psi_x)_{\widetilde{\#}}\gamma)(dy) = 
      \int_{\Sigma}
      f^\epsilon(y) \,\left|
      \det D G_x(y)\right|\left(\frac{2\pi}{\beta}\right)^{-\frac{d-k}{2}}\!\!
      \mathrm{e}^{-\frac{\beta|G_x(y)|^2}{2}}\,
      \mathbf{1}_{\mbox{Im}\Psi_x \setminus C_x}(y)\,
      \sigma_\Sigma(dy)\,.
\label{int-identity-pushfwd-f-eps}
    \end{equation}

    To obtain~\eqref{int-identity-pushfwd} hence~\eqref{proposal-density-q-push-forward-by-psi-repeat}, we take the limit $\epsilon \rightarrow 0$ in \eqref{int-identity-pushfwd-f-eps}. Since $f^\epsilon(y) = f(y)\mathbf{1}_{\Sigma\setminus \mathcal{Q}^\epsilon}(y)$ for all~$y\in \Sigma$, 
we have $|f^\epsilon| \leq |f|$ and, using \eqref{properties-of-sets-q-eps},
it is easy to show that $$\lim_{\epsilon \rightarrow 0} f^\epsilon(y) = 
f(y) \lim_{\epsilon \rightarrow 0} (1 - \mathbf{1}_{\mathcal{Q}^\epsilon}(y)) = 
f(y) (1 - \mathbf{1}_{\Sigma\setminus (\mbox{Im}\Psi_x\setminus C_x)}(y)) =
f(y)\mathbf{1}_{\mbox{Im}\Psi_x\setminus C_x}(y)$$ for all $y\in \Sigma$.
The result then follows from the dominated convergence theorem and the fact
that the set $\Sigma\setminus (\mbox{Im}\Psi_x\setminus C_x)$ has measure zero under the measure $(\Psi_x)_{\widetilde{\#}}\gamma$. 
\end{proof}

\subsection{Proofs of the results of Section~\ref{subsec-hmc}}
\label{sec-proofs-HMC}

We prove Propositions~\ref{prop-differentiability-hmc} and~\ref{prop-set-f-hmc},
then Lemma~\ref{lemma-connection-reversibility-and-rmr}, and finally Theorem~\ref{thm-algo-2-rattle}.

\begin{proof}[Proof of Proposition~\ref{prop-differentiability-hmc}]
  The proof of Proposition~\ref{prop-differentiability-hmc} is similar to the proof of Proposition~\ref{prop-differentiability-mala}. To write the result, we introduce a matrix~$U_{M,x}\in \mathbb{R}^{d\times (d-k)}$ whose columns
form an orthonormal basis of $T^*_x\Sigma$ with respect to
$\langle\cdot,\cdot\rangle_{M^{-1}}$, i.e.\ $U_{M,x}^TM^{-1}U_{M,x}=I_{d-k}$
and $\nabla\xi(x)^TM^{-1}U_{M,x}=0$. For a given point~$x\in \Sigma$, the map $x'
  \mapsto U_{M,x'}$ can be assumed to be smooth in a neighborhood of $x$ by a
  construction similar to the one performed in
  Section~\ref{sec:construction_set_valued_MALA}. Since $U_{M,x}U_{M,x}^T
  M^{-1} = P_M(x)$ (where $P_M$ is the projection defined by \eqref{projection-pm}), the map $x \mapsto U_{M,x}$ allows to rephrase the function~$G_{M,x}$ in~\eqref{Gx-hmc} as
  \begin{equation}
  G_{M,x}(y) = \frac{1}{\tau} U_{M,x}U_{M,x}^T\left(y-x+\frac{\tau^2}{2} M^{-1}\nabla \overline{V}(x)\right).
    \label{gx-hmc-rewrite}
  \end{equation}
Note that the map $G_{M,x}$ seen as a function from $\Sigma$ to $T^*_x\Sigma$ does not depend on the choice of the matrix $U_{M,x}$. For a given $x \in \Sigma$, we can then verify, with an argument similar to the one used to write~\eqref{equivalent-definitions-of-cx}, that
  \begin{align}
    C_{M,x} = \left\{y\in \Sigma\,\middle|\, \det\left(\nabla\xi(y)^TM^{-1}\nabla\xi(x)\right)=0\right\} = \left\{y\in \Sigma\,\middle|\, \det\left(U^T_{M,x}M^{-1}U_{M,y}\right)=0\right\} \,.
    \label{set-c-m-x-two-expressions}
  \end{align}
  Therefore, \eqref{gx-hmc-rewrite} and the rightmost equality in \eqref{set-c-m-x-two-expressions} imply
  that $C_{M,x}$ is the set of critical points of the map $G_{M,x}$ defined in \eqref{Gx-hmc}.

  The first item of Proposition~\ref{prop-differentiability-hmc} can be obtained by applying the implicit function theorem to
  the equation $g(\bar{z},y)=0$ starting from the solution
  $g(z,x^1)=0$, where the map $g: T^*\Sigma\times \Sigma \rightarrow
  \mathbb{R}^{d-k}$, defined as $g(\bar{z},y)=U_{M,\bar{x}}^TM^{-1}(G_{M,\bar{x}}(y)-\bar{p})$ for
  $\bar{z}=(\bar{x},\bar{p})\in T^*\Sigma$ and $y \in \Sigma$, is
  $C^1$-differentiable (recall that $\bar{x}\mapsto U_{M,\bar{x}}\in \mathbb{R}^{d\times (d-k)}$ is locally $C^1$-differentiable). In fact, using \eqref{gx-hmc-rewrite}, we have 
  \begin{align*}
    g(\bar{z},y)=U_{M,\bar{x}}^TM^{-1}(G_{M,\bar{x}}(y)-\bar{p}) =
    \frac{1}{\tau} U_{M,\bar{x}}^T\left(y-\bar{x} + \frac{\tau^2}{2}
    M^{-1}\nabla \overline{V}(\bar{x})-\tau M^{-1}\bar{p}\right)\,.
  \end{align*}
  Since $x^1\in \Sigma\setminus C_{M,x}$, a simple computation shows that the differential of~$g$ in the $y$-variable
  is invertible at $(\bar{z}, y) = (z, x^1)$,
  which allows to express~$y$ as a $C^1$-differentiable function $\Upsilon_1$
  of~$\bar{z}$ with the desired properties. The uniqueness of $\Upsilon_1$ is implied by the implicit function theorem. This proves the first item.

  Concerning the second item, using the facts that 
  $\Upsilon_1(z)=x^1$ and $G_{M,\bar{x}}(\Upsilon_1(\bar{z}))=\bar{p}$ for all $\bar{z}=(\bar{x},
  \bar{p})\in \mathcal{O}$, as well as 
  \eqref{F-map-related-to-rattle}, \eqref{lambda-p-exp} and~\eqref{Gx-hmc}, 
  one can verify the claim regarding the Lagrange multiplier functions in
  \eqref{multiplier-c-explicit-phase-space} by direct computations.
    Assume that $\widetilde{\lambda}_x,
    \widetilde{\lambda}_p:\mathcal{O}'\rightarrow \mathbb{R}^k$ is another pair of Lagrange multiplier functions.
    From \eqref{F-map-related-to-rattle} we can compute (see \eqref{lambda-x-by-x1})
	$$\forall\bar{z}=(\bar{x}, \bar{p}) \in \mathcal{O}', \quad
	\widetilde{\lambda}_x(\bar{z}) = \frac{1}{\tau}
	\left[(\nabla\xi^TM^{-1}\nabla\xi)(\bar{x})\right]^{-1}\nabla\xi(\bar{x})^T\left[\widetilde{\Upsilon}_1(\bar{z}) - \bar{x} - \tau M^{-1}\bar{p} + \frac{\tau^2}{2} M^{-1} \nabla
	\overline{V}(\bar{x})\right]\,,$$
	where $\widetilde{\Upsilon}_1(\bar{z})=\Pi(F(\bar{z},
	\widetilde{\lambda}_x(\bar{z}), \widetilde{\lambda}_p(\bar{z})))$ and
	$\Pi$ is the projection in \eqref{projection-map-pi}.
	Moreover, $\widetilde{\Upsilon}_1(z)=x^1$ and
	$G_{M,\bar{x}}(\widetilde{\Upsilon}_1(\bar{z}))=\bar{p}$ for all $\bar{z}=(\bar{x}, \bar{p})\in \mathcal{O}'$.
  The uniqueness of the function $\Upsilon_1$ in the first item then implies 
  that $\widetilde{\Upsilon}_1= \Upsilon_1$ on $\mathcal{O}\cap \mathcal{O}'$,
  from which we obtain $\widetilde{\lambda}_x = \lambda_x$. The uniqueness of
  $\lambda_p$ is then a consequence of the fact that $\lambda_p$ is determined once $\lambda_x$ is given (see \eqref{lambda-p-exp}). This proves the second item.

  For the third item, note that the map $\Upsilon=F(\cdot, \lambda_x(\cdot), \lambda_p(\cdot))$, where the Lagrange multiplier functions are given in \eqref{multiplier-c-explicit-phase-space},
  is the one-step RATTLE with momentum reversal. The symplecticity of the
  RATTLE scheme (see~\cite{leimkuhler1994symplectic} and Section VII.1.4 of \cite{geometric-integration}) implies that $|\det
  D\Upsilon|\equiv 1$. Using this fact, we can find a neighborhood
  $\mathcal{O}$ of $z$ such that $\Upsilon$ is a $C^1$-diffeomorphism on~$\mathcal{O}$. Let us finally show that $\mathcal{O}$ can be chosen such that $\Upsilon(\bar{z})$ is the only element of $\mathcal{F}(\bar{z})$ on $\Upsilon(\mathcal{O})$ for all $\bar{z} \in \mathcal{O}$.
   Consider the map $G_M: \Sigma\times \Sigma\rightarrow T^*\Sigma$, defined by 
  $G_M(\bar{x},\bar{y})=(\bar{x}, G_{M,\bar{x}}(\bar{y}))$, for all $(\bar{x},\bar{y})\in \Sigma\times \Sigma$.
From the definition \eqref{Gx-hmc} of $G_{M,x}$, Assumption~\ref{assump-xi}, and the definition \eqref{projection-pm} of $P_M$, we know that $G_M$ is $C^1$-differentiable.
 Computing the differential of $G_M$ and using \eqref{gx-hmc-rewrite}--\eqref{set-c-m-x-two-expressions}, we find that $(x,x^1)$ is a regular point
 of $G_M$ since $x^1\not\in C_{M,x}$. Therefore, $G_M$ is a $C^1$-diffeomorphism (one-to-one) in a neighborhood of $(x,x^1)$.
Assume by contradiction that there is no neighborhood $\mathcal{O}$ such that $\Upsilon(\bar{z})$ is the only element of $\mathcal{F}(\bar{z})$ on
  $\Upsilon(\mathcal{O})$ for all $\bar{z} \in \mathcal{O}$. Then, we can find a sequence
  $z^{(i)}=(x^{(i)}, p^{(i)})\rightarrow z$ as $i\rightarrow +\infty$, such
  that for each $i \ge 1$ there are two different elements
  $\widetilde{z}^{(i),j}=(\widetilde{x}^{(i),j},\widetilde{p}^{(i),j})\in \mathcal{F}(z^{(i)})$ with $j=1,2$, and $\widetilde{z}^{(i),j}\rightarrow z'=(x^1,p^{1,-})$ as $i\rightarrow +\infty$, for $j=1,2$.
  In particular, for the position variable, we have 
  $\widetilde{x}^{(i),1} \neq \widetilde{x}^{(i),2}$ (since otherwise the momenta would be the same as well by~\eqref{lambda-p-exp}--\eqref{lambda-x-by-x1}) and it holds that $G_M(x^{(i)}, \widetilde{x}^{(i),j})=z^{(i)}$ for $j=1,2$, where
  $(x^{(i)},\widetilde{x}^{(i),j})\rightarrow (x,x^1)$ as $i\to+\infty$. However, this is in
  contradiction with the fact that $G_M$ is a $C^1$-diffeomorphism in a neighborhood of $(x,x^1)$. This shows the assertion in the last item.
\end{proof}

\medskip

\begin{proof}[Proof of Proposition~\ref{prop-set-f-hmc}]
The proof of Proposition~\ref{prop-set-f-hmc} is similar to the proof of Proposition~\ref{prop-set-f-mala}. 

  Let us first show that $\mathcal{N}$ is a closed set of measure zero. For a sequence $(x^{(i)},p^{(i)})$ in~$\mathcal{N}$ which converges to
  $(x,p) \in T^*\Sigma$, there exists $y^{(i)} \in C_{M,x^{(i)}}$ such
  that $p^{(i)} = G_{M,x^{(i)}}(y^{(i)})$, where $G_{M, x}$ is the map defined
  in \eqref{Gx-hmc}. Since $\Sigma$ is compact, $y^{(i)}$ converges upon
  extraction to $y\in \Sigma$. Note that $y \in C_{M,x}$, since 
  $$\det \left(\nabla\xi(y)^TM^{-1}\nabla\xi(x)\right)=\lim_{i\rightarrow +\infty}
  \det \left(\nabla\xi(y^{(i)})^TM^{-1}\nabla\xi(x^{(i)})\right) = 0,$$ and that $p=G_{M,x}(y)$. 
  This implies that $(x,p)\in \mathcal{N}$ and therefore $\mathcal{N}$ is a closed set. 
  Also, as shown in the proof of Proposition~\ref{prop-differentiability-hmc}
  (see \eqref{set-c-m-x-two-expressions}), $C_{M,x}$ is the set of critical
  points of the map $G_{M,x}$. As a result, $\mathcal{N}\subset T^*\Sigma$ has measure zero in $T^*\Sigma$, since Sard's theorem implies that
  $G_{M,x}(C_{M,x})$ has zero Lebesgue measure in $T^*_x\Sigma$ for all $x\in
  \Sigma$. 
  The equality \eqref{set-n-alternative} follows directly from the definitions
  of the set $C_{M,x}$ in~\eqref{set-cm-x} and the map $G_{M,x}$ in \eqref{Gx-hmc}.
  This proves the assertion in the first item.

  Concerning the second item, we argue similarly as in the proof of the second item of Proposition~\ref{prop-set-f-mala}.
Assume by contradiction that, for some $z=(x,p)\in T^*\Sigma\setminus \mathcal{N}$, there are infinitely many elements $z_{1}, z_{2},\ldots$ in
  $\mathcal{F}(z)$.  Denote by $y^{(i)}=\Pi(z_i)$ the position variables of $z_i$ for $i \ge 1$ (see \eqref{projection-map-pi}).
  Since $z_i \in \mathcal{F}(z)$, we have $G_{M,x}(y^{(i)})=p$ for all $i \ge 1$.
  Using the fact that the momenta $p^{1,-}$ are determined given the position
  variables $x^1$ of the output state $(x^1, p^{1,-})$ of the one-step RATTLE scheme with momentum reversal~\eqref{psi-map-by-rattle}
  (see \eqref{lambda-p-exp}--\eqref{lambda-x-by-x1}), we know that there are
  infinitely many different (position) elements in $\{y^{(1)}, y^{(2)}, \ldots\}\subset \Sigma$.
  Since $\Sigma$ is compact, we can assume (upon extracting a subsequence)
  that $\lim_{i\rightarrow +\infty} y^{(i)}=y\in \Sigma$ and, therefore,
  $G_{M,x}(y)=\lim_{i\rightarrow +\infty} G_{M,x}(y^{(i)})=p$.
  The fact that $z \in T^*\Sigma \setminus \mathcal{N}$ implies that we must have $y\not\in C_{M,x}$.
  Therefore, $G_{M,x}$ is a local $C^1$-diffeomorphism in a neighborhood of~$y$ (the set $C_{M,x}$ is the set of critical points of $G_{M,x}$; see \eqref{set-c-m-x-two-expressions}).
  However, this is in contradiction with the fact that there are infinitely many
  different elements $y^{(1)}, y^{(2)}, \ldots$ arbitrarily close to $y$ such that $G_{M,x}(y^{(i)})=p$ for all $i \ge 1$.
  This shows the assertion in the second item.

  Concerning the third item, we use the same argument as in the proof of the third item of Proposition~\ref{prop-set-f-mala}.
  Assuming that $z=(x,p)\in T^*\Sigma\setminus \mathcal{N}$ and
  $\mathcal{F}(z)=\{z_1, z_2, \ldots, z_n\}$ with $n =|\mathcal{F}(z)| \geq 1$, the third assertion of Proposition~\ref{prop-differentiability-hmc} implies that we can find a neighborhood~$\mathcal{O}$ of~$z$ and $n$ different pairs of 
  Lagrange multiplier functions $\lambda_x^{(i)}, \lambda_p^{(i)}:
  \mathcal{O}\rightarrow \mathbb{R}^k$, $1 \le i \le n$, such that
  $\Upsilon^{(i)}(\cdot)=F(\cdot, \lambda_x^{(i)}(\cdot),
  \lambda_p^{(i)}(\cdot))$ is a $C^1$-diffeomorphism from~$\mathcal{O}$ to the
  neighborhood $\Upsilon^{(i)}(\mathcal{O})$ of~$z_i$ in $T^*\Sigma$ and,
  moreover, $\Upsilon^{(i)}(\bar{z})$ is the only element of
  $\mathcal{F}(\bar{z})$ in $\Upsilon^{(i)}(\mathcal{O})$ for all $\bar{z}\in \mathcal{O}$ and $1 \le i \le n$.
  Since $z_1, z_2, \ldots, z_n$ are $n$ different elements, by shrinking the neighborhood $\mathcal{O}$ if necessary,  
  we can assume that $\Upsilon^{(i)}(\mathcal{O})\cap
  \Upsilon^{(i')}(\mathcal{O})=\emptyset$ for $1 \le i\neq i' \le n$.
  This shows that $\mathcal{F}(\bar{z})$ contains at least $n$ elements for
  all $\bar{z}\in \mathcal{O}$. 
  In particular, for the state $z$, we have $\mathcal{F}(z)=\{z_1, z_2, \ldots, z_n\}\subset \cup_{i=1}^n \Upsilon^{(i)}(\mathcal{O})$.
   Assume by contradiction that there are elements $z^{(j)}=(x^{(j)}, p^{(j)})\in \mathcal{O}$ such
   that $\lim_{j\rightarrow+\infty} z^{(j)}=z$ and $\mathcal{F}(z^{(j)})$ has
   more than $n$ elements for $j \ge 1$.  
   Then, we can find $\widetilde{z}^{(j)}= (\widetilde{x}^{(j)},
   \widetilde{p}^{(j)}) \in \mathcal{F}(z^{(j)})\setminus\{ 
   \Upsilon^{(1)}(z^{(j)}),\ldots,\Upsilon^{(n)}(z^{(j)})\}$ such that
   $\widetilde{z}^{(j)} \not\in\cup_{i=1}^n \Upsilon^{(i)}(\mathcal{O})$
   (since $\Upsilon^{(i)}(z^{(j)})$ is the only element of $\mathcal{F}(z^{(j)})$ in $\Upsilon^{(i)}(\mathcal{O})$).
   We also have $G_{M, x^{(j)}}(\widetilde{x}^{(j)}) = p^{(j)}$ for $j \ge 1$.
   Since $\Sigma$ is compact, we can assume upon extracting a subsequence that $\lim_{j\rightarrow +\infty} \widetilde{x}^{(j)} = \widetilde{x}$.
  Using the fact that the momenta are determined (continuously) by the position variables in the output of
  the one-step RATTLE scheme with momentum reversal \eqref{psi-map-by-rattle} (see \eqref{lambda-p-exp}--\eqref{lambda-x-by-x1}),
  we have that $\lim_{j\rightarrow +\infty}
  \widetilde{z}^{(j)}=(\widetilde{x}, \widetilde{p}) \not\in\cup_{i=1}^n \Upsilon^{(i)}(\mathcal{O})$, for some momenta $\widetilde{p}$.
   This, together with the convergence of $(x^{(j)}, p^{(j)})$ to $(x,p)$,
   imply that $G_{M,x}(\widetilde{x})=p$ and, therefore, $(\widetilde{x}, \widetilde{p}) \in \mathcal{F}(z)$.
   However, this is in contradiction with the fact that there are only $n$
   elements in $\mathcal{F}(z)$ and that $\mathcal{F}(z)=\{z_1, z_2, \ldots, z_n\}\subset \cup_{i=1}^n \Upsilon^{(i)}(\mathcal{O})$.  This proves the assertion in the third item.

  Concerning the fourth item, the third item above implies that the subsets $\mathcal{B}_{i}$ are open for $i \geq 1$.
  By proceeding in the same way as at the end of the proof of Proposition~\ref{prop-set-f-mala}, it is easy to show that $\mathcal{B}_{0}$ is open, and that $\mathcal{N}$, $\mathcal{B}_{0}$, $\mathcal{B}_{1}, \ldots$ are disjoint subsets which form a partition of~$T^*\Sigma$. 
\end{proof}

\medskip

\begin{proof}[Proof of Lemma~\ref{lemma-connection-reversibility-and-rmr}]
  Denote by $\widetilde{q}(z,dz')$ the transition probability kernel of the Markov chain~$\widetilde{\mathscr{C}}$.
  Then, $\widetilde{q}(z,\cdot)=\mathcal{R}_\#(q(z,\cdot))$ for $z\in T^*\Sigma$, where
  $\mathcal{R}_\#(q(z,\cdot))$ denotes the push-forward probability measure of $q(z,\cdot)$ by the involution map $\mathcal{R}$.
Assume that the Markov chain $\mathscr{C}$ is reversible. Using the definition
  of the push-forward measures~(see \cite[Section~5.2]{ambrosio2005gradient}), as well as
  $\mathcal{R}\circ\mathcal{R}=\mbox{id}$, it holds, for any bounded measurable function $f: T^*\Sigma \times T^*\Sigma \rightarrow \mathbb{R}$,
  \begin{equation*}
    \begin{aligned}
    \int_{T^*\Sigma\times T^*\Sigma} f(z,z')\, \widetilde{q}(z,dz')\,\mu(dz) & = \int_{T^*\Sigma\times T^*\Sigma} f(z,z')\, \mathcal{R}_\#(q(z,\cdot))(dz')\,\mu(dz) \\
      & = \int_{T^*\Sigma\times T^*\Sigma} f(z,\mathcal{R}(z'))\, q(z,dz')\,\mu(dz) \\
      & = \int_{T^*\Sigma\times T^*\Sigma} f(z', \mathcal{R}(z))\, q(z,dz')\,\mu(dz) \\
      & = \int_{T^*\Sigma\times T^*\Sigma} f(\mathcal{R}(z'), \mathcal{R}(z))\, \mathcal{R}_\#(q(z,\cdot))(dz')\,\mu(dz) \\
      & = \int_{T^*\Sigma\times T^*\Sigma} f(\mathcal{R}(z'),\mathcal{R}(z))\,
    \widetilde{q}(z,dz')\,\mu(dz) \,,
    \end{aligned}
  \end{equation*}
  which shows that the Markov chain $\widetilde{\mathscr{C}}$ is reversible up to momentum reversal.
  The proof of the converse statement (i.e.\ $\mathscr{C}$ is reversible
  if $\widetilde{\mathscr{C}}$ is reversible up to momentum reversal) is similar and is therefore omitted.
\end{proof}

\medskip

\begin{proof}[Proof of Theorem~\ref{thm-algo-2-rattle}]
  Let us denote by $\mathscr{C}$ the Markov chain $(z^{(i)})_{i\ge 0}$
  generated by Algorithm~\ref{algo-phase-space-mc-rattle}, and by $q(z,dz')$ its transition probability kernel.
  We prove that $\mathscr{C}$ is reversible up to momentum reversal by
  considering it as the composition of the transitions (in Algorithm~\ref{algo-phase-space-mc-rattle}) from $z^{(i)}$ to
  $z^{(i+\frac{1}{4})}$, then from $z^{(i+\frac{1}{4})}$ to $z^{(i+\frac{3}{4})}$, and finally from $z^{(i+\frac{3}{4})}$ to $z^{(i+1)}$.

  Denote by $q_1$, $q_2$, $\widetilde{q}_2$ 
  the transition probability kernels which correspond to the transitions from
  $z^{(i)}$ to $z^{(i+\frac{1}{4})}$, from $z^{(i+\frac{1}{4})}$ to
  $z^{(i+\frac{2}{4})}$, and from $z^{(i+\frac{1}{4})}$ to
  $z^{(i+\frac{3}{4})}$, respectively. Denote also by $\mathscr{C}_1$,
  $\mathscr{C}_2$, $\widetilde{\mathscr{C}}_2$ the corresponding Markov
  chains. The transition probability kernel~$q$ of the whole Markov chain~$\mathscr{C}$ is 
  then obtained as the composition of the transition probability kernels of the Markov chains~$\mathscr{C}_1$,
  $\widetilde{\mathscr{C}}_2$, $\mathscr{C}_1$, i.e.\
  \begin{equation}
    q(z,dz') = \int_{(z_1,z_2) \in T^*\Sigma\times T^*\Sigma}
    q_1(z,dz_1)\,\widetilde{q}_2(z_1, dz_2)\, q_1(z_2, dz')\,, \quad z,z'\in
    T^*\Sigma\,.
    \label{q-by-q1-q2-q1}
  \end{equation}
  Recall the definition of reversibility in Definition~\ref{def-reversible} and the definition of reversibility up to momentum reversal in Definition~\ref{def-rmr}.
  We state the following two claims (C1)--(C2):
\begin{enumerate}
  \item[(C1)]
    The Markov chain $\mathscr{C}_1$ is both reversible and reversible up to momentum reversal with respect to~$\mu$.
  \item[(C2)]
  The Markov chain $\widetilde{\mathscr{C}}_2$ is reversible up to momentum reversal with respect to $\mu$.
\end{enumerate}
Also, since $\widetilde{\mathscr{C}}_2$ is the Markov chain $\mathscr{C}_2$ followed by momentum reversal,
Lemma~\ref{lemma-connection-reversibility-and-rmr} implies that Claim~(C2) is equivalent to the following claim:
  \begin{enumerate}
    \item[(C2')] The Markov chain $\mathscr{C}_2$ is reversible with respect to $\mu$.
  \end{enumerate}

  The fact that the Markov chain~$\mathscr{C}$ is reversible up
  to momentum reversal with respect to $\mu$ is a consequence of Claims (C1)--(C2).
Indeed, using \eqref{q-by-q1-q2-q1}, Fubini's theorem, and the invariance of~$\mu$ under~$\mathcal{R}$,
  we obtain, for any bounded measurable function $f: T^*\Sigma \times T^*\Sigma \rightarrow \mathbb{R}$, 
    \begin{align*}
  & \int_{T^*\Sigma\times T^*\Sigma} f(z,z')\, q(z,dz')\,\mu(dz) \\
  & \qquad = \int_{T^*\Sigma\times T^*\Sigma\times T^*\Sigma\times T^*\Sigma} f(z,z')\, q_1(z,dz_1)\, \widetilde{q}_2(z_1, dz_2)\, q_1(z_2,dz') \,\mu(dz) \\
  & \qquad = \int_{T^*\Sigma\times T^*\Sigma\times T^*\Sigma\times T^*\Sigma}
	f(\mathcal{R}(z_1),z')\, q_1(z,dz_1)\, \widetilde{q}_2(\mathcal{R}(z), dz_2)\, q_1(z_2,dz') \,\mu(dz) \\
  & \qquad = \int_{T^*\Sigma\times T^*\Sigma\times T^*\Sigma\times T^*\Sigma}
	f(\mathcal{R}(z_1),z')\, q_1(\mathcal{R}(z),dz_1)\, \widetilde{q}_2(z, dz_2)\, q_1(z_2,dz') \,\mu(dz) \\
  & \qquad = \int_{T^*\Sigma\times T^*\Sigma\times T^*\Sigma\times T^*\Sigma}
	f(\mathcal{R}(z_1),z')\, q_1(z_2,dz_1)\, \widetilde{q}_2(z, dz_2)\,
	q_1(\mathcal{R}(z),dz') \,\mu(dz) \\
  & \qquad = \int_{T^*\Sigma\times T^*\Sigma\times T^*\Sigma\times T^*\Sigma}
	f(\mathcal{R}(z_1),z')\, q_1(z_2,dz_1)\, \widetilde{q}_2(\mathcal{R}(z), dz_2)\,
	q_1(z,dz') \,\mu(dz) \\
  & \qquad = \int_{T^*\Sigma\times T^*\Sigma\times T^*\Sigma\times T^*\Sigma}
	f(\mathcal{R}(z_1),\mathcal{R}(z))\, q_1(z_2,dz_1)\, \widetilde{q}_2(z', dz_2)\,
	q_1(z,dz') \,\mu(dz) \\
  & \qquad = \int_{T^*\Sigma\times T^*\Sigma}
	f(\mathcal{R}(z'),\mathcal{R}(z))\, q(z, dz') \,\mu(dz) \,,
    \end{align*}
  where we applied successively: the reversibility up to momentum reversal on
  $(z,z_1)$ (thanks to Claim~(C1)); a change of variable from $z$ to $\mathcal{R}(z)$; 
  the reversibility up to momentum reversal on $(z,z_2)$ (thanks to~Claim (C2)); 
  a change of variable from $z$ to $\mathcal{R}(z)$; 
  the reversibility up to momentum reversal on $(z,z')$ (thanks to~Claim (C1)). 

  Let us now conclude by proving Claim (C1) and Claim (C2'); starting with~(C1). Note that, in view of~\eqref{eq:kappa} and~\eqref{update-momenta}, it holds, for $z=(x,p)$ and $z'=(x', p')$ in~$T^*\Sigma$,
\begin{equation}
  q_1(z, dz')=\left(\frac{2\pi(1-\alpha^2)}{\beta}\right)^{-\frac{d-k}{2}} \exp
  \left(-\frac{\beta (p'-\alpha p)^T M^{-1} (p'-\alpha p)}{2(1-\alpha^2)}\right)\sigma^{M^{-1}}_{T^*_x
  \Sigma}(dp')\,\delta_x(dx')\,.
  \label{kernel-q1}
\end{equation}
  Therefore, using \eqref{eq:mu_tensor}, \eqref{eq:kappa} and \eqref{kernel-q1}, we can compute
\begin{equation}
  \label{q1-times-mu}
  \begin{aligned}
    q_1(z,dz')\,\mu(dz) & = \left(2\pi\beta^{-1}\right)^{-(d-k)} 
    (1-\alpha^2)^{-\frac{d-k}{2}}
    \exp \left(-\frac{\beta (\langle p', p'\rangle_{M^{-1}} -2\alpha \langle
    p,p'\rangle_{M^{-1}} + \langle p, p\rangle_{M^{-1}})}{2(1-\alpha^2)}\right)\\
    & \qquad \times \sigma^{M^{-1}}_{T^*_x \Sigma}(dp')\, \sigma^{M^{-1}}_{T^*_x
    \Sigma}(dp)\,\delta_x(dx')\,\nu_{\Sigma}^M(dx)\,,
\end{aligned}
\end{equation}
  where $\nu_{\Sigma}^M$ is defined in~\eqref{eq:def_nu_M} and  $\langle \cdot,\cdot \rangle_{M^{-1}}$ is defined in~\eqref{eq:scalar_product_M}.
  Using \eqref{q1-times-mu}, it is straightforward to verify that, for any
  bounded measurable function $f: T^*\Sigma\times T^*\Sigma \rightarrow \mathbb{R}$,  
\begin{equation}
  \begin{aligned}
  \int_{T^*\Sigma\times T^*\Sigma} f(z,z')\, q_1(z,dz')\,\mu(dz) & = \int_{T^*\Sigma\times T^*\Sigma} f(z',z)\, q_1(z,dz')\,\mu(dz) \\
   & = \int_{T^*\Sigma\times T^*\Sigma} f(\mathcal{R}(z'),\mathcal{R}(z))\, q_1(z,dz')\,\mu(dz) \,.
  \end{aligned}
\end{equation}
This shows that $\mathscr{C}_1$ is both reversible and reversible up to momentum reversal with respect to~$\mu$. Therefore, Claim (C1) is proved.

    It remains to prove Claim (C2'). According to Algorithm~\ref{algo-phase-space-mc-rattle}, we have
\begin{equation}
  \begin{aligned}
    q_2(z, dz') =&\sum_{j=1}^{|\Phi(z)|} \omega(z_j\,|\,z)\,
    \mathbf{1}_{\mathcal{D}_\Phi}(z,z_j)\,a(z_j\,|\,z)\,\delta_{z_j}(dz') \\
    &+ \left[1- \sum_{j=1}^{|\Phi(z)|} \omega(z_j\,|\,z)\,
    \mathbf{1}_{\mathcal{D}_\Phi}(z,z_j)\,a(z_j\,|\,z)\right] \delta_{z}(dz')\,, 
  \end{aligned}
    \label{q-2nd-step-phase-space}
  \end{equation}
  where $\mathcal{D}_\Phi$ is the set defined in \eqref{admissible-set-d-r}, $\Phi(z)=\{z_1, \ldots, z_n\}$, with $n=|\Phi(z)|$, $a(\cdot\,|\,z)$ is
  the acceptance probability defined by~\eqref{rate-phase-space-rattle}, and $\delta_z$ the Dirac measure centered at~$z \in T^*\Sigma$.
  Here and in the following, we adopt the convention that $\sum_{j=1}^{0} \cdot = 0$.
Claim (C2') is equivalent to the fact that, for any bounded measurable function $f: T^*\Sigma\times T^*\Sigma\rightarrow \mathbb{R}$,
  \begin{align}
    \int_{T^*\Sigma\times T^*\Sigma} f(z,z')\,q_2(z,dz')\,\mathrm{e}^{-\beta H(z)} \sigma_{T^*\Sigma}(dz) =
    \int_{T^*\Sigma\times T^*\Sigma} f(z',z)\,q_2(z,dz')\,\mathrm{e}^{-\beta H(z)} \sigma_{T^*\Sigma}(dz) \,.
    \label{q-does-not-change-rho}
  \end{align}
  Using \eqref{q-2nd-step-phase-space}, the integral on the left hand side
  above can be written as $I_1 + I_2$, with 
  \[
    \begin{aligned}
      I_1 & = \int_{T^*\Sigma}\,  \left[\sum_{z'\in \Phi(z)}
      \omega(z'\,|\,z)\,\mathbf{1}_{\mathcal{D}_\Phi}(z, z')\,a(z'\,|\,z)\,
      f(z, z')\,\right]\,\mathrm{e}^{-\beta H(z)}\,\sigma_{T^*\Sigma}(dz)\,, \\
    I_2 & = \int_{T^*\Sigma}\,
      \left[1-\sum_{z'\in \Phi(z)}\omega(z'\,|\,z)\,\mathbf{1}_{\mathcal{D}_\Phi}(z,z')\,a(z'\,|\,z)\right]\,f(z,z)\, \mathrm{e}^{-\beta H(z)}\,\sigma_{T^*\Sigma}(dz)\,.
  \end{aligned}
  \]
  The expression~\eqref{rate-phase-space-rattle} of the acceptance probability implies that
  \[
      I_1  = \int_{T^*\Sigma}\left[\sum_{z'\in \Phi(z)}\min\Big\{
      \omega\left( z'\,|\,z\right)\,\mathrm{e}^{-\beta H(z)},\,
    \omega\left(z\,|\,z'\right) \mathrm{e}^{-\beta H(z')} \,\Big\}\,
  \mathbf{1}_{\mathcal{D}_\Phi}\left(z,z'\right)
f(z, z')\right] \,\sigma_{T^*\Sigma}(dz)\,. 
  \]
  In the same way, the integral on the right hand side of \eqref{q-does-not-change-rho} can be written as the sum
  of $\widetilde{I}_1$ and~$I_2$, where 
   \[
     \widetilde{I}_1  = \int_{T^*\Sigma}\left[\sum_{z'\in \Phi(z)}\min\Big\{
      \omega\left( z'\,|\,z\right)\,\mathrm{e}^{-\beta H(z)},\,
    \omega\left(z\,|\,z'\right) \mathrm{e}^{-\beta H(z')} \,\Big\}\,
  \mathbf{1}_{\mathcal{D}_\Phi}\left(z,z'\right)
f(z', z)\right] \,\sigma_{T^*\Sigma}(dz)\,.
  \]
Therefore, it suffices to prove that $I_1=\widetilde{I}_1$.

To proceed, recall that $\Pi$ is the projection map defined in \eqref{projection-map-pi}.
Let us introduce the set
  \begin{equation}
  \begin{aligned}
    \widetilde{\mathcal{D}}_\Phi = \Big\{(z,z')\in T^*\Sigma\times T^*\Sigma~\Big|&\,z'\in \Phi(z)\,,~z\in
    \Phi(z')\,,\\
    &~\det\left[\nabla\xi(\Pi(z'))^TM^{-1}\nabla\xi(\Pi(z))\right]\neq 0 ~\Big\}\,,
  \end{aligned}
  \label{admissible-set-d-r-subset}
\end{equation}
  which is a subset of the set~$\mathcal{D}_\Phi$ defined in
  \eqref{admissible-set-d-r}. It plays a role similar to the set $\widetilde{\mathcal D}$ defined in~\eqref{d-bar-set-mala}. 
  
  Note that, if $z'\in\Phi(z)$ and $\Pi(z')\in C_{M, \Pi(z)}$, then $z \in
  \mathcal{N}$. Therefore,
  $\mathbf{1}_{\mathcal{D}_\Phi}(z,z')=\mathbf{1}_{\widetilde{\mathcal{D}}_\Phi}(z,z')$,
  up to elements~$z$ in the zero measure set $\mathcal{N}$ (see
  Proposition~\ref{prop-set-f-hmc}). Using this fact, we get
  \begin{equation}
    I_1 = \int_{T^*\Sigma}\left[\sum_{z'\in \Phi(z)}\min\Big\{
      \omega\left( z'\,|\,z\right)\,\mathrm{e}^{-\beta H(z)},\,
    \omega\left(z\,|\,z'\right) \mathrm{e}^{-\beta H(z')} \,\Big\}\,
  \mathbf{1}_{\widetilde{\mathcal{D}}_\Phi}\left(z,z'\right) f(z, z')\right]
    \,\sigma_{T^*\Sigma}(dz)\,. 
    \label{formula-of-i1}
  \end{equation}

Although the maps $\Phi^{(j)}$ in Assumption~\ref{assump-phi-phase-space} are only locally defined, a localization argument similar to the one used in the proof of Lemma~\ref{lemma-on-push-forward-by-psi} allows us to derive in Lemma~\ref{lemma-change-of-variable-phi} below a formula for a change of variables involving the maps~$\Phi^{(j)}$. Applying Lemma~\ref{lemma-change-of-variable-phi}, we obtain from \eqref{formula-of-i1} that 
\begin{align*}
  I_1 =& \int_{T^*\Sigma}\left[\sum_{z'\in \Phi(z)}\min\Big\{
      \omega\left( z\,|\,z'\right)\,\mathrm{e}^{-\beta H(z')},\,
    \omega\left(z'\,|\,z\right) \mathrm{e}^{-\beta H(z)} \,\Big\}\,
  \mathbf{1}_{\widetilde{\mathcal{D}}_\Phi}\left(z,z'\right)
f(z',z)\right] \,\sigma_{T^*\Sigma}(dz) \\
=& \int_{T^*\Sigma}\left[\sum_{z'\in \Phi(z)}\min\Big\{
      \omega\left( z'\,|\,z\right)\,\mathrm{e}^{-\beta H(z)},\,
    \omega\left(z\,|\,z'\right) \mathrm{e}^{-\beta H(z')} \,\Big\}\,
  \mathbf{1}_{\mathcal{D}_\Phi}\left(z,z'\right)
f(z',z)\right] \,\sigma_{T^*\Sigma}(dz) \\
     =&\, \widetilde{I}_1\,,
  \end{align*}
  where we used again the fact that $\mathbf{1}_{\mathcal{D}_\Phi}(z,z')=\mathbf{1}_{\widetilde{\mathcal{D}}_\Phi}(z,z')$,
  except for~$z$ in the zero measure set $\mathcal{N}$ (see Proposition~\ref{prop-set-f-hmc}), to derive the second equality.
  This completes the proof of Claim (C2').
\end{proof}

Finally, we present the proof of the following change of variables formula used in the proof of Theorem~\ref{thm-algo-2-rattle} above.
\begin{lemma}
  Let $\widetilde{\mathcal{D}}_\Phi$ be the set defined in 
  \eqref{admissible-set-d-r-subset}. For any bounded measurable function
  $f:T^*\Sigma\times T^*\Sigma \rightarrow~\mathbb{R}$, 
  \begin{equation}
       \int_{T^*\Sigma}\left[\sum_{z'\in \Phi(z)}
     \mathbf{1}_{\widetilde{\mathcal{D}}_\Phi}\big(z,z'\big)
     f\left(z, z'\right) \right]
     \,\sigma_{T^*\Sigma}(dz) = \int_{T^*\Sigma} \left[\sum_{z'\in \Phi(z)}
      \mathbf{1}_{\widetilde{\mathcal{D}}_\Phi}\big(z,z'\big)
     f\left(z', z\right) \right] \,\sigma_{T^*\Sigma}(dz)\,.
    \label{change-of-variables-for-phi}
  \end{equation}
  \label{lemma-change-of-variable-phi}
\end{lemma}

\begin{proof}
  For each $z\in T^*\Sigma$, there are $m$ elements $z_1, \cdots, z_m$ in~$\Phi(z)$, where
  $0 \le m \le |\Phi(z)|$, such that $\Pi(z_j)\in \Sigma\setminus C_{M,\Pi(z)}$ for $j \le m$. 
  In other words, $m= \big|\{z'\in \Phi(z)\,|\, \Pi(z')\not\in C_{M,\Pi(z)}\}\big|$.
  We write~$m(z)$ in the sequel in order to emphasize the dependence of this integer on~$z$. Define the sets $\widetilde{\mathcal{B}}_l=\{z\in T^*\Sigma\,|\,m(z)=l\}$, for $l\ge 0$.
  Clearly, the family $(\widetilde{\mathcal{B}}_l)_{l\ge 0}$ is pairwise
  disjoint and we have $T^*\Sigma=\cup_{l\ge 0} \widetilde{\mathcal{B}}_l$.
  Using the same argument as in the proof of Proposition~\ref{prop-set-f-hmc},
  we can show that the set~$\widetilde{\mathcal{B}}_l$ is open for $l \ge 1$.
  Therefore, $\widetilde{\mathcal{B}}_0= T^*\Sigma\setminus (\cup_{l\ge 1}
  \widetilde{\mathcal{B}}_l)$ is a closed subset of $T^*\Sigma$.
Note that, for $z,z'\in T^*\Sigma$ with $z=(x,p)$, the condition $z'\in \Phi(z)$ implies that $p=G_{M,x}(\Pi(z'))$ (see the discussion before~\eqref{Gx-hmc}). Therefore, $\widetilde{\mathcal{B}}_l \subset \{(x,G_{M,x}(y)), \ (x,y)\in \Sigma\times\Sigma\}$ for any $l \geq 1$. The compactness of $\Sigma$ and the continuity
  of the map~$(x,y)\mapsto G_{M,x}(y)$ then guarantee that $\cup_{l\ge 1} \widetilde{\mathcal{B}}_l$ is bounded.  In particular, the supports of the integrands on both sides of
  \eqref{change-of-variables-for-phi} are contained in the bounded subset
  $\cup_{l\ge 1} \widetilde{\mathcal{B}}_l$ (for $z\in
  \widetilde{\mathcal{B}}_0$, we have $m(z)=0$ and $\mathbf{1}_{\widetilde{\mathcal{D}}_\Phi}(z,z') =0$ for all $z'\in \Phi(z)$).
  Since \eqref{change-of-variables-for-phi} holds obviously when $\cup_{l\ge 1} \widetilde{\mathcal{B}}_l=\emptyset$, in the following we assume that $\cup_{l\ge 1} \widetilde{\mathcal{B}}_l\neq \emptyset$.

  Similarly to the proof of Lemma~\ref{lemma-on-push-forward-by-psi}, we start
  by constructing a partition of unity of the space~$T^*\Sigma$. 
For each $z \in \widetilde{\mathcal{B}}_l$, where $l \ge 1$, 
since $\widetilde{\mathcal{B}}_l$ is open, we can find a neighborhood $\mathcal{O}_z$ of $z$ and $l$ $C^1$-diffeomorphisms
  $\Phi^{(j)}$, $1 \le j \le l$, provided by Assumption~\ref{assump-phi-phase-space}, such
  that $\mathcal{O}_z \subseteq \widetilde{\mathcal{B}}_l$. In particular,
  $m(\bar{z})=l$ for all $\bar{z}\in \mathcal{O}_z$. 
  Since the images $\Phi^{(j)}(z)$ are different for different indices $j$, 
  by shrinking the neighborhood $\mathcal{O}_z$,
   we can assume without loss of generality that $\Phi^{(j)}(\mathcal{O}_z) \cap
  \Phi^{(j')}(\mathcal{O}_z) = \emptyset$, for $1 \le j \neq j' \le l$.
  Consider the map $\Upsilon^{(j)}$ given by the third item of
  Proposition~\ref{prop-differentiability-hmc}, which is a $C^1$-diffeomorphism from
  a neighborhood $\mathcal{O}_{z}'$ of $z$ to the neighborhood 
  $\Upsilon^{(j)}(\mathcal{O}_z')$ of $\Phi^{(j)}(z)$.  Using the fact that
  $\Upsilon^{(j)}(\bar{z})$ is the only element of $\mathcal{F}(\bar{z})$ such
  that $\Upsilon^{(j)}(\bar{z}) \in \Upsilon^{(j)}(\mathcal{O}_z')$ for all
  $\bar{z}\in \mathcal{O}_z'$ (see Proposition~\ref{prop-differentiability-hmc}), we can verify that 
  $\Phi^{(j)}(\bar{z}) =
  \Upsilon^{(j)}(\bar{z})$ for all $\bar{z}$ in the neighborhood $\mathcal{O}_z\cap (\Phi^{(j)})^{-1}(\Upsilon^{(j)}(\mathcal{O}_z'))$ of $z$. For simplicity, we reuse the notation and denote again by $\mathcal{O}_z$ the neighborhood of $z$ such that $\Phi^{(j)}(\bar{z}) = \Upsilon^{(j)}(\bar{z})$, for all $\bar{z} \in \mathcal{O}_z$ and for $1 \le j \le l$. 
 The discussion above implies that 
  \begin{equation}
    \forall~\bar{z} \in \mathcal{O}_z, \quad \Big\{z'\in \Phi(\bar{z})\,\Big|\, \Pi(z')\not\in C_{M,\Pi(\bar{z})}\Big\} = 
    \Big\{\Phi^{(1)}(\bar{z}), \ldots, \Phi^{(l)}(\bar{z})\Big\} = 
    \Big\{\Upsilon^{(1)}(\bar{z}), \ldots, \Upsilon^{(l)}(\bar{z})\Big\}\,.
    \label{good-states-by-phi-and-Upsilon}
  \end{equation}
  Note that both $\Phi^{(j)}$ and $\Upsilon^{(j)}$ depend on the basis point $z$, although this is not made explicit in the notation. 

  For any $\epsilon>0$, 
  the set $\widetilde{\mathcal{B}}_0^\epsilon:= \{z\in T^*\Sigma\,|\,
  d_{T^*\Sigma}(z, \widetilde{\mathcal{B}}_0)<\epsilon\}$, where $d_{T^*\Sigma}$ denotes the Riemannian distance on $T^*\Sigma$,
  is a neighborhood of $\widetilde{\mathcal{B}}_0$, such that 
  \begin{equation}
 \forall\, 0 < \epsilon
    < \epsilon', \quad \widetilde{\mathcal{B}}_0\subseteq \widetilde{\mathcal{B}}^\epsilon_0
    \subseteq \widetilde{\mathcal{B}}^{\epsilon'}_0\,, \qquad \mbox{and} \qquad \bigcap_{\epsilon>0} \widetilde{\mathcal{B}}^{\epsilon}_0 = \widetilde{\mathcal{B}}_0\,,
      \label{properties-of-sets-b-eps-hmc}
  \end{equation}
  where the second identity holds because $\widetilde{\mathcal{B}}_0$ is a closed set. Let us now fix $\epsilon>0$. 
  Since $\cup_{l\ge 1} \widetilde{\mathcal{B}}_l$ is bounded and $\widetilde{\mathcal{B}}_0^\epsilon$ is open,
  $T^*\Sigma\setminus \widetilde{\mathcal{B}}_0^\epsilon
  \subset \cup_{l\ge 1} \widetilde{\mathcal{B}}_l$ is a compact subset of~$T^*\Sigma$. The union of the neighborhoods~$\mathcal{O}_z$ in~\eqref{good-states-by-phi-and-Upsilon} for
  $z\in \cup_{l\ge 1} \widetilde{\mathcal{B}}_l$ forms an open covering of $T^*\Sigma\setminus \widetilde{\mathcal{B}}_0^\epsilon$. 
  Since $T^*\Sigma\setminus \widetilde{\mathcal{B}}_0^\epsilon$ is compact, there exists a
finite subcovering, which we denote by $\mathcal{O}^{(1)}, \ldots, \mathcal{O}^{(n)}$, $n \ge 1$.
  We define $\mathcal{O}^{(0)}=\widetilde{\mathcal{B}}_0^\epsilon$. It is then clear that 
$\mathcal{O}^{(0)}, \mathcal{O}^{(1)}, \ldots, \mathcal{O}^{(n)}$ form a
finite open covering of the entire phase space $T^*\Sigma$. Let $\rho_0,
\rho_1, \ldots, \rho_n: T^*\Sigma\rightarrow [0,1]$ be a partition of unity
  corresponding to this open covering $\mathcal{O}^{(0)}, \mathcal{O}^{(1)}, \ldots, \mathcal{O}^{(n)}$,
such that the support of the continuous function $\rho_i$ is contained in~$\mathcal{O}^{(i)}$ and that $\sum_{i=0}^{n}\rho_i(z)=1$
for all $z \in T^*\Sigma$.

  Next, for $1\le i \le n$, we denote by $m_i$ the number $m(z)$ for $z\in
  \mathcal{O}^{(i)}$, since it is constant on~$\mathcal{O}^{(i)}$. 
  For each $i$, we denote by $\Upsilon^{(i,j)}$, where $1 \le j
  \le m_i$, the $m_i$ $C^1$-diffeomorphisms on $\mathcal{O}^{(i)}$ discussed above.
  With this notation, from \eqref{good-states-by-phi-and-Upsilon} we have 
  \begin{equation}
    \forall~\bar{z} \in \mathcal{O}_i, \qquad \Big\{z'\in \Phi(\bar{z})\,\Big|\, \Pi(z')\not\in C_{M,\Pi(\bar{z})}\Big\} = 
    \Big\{\Upsilon^{(i,1)}(\bar{z}), \ldots, \Upsilon^{(i,m_i)}(\bar{z})\Big\}\,.
    \label{good-states-by-Upsilon}
  \end{equation}
  We define $f^\epsilon(z,z')=f(z,z')\mathbf{1}_{T^*\Sigma\setminus
  \widetilde{\mathcal{B}}_0^\epsilon}(z) \mathbf{1}_{T^*\Sigma\setminus
  \widetilde{\mathcal{B}}_0^\epsilon}(z')
  =f(z,z') \mathbf{1}_{T^*\Sigma\setminus \mathcal{O}^{(0)}}(z)
  \mathbf{1}_{T^*\Sigma\setminus \mathcal{O}^{(0)}}(z')$, for all $z,z' \in T^*\Sigma$.
  Since the support of $\rho_0$ is contained in $\mathcal{O}^{(0)}$, we have
  \begin{equation}
  \forall~z,z' \in T^*\Sigma, \qquad
  f^\epsilon(z,z') \rho_0(z)= f^\epsilon(z,z') \rho_0(z') = 0\,.
    \label{feps-times-rho0-is-zero}
  \end{equation}
  For $f^\epsilon$, we compute the left hand side of \eqref{change-of-variables-for-phi} as
    \begin{align}
      & \int_{T^*\Sigma}\left[\sum_{z'\in \Phi(z)}
      \mathbf{1}_{\widetilde{\mathcal{D}}_\Phi}\big(z, z'\big)
      f^\epsilon\left(z, z'\right) \right] \,\sigma_{T^*\Sigma}(dz)\notag \\
    & = \sum_{i=0}^n\sum_{i'=0}^n\int_{T^*\Sigma}\left[\sum_{
    z'\in \Phi(z),\Pi(z')\not\in C_{M,\Pi(z)}}
    \mathbf{1}_{\widetilde{\mathcal{D}}_\Phi}\big(z,z'\big)
    \rho_{i}(z)\rho_{i'}(z') f^\epsilon\left(z, z'\right) \right] \,\sigma_{T^*\Sigma}(dz)\notag \\
      & = \sum_{i=1}^n\sum_{i'=1}^n\int_{\mathcal{O}^{(i)}}\left[\sum_{
    z'\in \Phi(z),\Pi(z')\not\in C_{M,\Pi(z)}}
    \mathbf{1}_{\widetilde{\mathcal{D}}_\Phi}\big(z,z'\big)
    \rho_{i}(z)\rho_{i'}(z') f^\epsilon\left(z, z'\right) \right] \,\sigma_{T^*\Sigma}(dz)\notag \\
    & =
    \sum_{i=1}^n\sum_{i'=1}^n\sum_{j=1}^{m_{i}}\int_{\mathcal{O}^{(i)}}\left[
      \mathbf{1}_{\widetilde{\mathcal{D}}_\Phi}\big(z,\Upsilon^{(i,j)}(z)\big)
    \rho_{i}(z)\rho_{i'}(\Upsilon^{(i,j)}(z)) f^\epsilon\left(z,
 \Upsilon^{(i,j)}(z)\right)\right]
    \,\sigma_{T^*\Sigma}(dz) \label{int-feps-lemma-for-phase-space-1} \\
    & =
    \sum_{i=1}^n\sum_{i'=1}^n\sum_{j=1}^{m_{i}}\int_{\Upsilon^{(i,j)}(\mathcal{O}^{(i)})}\left[\mathbf{1}_{\widetilde{\mathcal{D}}_\Phi}\big((\Upsilon^{(i,j)})^{-1}(z),z\big)
    \rho_{i}((\Upsilon^{(i,j)})^{-1}(z))\rho_{i'}(z)
      f^\epsilon\left((\Upsilon^{(i,j)})^{-1}(z), z\right) \right] \,\sigma_{T^*\Sigma}(dz)\notag \\
    & =
    \sum_{i=1}^n\sum_{i'=1}^n\sum_{j=1}^{m_{i}}\int_{\Upsilon^{(i,j)}(\mathcal{O}^{(i)})\cap
      \mathcal{O}^{(i')}}\left[\mathbf{1}_{\widetilde{\mathcal{D}}_\Phi}\big((\Upsilon^{(i,j)})^{-1}(z),z\big)
    \rho_{i}((\Upsilon^{(i,j)})^{-1}(z))\rho_{i'}(z)
      f^\epsilon\left((\Upsilon^{(i,j)})^{-1}(z), z\right) \right]
      \,\sigma_{T^*\Sigma}(dz)\notag \,.
  \end{align}
  In the first equality of \eqref{int-feps-lemma-for-phase-space-1}, we have used
  $\sum_{i=0}^n\rho_i(z)=1$ for all $z \in T^*\Sigma$ and that  
$\mathbf{1}_{\widetilde{\mathcal{D}}_\Phi}(z,z')=0$ when $z'\in \Phi(z)$ and $\Pi(z')\in C_{M,\Pi(z)}$;
   in the second one, we have used~\eqref{feps-times-rho0-is-zero} 
  and that the support of~$\rho_i$ is contained in $\mathcal{O}^{(i)}$; in the third equality
  of~\eqref{int-feps-lemma-for-phase-space-1}, we have used 
  \eqref{good-states-by-Upsilon} and the fact that $m(z)=m_i$ is constant on
  $\mathcal{O}^{(i)}$ for $i=1,\ldots, n$; in the fourth equality, we have
  used the fact that $\Upsilon^{(i,j)}$ is a $C^1$-diffeomorphism
  on~$\mathcal{O}^{(i)}$ and the change of variables $z\leftarrow (\Upsilon^{(i,j)})^{-1}(z)$, together with the fact that 
  $|\det(D\Upsilon^{(i,j)})|\equiv 1$ (see the third item of
  Proposition~\ref{prop-differentiability-hmc}); the last equality of~\eqref{int-feps-lemma-for-phase-space-1} follows since
the support of $\rho_{i'}$ is contained in $\mathcal{O}^{(i')}$.

Now we want to relate the preimages of $z$ in the last line of 
\eqref{int-feps-lemma-for-phase-space-1} to the images of $z$.
Fix the indices $i,i', j$ in the last line of \eqref{int-feps-lemma-for-phase-space-1}.
For any $z\in \Upsilon^{(i,j)}(\mathcal{O}^{(i)})\cap \mathcal{O}^{(i')}$ such that
$\mathbf{1}_{\widetilde{\mathcal{D}}_\Phi}\big(z',z\big) \rho_{i}(z') \neq 0$, 
where $z'=(\Upsilon^{(i,j)})^{-1}(z)$, we have $z' \in \mathcal{O}^{(i)}$ since the support of
$\rho_i$ is in $\mathcal{O}^{(i)}$. The definition~\eqref{admissible-set-d-r-subset} of 
$\widetilde{\mathcal{D}}_\Phi$ implies that $z' \in \Phi(z)$ and
$\Pi(z')\not\in C_{M,\Pi(z)}$, therefore,
using \eqref{good-states-by-Upsilon} there is an index $j' \in \{1,\dots,m_{i'}\}$ such that $z'=\Upsilon^{(i',j')}(z)$. For any other index $j'_1 \in \{1,\dots,m_{i'}\}$ with $j'_1\neq j'$, 
we claim that $\mathbf{1}_{\widetilde{\mathcal{D}}_\Phi}\big(\widetilde{z},z\big)
\rho_{i}(\widetilde{z})=0$ with $\widetilde{z} = \Upsilon^{(i',j'_1)}(z)\neq
z'$. Indeed, suppose that $\mathbf{1}_{\widetilde{\mathcal{D}}_\Phi}\big(\widetilde{z},z\big)
\rho_{i}(\widetilde{z})\neq 0$, then $\widetilde{z}\in \mathcal{O}^{(i)}$ and,
by definition of~$\widetilde{\mathcal{D}}_\Phi$ and using \eqref{good-states-by-Upsilon}, there is an index $j_1 \in \{1,\dots,m_i\}$
such that $z=\Upsilon^{(i, j_1)}(\widetilde{z})$. When $j_1=j$, 
$\Upsilon^{(i,j)}$ maps both~$z'$ and $\widetilde{z}$ to $z$, which is in
contradiction with the fact that $\Upsilon^{(i,j)}$ is a $C^1$-diffeomorphism on~$\mathcal{O}^{(i)}$.
When $j_1\neq j$, we would have $z \in \Upsilon^{(i,j)}(\mathcal{O}^{(i)}) \cap
\Upsilon^{(i,j_1)}(\mathcal{O}^{(i)})\neq \emptyset$, which is in
contradiction with our definition of the neighborhood $\mathcal{O}^{(i)}$. Therefore, 
for given $z\in \Upsilon^{(i,j)}(\mathcal{O}^{(i)})\cap \mathcal{O}^{(i')}$, 
there exists one and only one index $j' \in \{1,\dots, m_{i'}\}$ such that 
$\mathbf{1}_{\widetilde{\mathcal{D}}_\Phi}\big(\Upsilon^{(i',j')}(z),z\big)
\rho_{i}(\Upsilon^{(i',j')}(z))\neq 0$.
Moreover, this index~$j'$ satisfies that $\Upsilon^{(i',j')}(z) =
(\Upsilon^{(i,j)})^{-1}(z)$.
Using this fact, we can replace the preimages in the last line of \eqref{int-feps-lemma-for-phase-space-1}
by including a summation over the index~$j'$ among $1,\ldots, m_{i'}$ and obtain
  \begin{align*}
      & \int_{T^*\Sigma}\left[\sum_{z'\in \Phi(z)}
      \mathbf{1}_{\widetilde{\mathcal{D}}_\Phi}\big(z, z'\big)
      f^\epsilon\left(z, z'\right) \right] \,\sigma_{T^*\Sigma}(dz)\notag \\
    & =
    \sum_{i=1}^n\sum_{i'=1}^n\sum_{j=1}^{m_{i}} \sum_{j'=1}^{m_{i'}}
    \int_{\Upsilon^{(i,j)}(\mathcal{O}^{(i)})\cap
      \mathcal{O}^{(i')}}\left[\mathbf{1}_{\widetilde{\mathcal{D}}_\Phi}\big(\Upsilon^{(i',j')}(z),z\big)
    \rho_{i}\left(\Upsilon^{(i',j')}(z)\right)\rho_{i'}(z)
      f^\epsilon\left(\Upsilon^{(i',j')}(z), z\right) \right] \,\sigma_{T^*\Sigma}(dz)\notag \\
    & =
    \sum_{i=1}^n\sum_{i'=1}^n \sum_{j'=1}^{m_{i'}}
    \int_{\cup_{j=1}^{m_i}\big(\Upsilon^{(i,j)}(\mathcal{O}^{(i)})\cap
      \mathcal{O}^{(i')}\big)}\left[\mathbf{1}_{\widetilde{\mathcal{D}}_\Phi}\big(\Upsilon^{(i',j')}(z),z\big)
    \rho_{i}\left(\Upsilon^{(i',j')}(z)\right)\rho_{i'}(z)
      f^\epsilon\left(\Upsilon^{(i',j')}(z), z\right) \right] \,\sigma_{T^*\Sigma}(dz)\notag 
  \end{align*}
  where in the second equality we have used the fact
  that~$\Upsilon^{(i,j)}(\mathcal{O}^{(i)})\cap \mathcal{O}^{(i')}$ are
  disjoint subsets of $\mathcal{O}^{(i')}$ for different indices~$j$, thanks to the definitions of the neighborhoods $\mathcal{O}^{(i)}$.  
Moreover, for any $z\in \mathcal{O}^{(i')} \setminus \big(\cup_{j=1}^{m_i}
\Upsilon^{(i,j)}(\mathcal{O}^{(i)}) \big)$, the function
$\mathbf{1}_{\widetilde{\mathcal{D}}_\Phi}\big(\Upsilon^{(i',j')}(z),z\big)
\rho_{i}(\Upsilon^{(i',j')}(z))$ must be zero for all~$j'$, since there is no function
$\Upsilon^{(i,j)}$ that maps states in $\mathcal{O}^{(i)}$ to $z$.
Therefore, we can enlarge the integration domain to $\mathcal{O}^{(i')}$ and obtain
  \begin{align}
      & \int_{T^*\Sigma}\left[\sum_{z'\in \Phi(z)}
      \mathbf{1}_{\widetilde{\mathcal{D}}_\Phi}\big(z, z'\big)
      f^\epsilon\left(z, z'\right) \right] \,\sigma_{T^*\Sigma}(dz)\notag \\
    & = \sum_{i=1}^n\sum_{i'=1}^n \sum_{j'=1}^{m_{i'}}
    \int_{\mathcal{O}^{(i')}}\left[\mathbf{1}_{\widetilde{\mathcal{D}}_\Phi}\big(\Upsilon^{(i',j')}(z),z\big)
    \rho_{i}\left(\Upsilon^{(i',j')}(z)\right)\rho_{i'}(z)
      f^\epsilon\left(\Upsilon^{(i',j')}(z), z\right) \right]
      \,\sigma_{T^*\Sigma}(dz)\notag \\
    & = \sum_{i=0}^n\sum_{i'=1}^n 
    \int_{\mathcal{O}^{(i')}}\sum_{z'\in \Phi(z), \Pi(z')\not\in
    C_{M,\Pi(z)}}\left[\mathbf{1}_{\widetilde{\mathcal{D}}_\Phi}\big(z',z\big)
    \rho_{i}(z')\rho_{i'}(z)
      f^\epsilon\left(z', z\right) \right]
      \,\sigma_{T^*\Sigma}(dz) \label{change-of-variable-phi-feps}\\
    & = \sum_{i'=1}^n 
    \int_{\mathcal{O}^{(i')}}\sum_{z'\in \Phi(z), \Pi(z')\not\in
    C_{M,\Pi(z)}}\left[\mathbf{1}_{\widetilde{\mathcal{D}}_\Phi}\big(z',z\big)
    \rho_{i'}(z) f^\epsilon\left(z', z\right) \right]
      \,\sigma_{T^*\Sigma}(dz) \notag\\
    & = \sum_{i'=0}^n \int_{T^*\Sigma}\sum_{z'\in \Phi(z)}\left[\mathbf{1}_{\widetilde{\mathcal{D}}_\Phi}\big(z',z\big)
    \rho_{i'}(z) f^\epsilon\left(z', z\right) \right]
      \,\sigma_{T^*\Sigma}(dz) \notag\\
    & = \int_{T^*\Sigma}\sum_{z'\in
    \Phi(z)}\left[\mathbf{1}_{\widetilde{\mathcal{D}}_\Phi}\big(z,z'\big)
     f^\epsilon\left(z', z\right) \right] \,\sigma_{T^*\Sigma}(dz)\,, \notag
\end{align}
where in the second equality we have used \eqref{good-states-by-Upsilon} and \eqref{feps-times-rho0-is-zero};
in the third equality we have used $\sum_{i=0}^n\rho_i(z)=1$ for all $z$; in
the fourth equality we have used the fact that the support of $\rho_{i'}$ is contained
in $\mathcal{O}^{(i')}$, the fact that $\mathbf{1}_{\widetilde{\mathcal{D}}_\Phi}(z',z)=0$ when $z'\in \Phi(z)$ and $\Pi(z')\in C_{M,\Pi(z)}$, as well as \eqref{feps-times-rho0-is-zero}; in the last equality we have used again
$\sum_{i'=0}^n\rho_{i'}(z)=1$ for all $z$, as well as the symmetry
$\mathbf{1}_{\widetilde{\mathcal{D}}_\Phi}(z',z)
=\mathbf{1}_{\widetilde{\mathcal{D}}_\Phi}(z,z')$. Therefore, we have proved
\eqref{change-of-variables-for-phi} for the function $f^\epsilon$.

Finally, we consider the limit $\epsilon \rightarrow 0$.
Note that, for any $z,z'\in T^*\Sigma$, the equality $\mathbf{1}_{\widetilde{\mathcal{D}}_\Phi}(z,z') =1$
implies that $z,z'\in \cup_{l\ge 1}\widetilde{\mathcal{B}}_l=T^*\Sigma\setminus \widetilde{\mathcal{B}}_0$.
In other words, we have 
\begin{equation}
\forall~ z,z'\in T^* \Sigma, \qquad \mathbf{1}_{\widetilde{\mathcal{D}}_\Phi}(z,z') =
  \mathbf{1}_{\widetilde{\mathcal{D}}_\Phi}(z,z')
  \mathbf{1}_{T^*\Sigma\setminus \widetilde{\mathcal{B}}_0}(z)
  \mathbf{1}_{T^*\Sigma\setminus \widetilde{\mathcal{B}}_0}(z')\,.
  \label{an-identity-for-indicator-function}
\end{equation}
    Therefore, using \eqref{properties-of-sets-b-eps-hmc} and the definition of $f^\epsilon$, we obtain, for any~$z \in T^*\Sigma$,
 \[
   \begin{aligned}
     \lim_{\epsilon \rightarrow 0}&\sum_{z'\in \Phi(z)} \mathbf{1}_{\widetilde{\mathcal{D}}_\Phi}(z,z') f^\epsilon(z, z')\\
    & = 
     \sum_{z'\in \Phi(z)}\left[ \mathbf{1}_{\widetilde{\mathcal{D}}_\Phi}(z,z') f(z, z') \lim_{\epsilon \rightarrow 0} \left(\mathbf{1}_{T^*\Sigma\setminus
\widetilde{\mathcal{B}}^\epsilon_0}(z) \mathbf{1}_{T^*\Sigma\setminus
     \widetilde{\mathcal{B}}^\epsilon_0}(z')\right)\right]\\
     &=
     \sum_{z'\in \Phi(z)}\left[ \mathbf{1}_{\widetilde{\mathcal{D}}_\Phi}(z,z') f(z, z')
      (1- \mathbf{1}_{\widetilde{\mathcal{B}}_0}(z))
      (1-\mathbf{1}_{\widetilde{\mathcal{B}}_0}(z'))\right]\\
     &= \sum_{z'\in \Phi(z)} \left[\mathbf{1}_{\widetilde{\mathcal{D}}_\Phi}(z,z')  \mathbf{1}_{T^*\Sigma\setminus \widetilde{\mathcal{B}}_0}(z)
  \mathbf{1}_{T^*\Sigma\setminus \widetilde{\mathcal{B}}_0}(z') f(z,
     z')\right]\\
     &= \sum_{z'\in\Phi(z)} \mathbf{1}_{\widetilde{\mathcal{D}}_\Phi}(z,z') f(z, z')\,,
   \end{aligned}
  \]
  where we have used \eqref{an-identity-for-indicator-function} to obtain the last equality.
  Moreover, from the same calculation, we also have the identity $\lim_{\epsilon
  \rightarrow 0}\sum_{z'\in \Phi(z)}
  \mathbf{1}_{\widetilde{\mathcal{D}}_\Phi}(z,z') f^\epsilon(z', z)
  = \sum_{z'\in\Phi(z)} \mathbf{1}_{\widetilde{\mathcal{D}}_\Phi}(z,z') f(z', z)$, for all $z \in T^*\Sigma$ .
 We thus obtain \eqref{change-of-variables-for-phi} by taking the limit~$\epsilon \rightarrow 0$ 
    in \eqref{change-of-variables-for-phi} for the functions $f^\epsilon$ and applying the dominated convergence theorem. 
\end{proof}

\section*{Acknowledgements}

This study was inspired by the work of Paul Breiding and Orlando
Marigliano~\cite{breidling-random-point}. The authors thank them for
fruitful discussions on the computations of multiple projections on
submanifolds. WZ is funded by the Deutsche Forschungsgemeinschaft (DFG, German
Research Foundation) under Germany's Excellence Strategy --- The Berlin
Mathematics Research Center MATH+ (EXC-2046/1, project ID: 390685689). TL and
GS have received funding from the European Research Council (ERC) under the
European Union's Horizon 2020 research and innovation programme (grant
agreement No 810367). TL and GS also benefited from the scientific environment
of the Laboratoire International Associ\'e between the Centre National de la
Recherche Scientifique and the University of Illinois at Urbana-Champaign.
Finally, part of this work was done while TL was visiting the CNRS-ICL International Research Laboratory. TL would like to thank the Leverhulme Trust for funding as well as the Department of Mathematics at Imperial College of London for its hospitality.


\end{document}